\documentclass[11pt,leqno]{amsart}
\usepackage{amssymb}
\usepackage{graphicx}
\usepackage{makecell}
\numberwithin{equation}{section}
\usepackage{fullpage}

\usepackage{hyperref}
\allowdisplaybreaks[2]

\renewcommand\d{\partial}

\def\eps{\varepsilon }

\def\bdiag{\textrm{block-diag}}
\def\diag{\textrm{diag}}

\renewcommand\d{\partial}

\newcommand\R{\mathbb R}

\def\eps{\varepsilon}

\newcommand\br{\begin{remark}}
\newcommand\er{\end{remark}}
\newcommand\bp{\begin{pmatrix}}
\newcommand\ep{\end{pmatrix}}
\newcommand{\be}{\begin{equation}}
\newcommand{\ee}{\end{equation}}
\newcommand\ba{\begin{equation}\begin{aligned}}
\newcommand\ea{\end{aligned}\end{equation}}

\newcommand{\bap}{\begin{app}}
\newcommand{\eap}{\end{app}}
\newcommand{\begs}{\begin{exams}}
\newcommand{\eegs}{\end{exams}}
\newcommand{\beg}{\begin{example}}
\newcommand{\eeg}{\end{example}}
\newcommand{\bpr}{\begin{proposition}}
\newcommand{\epr}{\end{proposition}}
\newcommand{\bt}{\begin{theorem}}
\newcommand{\et}{\end{theorem}}
\newcommand{\bc}{\begin{corollary}}
\newcommand{\ec}{\end{corollary}}
\newcommand{\bl}{\begin{lemma}}
\newcommand{\el}{\end{lemma}}
\newcommand{\bd}{\begin{definition}}
\newcommand{\ed}{\end{definition}}
\newcommand{\brs}{\begin{remarks}}
\newcommand{\ers}{\end{remarks}}

\newcommand{\U }{\mathcal{U}}

\newcommand{\N}{\mathcal{N}}

\newcommand{\const}{\text{\rm constant}}
\newcommand{\Id}{{\rm Id }}

\newcommand{\blockdiag}{{\rm blockdiag }}

\DeclareMathOperator{\sgn}{sgn}
\newtheorem{theorem}{Theorem}[section]
\newtheorem{proposition}[theorem]{Proposition}
\newtheorem{corollary}[theorem]{Corollary}
\newtheorem{lemma}[theorem]{Lemma}

\theoremstyle{remark}
\newtheorem{remark}[theorem]{Remark}
\theoremstyle{definition}
\newtheorem{definition}[theorem]{Definition}

\newtheorem{example}[theorem]{Example}

\newtheorem{obs}[theorem]{Observation}
\newcommand{\f}{\frac}

\newcommand{\beq}{\begin{equation}}
\newcommand{\eeq}{\end{equation}}
\title{Stability of hydraulic shock profiles}
\author{Zhao Yang}
\address{Indiana University, Bloomington, IN 47405}
\email{yangzha@indiana.edu}
\thanks{Research of Z.Y. was partially supported by the Hazel King Thompson Summer Reading Fellowship,
a Mathematics Department Research Assistantship, and the College of Arts and Sciences Dissertation Year Fellowship}

\author{Kevin Zumbrun}
\address{Indiana University, Bloomington, IN 47405}
\email{kzumbrun@indiana.edu}
\thanks{Research of K.Z. was partially supported
under NSF grants no. DMS-1400555 and DMS-1700279}

\begin{document}

\begin{abstract}
We establish nonlinear $H^2\cap L^1 \to H^2$ stability with sharp rates of decay in $L^p$, $p\geq 2$,
of general hydraulic shock profiles, with or without subshocks, of the inviscid Saint-Venant equations of shallow water flow,
under the assumption of Evans-Lopatinsky stability of the associated eigenvalue problem.
We verify this assumption numerically for all profiles, giving in particular the first nonlinear
stability results for shock profiles with subshocks of a hyperbolic relaxation system. 
\end{abstract}

\date{\today}
\maketitle


%TODO: review these

\tableofcontents

\section{Introduction}\label{s:intro}
In this paper, by a combination of rigorous analysis and numerical verification, 
we establish nonlinear stability of nondegenerate hydraulic shock profiles of the 
inviscid Saint-Venant equations for inclined shallow water flow, across their entire domain of existence, 
in particular including large-amplitude profiles containing subshock discontinuities.
Specifically, assuming spectral stability in the sense of Majda-Erpenbeck \cite{Ma,Er1,Er2,HuZ,Z1},
we prove linear and nonlinear $H^2\cap L^1\to H^2$ 
phase asymptotic orbital stability, 
with sharp rates of decay in $L^p$, $p\geq 2$. 
We then verify the spectral stability condition numerically, by exhaustive Evans-Lopatinsky/Evans function computations. 

The inviscid Saint-Venant equations
\ba \label{sv}
\d_th+\d_xq&=0,\\
\d_tq+\d_{x}\left(\frac{q^2}{h}+\frac{h^2}{2F^2}\right)&=h-\frac{|q|q}{h^2},
\ea
here given in nondimensional form, model inclined shallow water flow,
where $h$ is fluid height; $q=hu$ is total flow, with $u$ fluid velocity;
and $F>0$ is the {\it Froude number}, a nondimensional parameter depending on reference height/velocity and inclination.
Among other applications, they are commonly used in the hydraulic engineering literature to describe flow in a dam spillway, channel, or etc.; see, e.g.,  \cite{BM,Je,Br1,Br2,Dr,JNRYZ} for further discussion.

Equations \eqref{sv} form a {\it hyperbolic system of balance laws} \cite{La,Bre, Da}, 
with the first equation representing conservation of fluid and the second balance between change of momentum 
and the opposing forces of gravity ($h$)
and turbulent bottom friction ($-h^{-2}|q|q$).  
More specifically, they compose a $2\times 2$ {\it relaxation system} \cite{W,L1,Bre,Da}, 
with associated formal equilibrium equation
\be\label{CE}
\d_t h + \d_x  q_*(h)=0,
\ee
where $q_*(h):=h^{3/2}$ is the value of $q$ for which gravity and bottom forces cancel.
That is, near-equilibrium behavior is formally modeled by a scalar conservation law, or {\it generalized}
(inviscid) {\it Burgers equation}.
On the other hand, short-time, or transient, 
behavior is formally modeled by the first-order part of \eqref{sv}, with zero-order forcing term $ h-h^{-2}q^2$ ($q>0$)
set to zero; for later reference, we note that this coincides with the 
{\it equations of isentropic $\gamma$-law gas dynamics} with $\gamma =2$ \cite{Bre,Da,Sm}.

As discussed, e.g., in \cite{W,L1,JK}, the formal approximation \eqref{CE} is valid for general $2\times 2$
relaxation systems in the vicinity of an equilibrium point 
$(h,q)=(h_0,q_*(h_0)) $ provided there holds the {\it subcharacteristic condition} that the characteristic speed
$q_*'(h_0)$ of \eqref{CE} lies between the characteristic speeds of \eqref{sv}.
This is also the condition for {\it hydrodynamic stability}, or stability under perturbation of a constant equilibrium
flow $(h,q)(x,t)\equiv (h_0, q_*(h_0))$: for the Saint-Venant equations, the classical Froude number
condition of Jeffreys \cite{Je},
\be\label{hydrostab}
F<2.
\ee
In this regime, one may expect persistent asymptotically-constant traveling wave solutions 
\be\label{prof}
(h,q)(x,t)= (H,Q)(x-ct), \quad \lim_{z\to - \infty}(H,Q)(z)= (H_L,Q_L), \; \lim_{z\to +\infty}(H,Q)(z)= (H_R,Q_R), 
\ee
analogous to shock waves of \eqref{CE}, known as {\it relaxation shocks}, or relaxation profiles;
in the context of \eqref{sv}, we shall call these {\it hydraulic shock profiles}.
In the complementary regime $F>2$, one expects, rather, complex behavior and pattern formation \cite{JK,JNRYZ,BJNRZ}.

Indeed, we have the following description of existence (Section \ref{s:profiles}). 
Here and elsewhere, let $[h]=h(x^+)-h(x^-)$ of a quantity $h$ across a discontinuity located at $x$.

\bpr\label{existprop}
Let $(H_L,H_R, c)$ be a triple for which there exists an entropy-admissible shock solution in the sense of Lax \cite{La}
with speed $c$ of \eqref{CE} connecting left state $H_L$ to right state $H_R$, i.e., $H_L>H_R>0$ and
$c[H]=[q_*(H)]$.
Then, there exists a corresponding hydraulic shock profile \eqref{prof} with $Q_L=q_*(H_L)$ and $Q_R=q_*(H_R)$
precisely if $0<F<2$.
The profile is smooth for $H_L> H_R> H_L \frac{2F^2}{1+2F+\sqrt{1+4F}}$, and nondegenerate in the sense that $c$ is not a
characteristic speed of \eqref{sv} at any point along the profile.
For $0<H_R< H_L \frac{2F^2}{1+2F+\sqrt{1+4F}}$, the profile is nondegenerate and piecewise smooth, with a single
discontinuity consisting of an entropy-admissible shock of \eqref{sv}.
At the critical value $H_R= H_L \frac{2F^2}{1+2F+\sqrt{1+4F}}$, $H_R$ is characteristic,
and there exists a degenerate profile that is continuous but not smooth, with discontinuous derivative at $H_R$.
For $F>2$, there exist smooth ``reverse shock'' profiles connecting the endstates in the opposite
direction $H_R\to H_L$, precisely when $H_R<H_L<H_R\frac{1+2F -\sqrt{1+4F}}{2}$. 
In the degenerate case $H_L=H_R\frac{1+2F -\sqrt{1+4F}}{2}$, $H_L$ is characteristic and
there exists an uncountable family of degenerate entropy-admissible piecewise smooth homoclinic profiles connecting $H_R$ to itself, but no smooth profiles. 
In all cases, these are the only entropy-admissible piecewise smooth, asymptotically-constant traveling waves of \eqref{sv},
and $c, Q>0$.
\epr

This corresponds to the picture for general relaxation systems \cite{W,L1,YoZ,MZ}, wherein smooth 
relaxation profiles are known to exist for small-amplitude equilibrium shocks near equilibrium points that are stable as constant solutions, but larger-amplitude profiles contain discontinuities, or ``subshocks'', if they exist at all.
Meanwhile, profiles initiating from an unstable equilibrium typically connect endstates in a reverse direction corresponding
to a non-entropy admissible shock of \eqref{CE} \cite{YoZ} (and in any case cannot be stable as solutions of the associated
relaxation system \cite{MZ,MZ2}).
Accordingly, we focus hereafter on the case $0<F<2$ for which hydraulic shock profiles exist in the proper direction,
and examine the stability of such profiles as solutions of \eqref{sv}.

\subsection{Main results}\label{s:mainresults}
We first recall that system \eqref{sv} is of classical {\it Kawashima class}, meaning that it is of symmetrizable hyperbolic type, with a symmetrizer that simultaneously symmetrizes the linearized zero-order relaxation (or ``balance'') term; see Observation \ref{symmobs}.
By the analytical results of \cite{MZ,MZ3}, therefore, we obtain immediately spectral, linearized, and nonlinear 
stability and asymptotic orbital 
stability with sharp rates of smooth hydraulic shock profiles of sufficiently small amplitude, for any fixed endstate
$H_L$.
Moreover, by \cite{MZ2}, we obtain the same linearized and nonlinear stability results for smooth profiles of 
{\it arbitrary amplitude}, provided they are spectrally stable in the sense of a standard Evans function condition,
and nondegenerate in the sense that hyperbolic characteristics do not coincide along the profile with the speed of the wave.
Hence, {\it the smooth nondegenerate case may be treated by existing analysis}, reducing to a 
standard numerical Evans function study of intermediate-amplitude waves, as carried out for example in 
\cite{BHRZ,BHZ,BLeZ,BLZ,HLyZ1}.

We focus here on the complementary large-amplitude case of {\it nondegenerate shock profiles containing subshocks},
or $0<H_R< H_L \frac{2F^2}{1+2F+\sqrt{1+4F}}$.
The degenerate case $H_R= H_L \frac{2F^2}{1+2F+\sqrt{1+4F}}$ we do not treat.
For perturbations satisfying appropriate compatibility conditions at the shock, in particular for perturbations supported
away from the shock, short-time $H^s$ existence follows by the analysis of Majda \cite{Ma,Me}, as noted in \cite{JLW}.
However, so far as we know, there were no results up to now on large-time behavior or existence under perturbation
of relaxation profiles containing subshocks.
Our main result is the following theorem establishing global existence and nonlinear 
phase-asymptotic orbital stability 
in this case, with sharp rates of decay, assuming spectral stability in the sense of an Evans-Lopatinsky condition 
analogous to that of the smooth profile case.

\bt\label{main}
For $0<F<2$ and $0<H_R< H_L \frac{2F^2}{1+2F+\sqrt{1+4F}}$, let $\overline{W}=(H,Q)$ be a hydraulic shock profile \eqref{prof},
and $v_0$ be an initial perturbation supported away from the subshock discontinuity of $\overline{W}$ of norm $\eps$ sufficiently small
%TODO: typo in Metivier \cite{Me}???? OR?  %in $H^{s+1/2}\cap L^1$, $s\geq 2$. 
in $H^{s}\cap L^1$, $s\geq 2$. 
Moreover, assume that $\overline{W}$ is \emph {spectrally stable} in the sense of the Evans-Lopatinsky condition
defined in Section \ref{s:lop}. 
Then, for initial data $\tilde W_0:=\overline{W}_0+v_0$, there exists a global solution of \eqref{sv}, 
with a single shock located at $ct- \eta(t)$, and $H^s$ to either side of the shock, satisfying
for $t\geq 0$, $2\leq p\leq \infty$, and some limiting phase $\eta_\infty$: 
\ba\label{mainests}
|\tilde W(\cdot,t)-\overline{W}(\cdot-ct +\eta(t))|_{H^s}&\leq C\eps (1+t)^{-1/4},\\
|\tilde W(\cdot,t)-\overline{W}(\cdot-ct +\eta(t))|_{L^p}&\leq C\eps (1+t)^{-(1/2)(1-1/p)}, \\
|\dot \eta(t)|&\leq C\eps (1+t)^{-(1/2)},\\
|\eta(t)|&\leq C\eps ,\\
|\eta(t)-\eta_\infty|&\leq C\eps(1+t)^{-1/4+\upsilon} + C|v_0|_{L^1(|x|\geq t/C)}
\ea
for any $\upsilon>0$, and some $C=C(\upsilon)>0$. In particular, $\eta(t)\to \eta_\infty$ as $t\to +\infty$.
\et

%%%%%%%%%%%%%%%%%%%%%%%%%%%%%%%%%
Estimates \eqref{mainests}(i)-(iv) may be recognized as exactly the same as those given for smooth profiles in 
\cite[Thm. 1.2]{MZ2}, but with $\eta$ now an exact shock location forced by the presence of a 
discontinuity rather than an approximate location designed to optimize errors as in the smooth case.
Estimate \eqref{mainests}(v), upgrading asymptotic orbital stability to phase-asymptotic orbital stability, 
is new even in the smooth case.
We complement these results by systematic numerical studies verifying the Evans-Lopatinsky condition for 
nondegenerate hydraulic shock profiles containing subshocks, and the Evans condition for nondegenerate smooth profiles,
across their full domain of existence.  
%%%%%%%%%%%%%%%%%%%%%%%%%%%%%%%%
Together with our analytical results, 
this yields both linearized and nonlinear phase-asymptotic orbital stability 
of (all) nondegenerate hydraulic shock profiles of \eqref{sv}, 
that is, {\it asymptotic convergence under perturbation to a nearby translate of the original wave.}
%the strongest possible (due to translation invariance) notion of stability for a traveling wave \cite{Sa,He}.
Note, due to translation invariance, that this is the strongest possible notion of stability for a traveling wave
\cite{Sa,He,L2,ZH}.
%%%%%%%%%%%%%%%%%%%%%%%%%%%%%%%%%

\br\label{noratermk}
As noted in \cite{MZ2} for the smooth case, the rates \eqref{mainests} are sharp,
In particular, as noted in \cite{MZ,MZ2}, under the
very weak localization $v_0\in L^1\cap H^s$ assumed on the initial perturbation, it is not possible
to give a rate for the convergence $\eta(t)\to \eta_\infty$, even at the linearized level.
For, by translating the initial perturbation farther and farther toward infinity, an operation
that does not change its norm, we may by finite propagation speed of the underlying hyperbolic model,
delay indefinitely the interaction of the perturbation with the component subshock of the traveling wave.
However, conservation of mass principles \cite{L1,L2}, applied to the linearized problem,
imply that, to linear order in perturbation norm $\eps$ 
the asymptotic shock location depends only on the ``total perturbation mass''
$\int_{-\infty}^{+\infty}h_0(y)\, dy$, hence is independent of translation.
These two facts together are inconsistent with convergence at a fixed rate depending only on $\eps=|v_0|_{L^1\cap H^s}$.
\er

\subsection{Discussion and open problems}\label{s:discussion}
Large-amplitude hydraulic shock profiles are physically interesting from the point of view of
dam break or river bore phenomena.
Our results bear on the question whether the Saint-Venant equations \eqref{sv} typically used in hydraulic engineering
can model such phenomena.
An interesting question for further investigation is whether the modeling of additional physical effects such as viscosity or capillarity become important at large amplitudes, radically changing behavior, or whether the solutions studied here indeed accurately capture behavior even in the discontinuous regime.
We mention also the recent introduction in \cite{RG1,RG2} of vorticity to model \eqref{sv}, yielding effectively a 
$3\times 3$ relaxation model with scalar equilibrium system.
In the unstable, pattern formation regime analogous to $F>2$ for \eqref{sv}, 
this augmented model is seen to give much closer correspondence in wave form for periodic roll wave patterns to that 
seen in experiment in  \cite{Br1,Br2}.
A very interesting open problem would be to study existence and stability of hydraulic shocks for this more 
complicated model, in particular comparing results to Saint-Venant profiles and experiment,

On the mathematical side, our main contribution here is the treatment for the first time of nonlinear stability of 
relaxation profiles containing subshocks, a topic that so far as we know has up to now not been addressed.
(Though see \cite{DR1,DR2} for related, contemporary, studies of stability of discontinuous solutions of scalar balance laws.)
Indeed, at the outset it is perhaps not clear what is the proper framework in which this problem
should be approached, as smooth and discontinuous shocks have been treated in the literature 
by quite different and at first sight incompatible techniques.  
However, a useful bridge between these two (continuous and discontinuous) domains comes from the study of 
smooth boundary layer solutions of initial boundary value problems in \cite{YZ,NZ} and the treatment of piecewise 
smooth detonation waves in \cite{JLW}, in particular the suggestive use of the ``good unknown'' to separate interior 
and boundary problems in a convenient way.

Combining these two approaches allows us to formulate the linearized problem by an inverse Laplace transform
representation similar to that appearing for smooth profiles in \cite{ZH,MZ,YZ,NZ}, and thereby to obtain detailed
pointwise Green function bounds by analogous (stationary phase, or Riemann saddlepoint) techniques.
This allows us as in the smooth profile case to set up a nonlinear iteration based on contraction mapping, for which
the nonlinear source loses one derivative.
The nonlinear argument is then closed by an energy-based ``nonlinear damping'' estimate on the half-line
modifying the corresponding large-amplitude estimate of \cite{MZ2}) on the whole line, 
which controls higher Sobolev norms in terms of $L^2$ and an exponentially decaying multiple of the initial high norm, 
thus closing the iteration.

A key new ingredient in the half-line argument is the observation that the hyperbolic Friedrichs 
symmetrizer $\tilde{A}_\alpha^0$ used in the symmetric hyperbolic part of the energy estimates may be chosen so that the 
boundary conditions become maximally dissipative, a special feature of the one-dimensional case.
A second new ingredient is the use of ``Strichartz-type'' bounds (Lemma \ref{Glemaux}) to control new trace
terms arising in phase bounds for the nonsmooth case; the resulting ``vertical estimate'' \eqref{vert}
controlling time integrals at fixed spatial location seems of interest in its own right.
A further novelty in the analysis
is the introduction of a new ``approximate characteristic'' argument by which we can roughly decompose tail
from center contributions of the initial perturbation,
to obtain convergence of the phase $\eta(t)$ as $t\to +\infty$.
The latter result is new even in the smooth case.

The treatment for discontinuous waves of decay in low norms $L^p$, $1\leq p\leq 2$ is an interesting open problem that we expect could
be carried out by a suitable modification of the argument for the smooth case in \cite{MZ2}.
A very interesting novelty in either smooth or nonsmooth case, would be to prove decay in $L^1$ at nonuniform rate $\int_{[-a_-t/2,-a_+t/2]^c}|v_0(x)|\, dx$
determined by the tail of the initial data,
where $a_\pm$ denote limiting equilibrium characteristic speeds as $x\to \pm \infty$ for the linearized equations about the wave.
This should be possible using an $L^1$ version of the approximate characteristic estimate \eqref{charest} developed here.
Such a result would at the same time give an alternative, shorter proof of convergence $\eta(t)\to \eta_\infty$ of the phase,
based on conservation of mass of the unforced first coordinate $u$, similar to the classical argument of \cite{L2} for shock profiles
of viscous conservation laws.

We note that all of our nonlinear arguments extend to nondegenerate piecewise smooth relaxation shocks of 
general $n\times n$ systems with scalar equlibrium systems, in particular to the $3\times 3$
Richard-Gavrilyuks (RG) model of \cite{RG1,RG2}.
Thus, the stability problem in that case reduces to an examination of the existence and spectral stability problems.
For $n\times n$ relaxation systems with $r\times r$ equilibrium systems, $r>1$, Lax shocks of the equilibrium system
admit $r-1>0$ outgoing characteristic modes, leading to new, algebraically-decaying contributions from G
source terms in the nonlinear Rankine-Hugoniot equations for which our our current $L^p$-based nonlinear iteration scheme
appears not to close.
However, this should be treatable under further localization conditions on the initial perturbation by a more
detailed pointwise analysis as in \cite{HoZ,RaZ,HRZ}.

Though we do not show it here, in the present case for which the equilibrium behavior corresponds to a scalar shock, 
given the $H^s$ bounds established in Theorem \ref{main},
the weighted norm method of Sattinger \cite{Sa} can be applied in straightforward fashion to yield exponential
decay of $|v(t)|_{L^\infty}$, assuming spatial exponential decay on the initial perturbation.
This yields time-exponential convergence of the phase to a limiting value, giving
the stronger results of time-exponential phase-asymptotic orbital stability.
Similarly, assuming algebraic decay at rate $|v_0(x)|\leq C(1+|x|)^{-r}$, $1<r\leq 3/2$ of the initial data, 
a pointwise analysis as in \cite{Ho,HoZ,RaZ,HRZ} should give time-algebraic convergence to a limiting phase
at rate $|v|_{L^\infty}\leq C(1+t)^{1-r}$, reflecting the rate at which ``mass'', or integral of the conserved quantity $u$,
is convected from initial data to the shock center: more precisely, the rate at which residual mass 
%$\int_{[-a_-t/2, -a_+t/2]^c}v_0(x)\, dx$ converges to zero.
$\int_{[-a_-t, -a_+t]^c}v_0(x)\, dx$ converges to zero, where $a_\pm$ are the characteristic velocities
of the limiting equilibrium systems at $x\to \pm \infty$.
%%%%%%%%%%%%%%%
This rate if not the precise characteristic description is 
%in fact 
obtained in the present analysis
%, at least 
for $r<5/4$; see Remark \ref{linphasermk} for further discussion.  
%For $r\geq 5/4$, our analysis yields the nonsharp rate $(1+t)^{-1/4+\upsilon}$ for any $\upsilon>0$.
For $r\geq 5/4$, we get the nonsharp rate $(1+t)^{-1/4+\upsilon}$ for any $\upsilon>0$.

An interesting new issue in the nonsmooth case is compatibility at time $t=0$ of 
Rankine-Hugoniot conditions and initial perturbation.
In Figure \ref{timeev}, we display the results of numerical time-evolution of a perturbed subshock-type profile , 
first with initial perturbation supported away from the subshock (panels (a)-(d)) and second with piecewise
smooth initial perturbation
supported at the subshock (panels (e)-(h)) and incompatible with the Rankine-Hugoniot conditions at time $t=0$.
In both cases, stability is clear; however, in the second experiment one can see clearly an additional shock discontinuity originating from the subshock, generated by initial incompatibility.

\begin{figure}
\begin{center}
\includegraphics[scale=0.32]{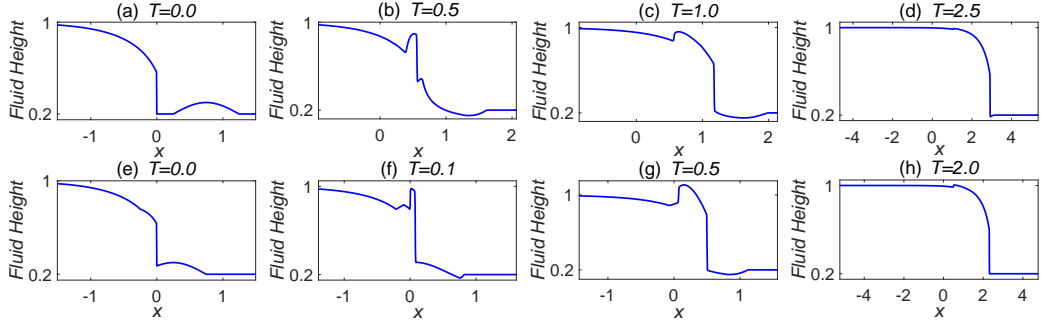}
\end{center}
\caption{
	Time-evolution study using CLAWPACK \cite{C1,C2}, illustrating stability under perturbation of a discontinuous hydraulic shock.  In (a) we show a perturbed profile with $C^\infty$ ``bump-type'' perturbation supported on an interval away from the subshock.  In (b) and (c) we show the solution at intermediate times $T=0.5$ and $1.0$ of the waveform in (a) after evolution under \eqref{sv}; stability and smoothness away from the subshock are clearly visible. In (d) we show
	the solution at time $T=2.5$, exhibiting convergence to a shift of the original waveform (slightly compressed in the horizontal direction due to scaling of the figure).
In (e) we show a perturbed profile with perturbation supported at the subshock.
In (f) and (g) we show the solution at times $T=0.1$ and $0.5$ of the waveform in (e) after evolution under \eqref{sv}; stability is again clear, 
but one can see also an additional shock discontinuity emerging from the subshock and propagating downstream, 
caused by incompatibility of the data with Rankine-Hugoniot conditions at time $0$. 	
In (h), we show the solution at time $T=2.0$, exhibiting convergence to a shift of the original waveform.
}
\label{timeev}
\end{figure}

An interesting open problem would be to analyze the second case by the introduction/tracking of this additional shock wave in the nonlinear Ansatz, ``relieving'' incompatibility at $t=0$.
More generally, it would be interesting to treat lower regularity perturbations than piecewise $H^2$, for example
in piecewise Lipshitz class by a paradifferential damping estimate following \cite{Me}.
To treat perturbations admitting shocks would also be interesting, but appears to require new ideas.
Likewise, in the setting of more general balance laws not admitting a damping estimate, it is not clear how to proceed
even for the case of arbitrarily smooth compatible initial perturbations.
As noted in \cite{JLW}, for example, the time-asymptotic stability of piecewise smooth Zeldovich--von Neumann--Doering 
(ZND) detonations is an important open problem.

Finally, it would be very interesting to attack by techniques like those used here
the open problem cited in \cite{JNRYZ} of nonlinear time-asymptotic
stability of discontinuous periodic ``roll wave'' solutions of \eqref{sv} or its $3\times 3$ analog (RG)
in the hydrodynamically unstable regime $F>2$.
It would appear that a Bloch wave analog of the linear analysis here would apply also for periodic waves,
similar to that of \cite{JZN,JNRZ} in the viscous periodic case; for the requisite Bloch wave framework for
discontinuous waves, see \cite{JNRYZ}.\footnote{
Though, note the degeneracy at $\lambda=0$ of spectral curves of roll wave solutions of
\eqref{sv} described in \cite[Rmks. 2.1 and 5.1]{JNRYZ}, making this case more complicated.
}
A difficulty is the apparent lack of a nonlinear damping estimate given instability of constant states.
However, as suggested by L. M. Rodrigues \cite{R}, one may hope that
an ``averaged'' energy estimate using ``gauge functions'', or specially chosen
weights generalizing the Goodman- and Kawashima-type estimates here,
as used to obtain damping estimates in the viscous case in \cite{RZ} might yield a nonlinear damping estimate here as well.

\medskip

{\bf Note:} Our numerical conclusions have subsequently been verified analytically
by generalized Sturm--Liouville considerations in \cite{SYZ}, yielding a complete analytical proof of stability.

\medskip

{\bf Acknowledgement.} 
We thank L. Miguel Rodrigues, Pascal Noble, and Mat Johnson for numerous enlightening discussions on 
the Saint-Venant equations and shallow water flow.
In particular, discussions in the course of our collaboration \cite{JNRYZ} 
on the parallel case of discontinuous periodic waves,
and especially ideas of Rodrigues \cite{R} toward the associated nonlinear stability problem,
were crucial in our approach to the simpler case of discontinuous shock profiles treated here.
%%%%%%%%%%%%%
Thanks also to Alexei Mailybaev and Dan Marchesin \cite{MM} for discussions on singular detonation waves 
in relaxation models for combustion that were the immediate impetus for our study of hydraulic shock profiles. Thanks to University Information Technology Services (UITS) division from Indiana University for providing the Karst supercomputer environment in which most of our computations were carried out. This research was supported in part by Lilly Endowment, Inc., through its support for the Indiana University Pervasive Technology Institute, and the Indiana METACyt Initiative. The Indiana METACyt Initiative at IU was also supported in part by Lilly Endowment, Inc.

\section{Hydraulic shock profiles of Saint-Venant equations}\label{s:profiles}
We begin by categorizing the family of hydraulic shock profiles, or piecewise smooth
traveling wave solutions of \eqref{sv} with discontinuities consisting of entropy-admissible shocks.
For closely related analysis, see the study of periodic ``Dressler'' waves in \cite[\S 2]{JNRYZ}; as
discussed in Remark \ref{dresslerrmk}, this corresponds to the degenerate case $H_s=H_L$, $F>2$ in our study here.
As the first-order derivative part of \eqref{sv} comprises the familiar equations of isentropic gas dynamics,
entropy-admissble discontinuities are in this case Lax $1$- or $2$-shocks satisfying the Rankine-Hugoniot jump
conditions and Lax characteristic conditions \cite{La,Sm}.

Consider the Saint-Venant equations \eqref{sv}
$$
\displaystyle
\d_th+\d_xq=0,\quad \d_tq+\d_{x}\left(\frac{q^2}{h}+\frac{h^2}{2F^2}\right)=h-\frac{|q|q}{h^2}.
$$
We seek a traveling wave solution $(h,q)=(H,Q)(x-ct)$ with $c$ constant and $(H(\xi),Q(\xi))$ smooth 
with 
\be\label{prof2}
\lim_{\xi\rightarrow-\infty}(H,Q)(\xi)=(H_L,Q_L), \quad \lim_{\xi\rightarrow +\infty}(H,Q)(\xi)=(H_R,Q_R),
\ee
with Lax $1$- or $2$-shocks at each discontinuity.
In smooth regions, we have therefore
\be\label{smooth}
-cH'+(Q)'=0,\quad -cQ'+\left(\frac{Q^2}{H}+ \frac{H^2}{2F^2}\right)'=H-\frac{|Q|Q}{H^2},
\ee
and at sub-shock discontinuities $\xi_j$, we have the Rankine-Hugoniot jump conditions
\be\label{jump}
-c[H]+[Q]=0,\quad -c[Q]+\left[\frac{Q^2}{H}+\frac{H^2}{2F^2}\right]=0,
\ee
where $[f]$ denotes the jump $f(\xi_j^+)-f(\xi_j^-)$ of a quantity $f$ at discontiuity $\xi_j$.

Our first observation is the standard one, true for general $n\times n$ relaxation systems of block structure
$w_t+F(w)_x=\bp 0\\r(w)\ep$, that $(H_L,Q_L)$ and $(H_R,Q_R)$ must necessarily be equilibria, with the
triple $(H_L,H_R,c)$ satisfying the Rankine-Hugoniot conditions 
\be\label{redRH}
c[H]=[q_*(H)]:= [H^{3/2}]
\ee
of the reduced equilibrium system \eqref{CE}, i.e., a (not necessarily entropy-admissible) shock of \eqref{CE}.

Integrating the first equation of \eqref{smooth}, and combining with the first equation of \eqref{jump} gives
\be\label{UHrel}
Q-c H\equiv \const =:-q_0.
\ee
Meanwhile, taking $(H',Q')\to 0$ in \eqref{smooth}(ii), we find that $H_L$ and $H_R$ must be equilibria of
the relaxation system \eqref{sv}, satisfying $Q_{L,R}=q_*(H_{L,R})=H_{L,R}^{3/2}$: in particular, note therefore that
{\it $Q_L, Q_R>0$} in the physical regime $H>0$ that we consider.
Substituting $Q_{L,R}=q_*(H_{L,R})$ into \eqref{UHrel} then gives \eqref{redRH}.
As $q_*(h )=h^{3/2}$ is convex, there are at most two such equilibrium solutions of \eqref{UHrel} for a given 
value of $q_0$, hence, for each possible left state $(H_L,Q_L)$ of \eqref{prof2}, and choice of speed $c$,
there is at most one possible right state $(H_R,Q_R)\neq (H_L,Q_L)$.
Moreover, for such a nontrivial right state to exist, since then $c=[q_*(h)]/[h]$ is given by the Rankine-Hugoniot
conditions for \eqref{CE}, {\it $c$ must necessarily be positive;} from now on, therefore, we take $c>0$.

Next, substituting \eqref{UHrel} in the second equation of \eqref{smooth}, we obtain the scalar ODE
\be\label{altsmooth}
\left(\frac{-q_0^2}{H^2}+\frac{H}{F^2}\right)H'=H-\left|-q_0+cH\right|(-q_0+cH)/H^2
\ee
and, substituting in the second equation of \eqref{jump}, the scalar jump condition
\be\label{altjump}
\left[\frac{q_0^2}{H}+\frac{H^2}{2F^2}\right]=0.
\ee
Since $-q_0+cH=Q$ is monotone in $H$, and (as noted just above) is positive at equilibria $(H_L,Q_L)$
and $(H_R,Q_R)$, we have that $Q$ is positive on $[H_L,H_R]$ and so we may drop the absolute values in
\eqref{altsmooth} in this regime, and in the larger regime $Q>0$, replacing \eqref{altsmooth} by
$$
\left(\frac{-q_0^2}{H^2}+\frac{H}{F^2}\right)H'=\frac{H^3-(-q_0+cH)^2}{H^2}.
$$
As the righthand side is cubic, with zeros at equilibria $H_L$ and $H_R$,
it factors as $(H-H_R)(H-H_L)(H-H_3)$, where $H_3$ is a third root that-- since as observed above,
there can be at most two-- is {\it not} an equilibrium of \eqref{sv}.
It follows that $Q_3=-q_0+cH_3$ must be negative, or else we would have a contradiction;
thus, $H_3< \min\{H_L,H_R\}$; this gives in passing $q_0>0$.

Writing \eqref{sv} in abstract form as $w_t + F(w)_x=(0, r(w))^T$, so that \eqref{smooth} becomes
$(dF(w)-c\Id)W'= (0, r(W))^T$, we see that \eqref{smooth} is singular precisely when the eigenvalues
$\alpha_\pm$ of $(dF-c\Id)$ take value $0$, where (see, e.g., \cite{Sm})
$\alpha_\pm= Q/H \pm \sqrt{H/F^2}-c$, hence by \eqref{UHrel}
\be\label{alphapm}
\alpha_\pm= -q_0/H \pm \sqrt{H}/F
\ee
along a shock profile.  As $q_0>0$, this happens precisely at the ``sonic point'' where $\alpha_+=0$, 
i.e., the shock speed agrees with a characteristic speed of the hyperbolic relaxation system, or, solving:
$ -q_0^2/H^2 + H/F^2=0$.
Comparing with \eqref{altsmooth}, we see that the scalar ODE becomes singular at the same value of $H$.
Following \cite{JNRYZ}, we denote this point as
\be\label{Hs}
H_s:= (q_0 F)^{2/3}.
\ee

Evidently along the profile, the signs of $\alpha_\pm$ are constant for $H$ to the right and left of $H_s$.
Taking $H\to + \infty$, we see that 
\be\label{sig}
\hbox{\rm $\alpha_-< 0 <\alpha_+$ for $H>H_s$ and
$\alpha_-, \alpha_+< 0$ for $H< H_s$.}
\ee
Recalling the Lax characteristic conditions \cite{La,Sm}, we find that the only possible entropy-admissible shock
connections are Lax $2$-shocks from points $\tilde H_L>H_s$ to points $\tilde H_R<H_s$, i.e.,
shocks for which $\alpha_-(\tilde H_L)<0<\alpha_+(\tilde H_L)$ and $\alpha_-(\tilde H_R), \alpha_+(\tilde H_R)<0$.
In particular, any such discontinuities are {\it decreasing in $H$}, with, moreover, $\tilde H_R< H_s < \tilde H_L$.

We find it convenient to introduce a fifth point $H_*$, defined as satisfying the scalar jump condition \eqref{altjump} 
(and thus, along the profile, by \eqref{UHrel}, the full jump conditions \eqref{jump})
when paired with value $H_R$.  Combining all information, we have
\be
\label{eqts}
\begin{aligned}
&H_L-\frac{Q_L^2}{H_L^2}=0,&&H_R-\frac{Q_R^2}{H_R^2}=0,\\
&Q_L-cH_L=Q_R-cH_R=-q_0,&&\frac{q_0^2}{H_*}+\frac{H_*^2}{2F^2}=\frac{q_0^2}{H_R}+\frac{H_R^2}{2F^2}.
\end{aligned}
\ee
Setting $\nu:=\sqrt{\frac{H_L}{H_R}}>1$ and solving for $c,q_0,H_*$ yields
\ba
c=\frac{\nu^2+\nu+1}{\nu+1}\sqrt{H_R}\;
,\quad q_0=\frac{\nu^2}{\nu+1}\sqrt{H_R^3},\quad
H_*=\left\{\begin{aligned}
&H_R\\
&\frac{-\nu-1+\sqrt{8F^2\nu^4+\nu^2+2\nu+1}}{2\left(\nu+1\right)}H_R
\\
&\frac{-\nu-1-\sqrt{8F^2\nu^4+\nu^2+2\nu+1}}{2\left(\nu+1\right)}H_R
\end{aligned}
\right.
\ea
from which we keep the nontrivial physically relevant (positive) solution 
\be 
\label{Hstar}
H_*:=\frac{-\nu-1+\sqrt{8F^2\nu^4+\nu^2+2\nu+1}}{2\left(\nu+1\right)}H_R.
\ee

Substituting $c,q_0$ in (\ref{altsmooth}) now yields
\be
\label{profileODE}
H'=\frac{F^2 \left(H - H_L\right) \left(H - H_R\right) \left(H-H_3\right)}{(H-H_s)(H^2+HH_s+H_s^2)}
\ee
where
\be
\label{H3Hs}
H_3:=\frac{\nu^2}{\nu^2+2\nu+1}H_R,\quad H_s:=\left(\frac{F\nu^2}{\nu+1}\right)^{\frac{2}{3}}H_R.
\ee
Since $\nu>1$, we have $H_3<H_R<H_L$, recovering our earlier observation on the ordering of roots $H_j$.

Our analysis of hydraulic shock profiles is based on the following case structure.

\bl\label{H*lem}
With the notation above:
\begin{itemize}
\item[i.]{$H_s>H_R$ is equivalent to $F\nu^2-\nu-1>0$, 
or $H_L>H_R\frac{1+2F+\sqrt{1+4F}}{2F^2}$. It is always satisfied when $F>2$.}
\item[ii.]{$H_s<H_L$ is equivalent to $F<\nu^2+\nu$, 
or $H_L>H_R\frac{1+2F -\sqrt{1+4F}}{2}$. It is always satisfied when $F<2$,
as is $H_*<H_L$.}
\end{itemize}
\el

\begin{proof}
The quadratic conditions in $\nu$ follow immediately from \eqref{H3Hs}, whence
the boundaries in terms of $H_L$ and $H_R$ follow by the quadratic formula.
Likewise, applying \eqref{Hstar}, we find that
$H_*<H_L$ is equivalent to $2F^2<\nu^2+\frac{1}{\nu^2}+2\nu+\frac{2}{\nu}+2$, which is
always satisfied for $F<2$, by the inequality $z+1/z\geq  2$ for $z>0$.
\eqref{Hstar}

\end{proof}

\bl\label{Hslem}
With the notation above, $H_s$ lies between $H_R$ and $H_*$, and there is an admissible Lax $2$-shock
between the larger of $H_*$, $H_R$ and the smaller.
\el

\begin{proof}
	The function $\tilde q(H):= q_0^2/H + H^2/2F^2$ appearing in the scalar jump condition $[\tilde q]=0$
	is convex, with $c'(H)=-q_0^2/H^2 + H/F^2$ equal to the prefactor in the lefthand side of \eqref{altsmooth},
	with $c'(H_s)=0$ uniquely specifying $H_s$. 
	By convexity, $c(H_*)=c(H_R)$ implies by Rolle's theorem that $c'(H_*)$ and $c'(H_R)$ have opposite signs, with
	$c'>0$ at the larger of the two points, and  $c'$
	vanishes somewhere between, hence $H_s\in (H_*, H_R)$.
	Recalling \eqref{sig}, we see that there is then an (entropy-admissible)
	Lax $2$-shock connecting the larger of $H_*$, $H_R$ to the smaller.
\end{proof}

\begin{proof}[Proof of Proposition \ref{existprop}]
As noted in the discussion above, in all cases necessarily $c>0$ for any shock profile, and $Q>0$ for 
$H \geq H_R$. Since smooth solutions of \eqref{profileODE} cannot cross equilibrium $H_R$, and entropy admissible 
shocks can only decrease $H$, we have that connecting profiles must satisfy $H>H_R$, and thus $Q>0$, for any
choice of parameters.

({\it Case $F<2$}.)
When $0<F<2$, $H_L>H_R\frac{1+2F+\sqrt{1+4F}}{2F^2}$, 
then $H_R<H_s<H_*$, and so, by the factorization \eqref{profileODE}, $H'<0$ on $(H_s,H_L)$,
and thus on $(H_*,H_L)$.
It follows that there exist discontinuous traveling wave solutions as depicted in Figure \ref{profile}(a),
consisting of a smooth piece emanating from the equilibrium of \eqref{altsmooth} at $H_L$ and continuing down to
$H_*$, followed by a Lax $2$-shock from $H_*$ to $H_R$.
However, there does not exist a smooth profile, as the solution emanating from $H_L$ cannot cross the
singular point $H_s$ to reach $H_R$; indeed, one may see by the factorization \eqref{profileODE} that $H'>0$
on $(H_R,H_s)$.

In the limiting case when $H_L=H_R\frac{1+2F+\sqrt{1+4F}}{2F^2}$, 
for which $H_s=H_*=H_R$, there exist piecewise smooth traveling wave solutions as depicted in Figure \ref{profile}(b),
with discontinuous derivative at the endpoint $H_R=H_s$.

In the small amplitude region $H_R<H_L<H_R\frac{1+2F+\sqrt{1+4F}}{2F^2}$, for which 
$H_*<H_s <H_R$, the corresponding smooth traveling wave profile does not pass the singular point, and so
there exist smooth traveling wave solutions as depicted in Figure \ref{profile}(c). 
However, there exist no solutions containing subshocks, as these would necessarily jump below $H_s<H_R$, and so the
solution could never return past $H_s$, since $H'<0$ on $(H_*,H_s)$ blocks approach by smooth solution, and since
any admissible discontinuities can only decrease the value of $H$. See Figure 3 (b) for domain of existence for traveling waves.

({\it Case $F>2$}.) The case $F>2$ goes similarly.  When $H_s>H_L$, we have,
examining the factorization \eqref{profileODE} and using $F>2$, that $H'>0$ on $(H_R,H_L)$, 
and so there exists a smooth ``reverse'' connection from $H_R$ to $H_L$.  
As $H_R<H_s<H_*$, we have in this case that also $H_*>H_L$, and, since also $H'<0$ on $(H_L,H_s)$, there is no way 
to reach $H_*$ starting from either $H_R$ or $H_L$, and so there can be no discontinuous profile connecting equilibria
$H_L$ and $H_R$ in either direction.
In the degenerate case $H_L=H_s$, we find that the factor $(H-H_s)$ in the singular prefactor
$-q_0^2/H^2 + H/F^2$ on the lefthand side of \eqref{profileODE} exactly cancels with the factor $(H-H_L)$ on
the righthand side, and so \eqref{profileODE} reduces to the nonsingular scalar ODE
\be
\label{redprofileODE}
H'=\frac{F^2  \left(H - H_R\right) \left(H-H_3\right)}{(H^2+HH_s+H_s^2)},
\ee
from which we find that $H'>0$ for all $H>H_R$, with no special significance to the point $H_L$.
Noting that $H_*>H_s=H_L$, we see that there exists an entropy-admissible piecewise smooth homoclinic profile
consisting of a smooth part initiating from $H_R$ and increasing to $H_*$, followed by a Lax $2$-shock from 
$H_*$ back to $H_R$, and finally a constant piece $H\equiv H_R$.  As $H_L$ is not an equilibrium of the 
reduced ODE \eqref{redprofileODE}, it cannot be an asymptotic limit and there is no profile connecting to it.
Since $H_R$ is a repellor, it can only be a limit at $+\infty$ if the profile is constant there, and so any connecting profile must be a discontinuous solution starting with a smooth piece from $H_R$ at $-\infty$ and ending with a constant piece
$H\equiv H_R$ near $+\infty$.
However, there exists an uncountable family of multiple-discontinuity homoclinic profiles, in which intermediate
shocks $(H_{2j},H_{2j+1}$ with  $H_*>  H_{2j} > H_s>  H_{2j+1}>H_R$ 
are arbitrarily placed in between, with smooth
pieces connecting  $H_{2j+1}$ to $H_{2j+2}$, where $H_{2j+1}<H_s<H_{2j+2}$.
%%%%%%%%%%%%%%%%%
In the remaining case $H_R<H_s<H_*< H_L$, we have $H'>0$ on $(H_s,H_L)$ and $H'<0$
on $(H_R,H_s)$, hence there is no smooth solution leaving either $H_R$ or $H_L$, and the only admissible
shock is from $H_*$ to $H_R$.  Thus, there is no admissible piecewise smooth profile joining the
two equilibria $H_L$, $H_R$ in either sense.
\end{proof}

\begin{figure}
\begin{center}
\includegraphics[scale=0.32]{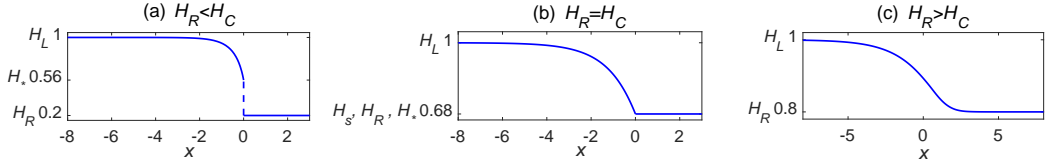}
\end{center}
\caption{Hydraulic shock profiles with $F=1.5$, $H_L=1$ and (a) $H_R=0.2$; (b) $H_R=\frac{9}{8+2\sqrt{7}}$; (c) $H_R=0.8$.}
\label{profile}
\end{figure}

\br\label{dresslerrmk}
The scenario \eqref{redprofileODE} treated in the degenerate case $H_s=H_L$, $F>2$ may be recognized as the
same one considered in \cite[\S 2]{JNRYZ} with regard to existence of periodic
entropy-admissible piecewise smooth relaxation profiles; indeed, existence of periodic and quasiperiodic
profiles follows by essentially the same construction used here to show existence of homoclinic ones.
\er

\begin{obs}[Rescaling]\label{scaleobs}
By scale-invariance of the Saint-Venant equations \cite{BL,JNRYZ}, we may perform the rescaling 
$$
\underline{H}(x)=H_LH(x/H_L),\;\underline{H}_R=\frac{H_R}{H_L}=\frac{1}{\nu^2},\;\underline{H}	_L=1
$$
to obtain a solution $\underline{H}$ for which the left limiting water height is $1$.
From now on, we omit the underline in $\underline{H}$, and simply take $H_L=1$. 
After rescaling, the domain of existence of hydraulic shock profiles with a sub-shock discontinuity is 
\be 
\label{domainexistence}
0<F<2,\;0<H_R<H_C:=\frac{2F^2}{1+2F+\sqrt{1+4F}}.
\ee
\end{obs}

\begin{obs}[Positivity]\label{posobs}
We have shown that $H$ and $Q$ are positive along hydraulic shock profiles $\overline{W}$, hence
also in their vicinity.  It follows that for purposes of investigating stability their stability,
we can drop the absolute value in \eqref{sv}(ii) and write the source term simply as $h-q^2/h^2$,
as we shall do from now on.
We see, further, that $u,c>0$ for steady flow down an incline.
\end{obs}

\section{Majda's type coordinate change and perturbation equations}\label{s:3}
We next recall the general framework introduced by Majda \cite{Ma,Me} for the study of stability of shock waves, 
converting the original free-boundary problem to a standard initial boundary-value problem on a fixed domain.
Consider a general system of balance laws
\be
\label{balancelaws}
w_t+F(w)_x-R(w)=0, \qquad w\in \R^n,
\ee
admitting a traveling wave solution $\overline{W}(x-ct)=\overline{W}(\xi)$ that is smooth and solves \eqref{balancelaws} on $\xi\gtrless 0$ and at $\xi=0$ has a discontinuity satisfying the Rankine-Hugoniot condition:
\be 
\label{rh1}
-c[\overline{W}]+[F(\overline{W})]=0
\ee 
where $[f(\xi)]=f(0^+)-f(0^-)$.

Let $w(x,t;s)$ be a family of perturbed solutions to \eqref{balancelaws} with shock at $x=\zeta(t;s)$ 
and 
$$w(x,t;0)=\overline{W}(x-ct),\quad \zeta(t;0)=ct.$$
Perform the Majda's type coordinate change \cite{Ma} $\tilde{t}=t$, $\xi=\xi(x,t;s)=x-\zeta(t;s)$ and set 
$$
u(\xi,\tilde{t};s):=w(x,t;s)
$$
so that in $u$ the shock front is fixed at $\xi=0$. In $u(\xi,\tilde{t};s)$, balance laws \eqref{balancelaws} become
\be 
\label{relaxation1}
u_{\tilde{t}}+\xi_{t}u_{\xi}+F(u)_\xi-R(u)=0,
\ee
and the Rankine-Hugoniot condition \eqref{rh1} becomes
\be 
\label{rh2}
\xi_{t}\Big|_{\xi=0}[u]+[F(u)]=0.
\ee
Now substituting 
\be 
\label{perturbation}
\xi(x,t;s)=x-ct+\eta(\tilde{t}),\; u(\xi,\tilde{t};s)=\overline{W}(\xi)+v(\xi,\tilde{t})
\ee 
in the interior equation \eqref{relaxation1} and putting linear order terms on the left and 
quadratic order terms on the right, we obtain that perturbations $\eta$, $v$ satisfy
\be
\label{perturbedinterior}
v_{\tilde{t}}+\eta_{\tilde{t}}\overline{W}'+\left((dF(\overline{W})-c\;\Id)v\right)_\xi-dR(\overline{W})v=-\eta_{\tilde{t}}v_\xi-N_1(v,v)_\xi+N_2(v,v)
\ee
where $N_j(v,v)=O(|v|^2)$.
Likewise, substituting \eqref{perturbation} in the Rankine-Hugoniot condition \eqref{rh2} and putting 
linear order terms on the left and quadratic order terms on the right, we obtain, on the boundary $\xi=0$, 
that perturbations $\eta$, $v$ satisfy
\be
\label{perturbedexterior}
\eta_{\tilde{t}}[\overline{W}]+[\left(dF(\overline{W})-c\;\Id\right)v]=-\eta_{\tilde{t}}[v]-[N_1(v,v)].
\ee
\begin{obs}
Specialized to the Saint-Venant equations \eqref{sv}, $N_1(v,v)$, $N_2(v,v)$ are 
\ba
N_1(v,v)=&\left(\begin{array}{c}0\\v^t\int_0^1(1-s)\left(\begin{array}{rr}2\frac{(Q+s v_2)^2}{(H+s v_1)^3}+\frac{1}{F^2}&-\frac{2(Q+s v_2)}{(H+s v_1)^2}\\-\frac{2(Q+s v_2)}{(H+s v_1)^2}&\frac{2}{H+s v_1}\end{array}\right)ds v\end{array}\right),\\
N_2(v,v)=&\left(\begin{array}{c}0\\v^t\int_0^1(1-s)\left(\begin{array}{rr}-\frac{6(Q+s v_2)^2}{(H+s v_1)^4}&\frac{4(Q+s v_2)}{(H+s v_1)^3}\\\frac{4(Q+s v_2)}{(H+s v_1)^3}&\frac{-2}{(H+s v_1)^2}\end{array}\right)ds v\end{array}\right).
\ea
\end{obs}
\section{The Evans-Lopatinsky determinant}\label{s:lop}
Continuing, we derive now a generalized spectral stability condition following \cite{Kr,Ma,Me,Er1,Er2,JLW,Z1,Z2} in the form of an appropriate
``stability function'', or {\it Evans-Lopatinsky determinant}.
Combining \eqref{perturbedinterior}, \eqref{perturbedexterior} along with initial conditions gives:
\ba
\label{lineareq}
v_{\tilde{t}}+\eta_{\tilde{t}}\overline{W}'+\left(Av\right)_\xi-Ev=&-\eta_{\tilde{t}}v_\xi-N_1(v,v)_\xi+N_2(v,v):=I_S(\eta_{\tilde{t}},v,v_\xi),\\
\eta_{\tilde{t}}[\overline{W}]+[Av]=&-\eta_{\tilde{t}}[v]-[N_1(v,v)]:=B_S(\eta_{\tilde{t}},v),\\
v(0,\xi)=&v_0(\xi),\\
\eta(0)=&\eta_0,
\ea
or in ``good unknown'' $\tilde{v}:=v+\eta\overline{W}'$ \cite{JLW,Z1,Z2,JNRYZ}:
\ba
\label{lineareqgood0}
\tilde{v}_{\tilde{t}}+\left(A\tilde{v}\right)_\xi-E\tilde{v}=&I_S,\\
\eta_{\tilde{t}}[\overline{W}]-\eta[R(\overline{W})]+[A\tilde{v}]=&B_S,\\
\tilde{v}(0,\xi)=&v_0(\xi)+\eta_0\overline{W}',\\
\eta(0)=&\eta_0,
\ea
where $A:=dF(\overline{W})-c\;\Id$ and $E:=dR(\overline{W})$. 

\begin{obs}\label{symmobs}
Specialized to the Saint-Venant equation \eqref{sv} with hydraulic shock profile, $A$ and $E$ are
\be
\label{AE}
A=\left(\begin{array}{cc} -c & 1\\ \frac{H}{F^2}-\frac{Q^2}{H^2} & \frac{2Q}{H}-c \end{array}\right),\quad E=\left(\begin{array}{cc} 0 & 0\\ \frac{2Q^2}{H^3}+1 & -\frac{2Q}{H^2} \end{array}\right).
\ee
From \eqref{AE}, we see in passing that the Saint-Venant equations are \emph{simultaneously symmetrizable} in the sense that
there exists a positive definite matrix 
$$
A^0= \left(\begin{array}{cc} \frac{2Q\left(F^2H^3+F^2Q^2+H^3\right)}{F^2H} & -H^3-2Q^2\\ -H^3-2Q^2 & 2HQ \end{array}\right)
$$
such that $A^0A$ and $A^0E$ are symmetric, and $A^0E$ is negative semidefinite.
\end{obs}

Setting $\tilde{\tilde{v}}=\tilde{v}-\eta_0\overline{W}'$ and $\tilde{\eta}=\eta-\eta_0$, then yields
\ba
\label{lineareqgood}
\tilde{\tilde{v}}_{\tilde{t}}+\left(A\tilde{\tilde{v}}\right)_\xi-E\tilde{\tilde{v}}=&I_S,\\
\tilde{\eta}_{\tilde{t}}[\overline{W}]-\tilde{\eta}[R(\overline{W})]+[A\tilde{\tilde{v}}]=&B_S,\\
\tilde{\tilde{v}}(0,\xi)=&v_0(\xi),\\
\tilde{\eta}(0)=&0.
\ea
Hereafter we use $t$, $x$ in place of $\tilde{t}$, $\xi$.

System \eqref{lineareqgood} is essentially the same set of equations studied in \cite{JLW,Z1,Z2} in the context of detonation waves
of the ZND model. As noted in \cite{JLW}, short time existence and continuous dependence in $H^s$, $s\geq 2$, is provided
by the (much simpler, one-d version of the multi-d) analysis of Majda and M\'etivier \cite{Ma,Me} for general
conservation laws; see Section \ref{s:damping} for further details.
In particular, we have for $H^s$ initial data, that a solution exists, is continuous in $H^s$ with respect to time,
and grows in $H^s$ at no more than exponential rate $Ce^{\alpha t}$, so long as $|v|_{H^2(\tilde R)}$ remains bounded;
that is, the solution is of ``exponential type''.
It follows from \cite{D} that the Laplace transform $\check v(x,\lambda):=\int_0^{+\infty} e^{-\lambda s} \tilde{\tilde{v}}(x,s)ds$ 
with respect to $t$ of a bounded solution $\tilde{\tilde{v}}\in H^s$ is well-defined in $H^s$, and that the original solution $\tilde{\tilde{v}}$ 
is recoverable by the inverse Laplace transform formula
\ba\label{ILT}
\tilde{\tilde{v}}(x,t)&:= \frac{1}{2\pi i}P.V.\int_{a-i\infty}^{a+i\infty}e^{\lambda t} \check v(x,\lambda) d\lambda,\\
\tilde{\eta}(t)&:= \frac{1}{2\pi i}P.V.\int_{a-i\infty}^{a+i\infty}e^{\lambda t} \check \eta(\lambda) d\lambda.
\ea

We now solve \eqref{lineareqgood} using the Laplace transform. Carrying out the Laplace transform on 
\eqref{lineareqgood}(i)--(ii) and denoting Laplace transform of $\tilde{\tilde{v}},\;\tilde{\eta},\;I_S,\;B_S$ as $\check{v},\;\check{\eta},\;\check{I}_S,\;\check{B}_S$, yields 
\ba
\label{lineareqgoodLP}
\check{v}_x=&A^{-1}(E-\lambda I-A_x)\check{v}+A^{-1}\check{I}_S(\lambda)+A^{-1}v_0:=\mathcal{A}(\lambda)\check{v}+A^{-1}\check{I}_S(\lambda)+A^{-1}v_0,\\
\check{B}_S(\lambda)=&\check{\eta}[\lambda\overline{W}-R(\overline{W})]+[A\check{v}].
\ea
\begin{definition}
Dropping the inhomogeneous source terms in \eqref{lineareqgoodLP}, the associated eigenvalue equation is defined as 
\ba
\label{eigen-eq}
\lambda\check{v}+(A\check{v})_x=&E\check{v},\\
\check{\eta}[\lambda\overline{W}-R(\overline{W})]+[A\check{v}]=&0.
\ea
\end{definition}
To solve \eqref{lineareqgoodLP}, by the conjugation lemma of \cite{MeZ}, we need to calculate eigenvalues of matrices $\lim_{x\rightarrow \pm \infty}\mathcal{A}(\lambda)=A^{-1}_\pm(E_\pm-\lambda I)$.\\
At $x=-\infty$, the two eigenvalues are
{\tiny
\ba
\label{eigenm}
\gamma_{1,-}(\lambda)=&\frac{F\nu\left(\nu+1\right)\left(-2F+F\nu+F\nu^2-2F\lambda+\sqrt{F^2{\left(\nu^2+\nu-2\right)}^2+4\lambda \nu\left(\nu+1\right)\left(-F^2+2\nu^2+2\nu\right)+4{\lambda}^2\nu^2{\left(\nu+1\right)}^2}\right)}{2\left(-F^2+\nu^4+2\nu^3+\nu^2\right)},\\
\gamma_{2,-}(\lambda)=&\frac{F\nu\left(\nu+1\right)\left(-2F+F\nu+F\nu^2-2F\lambda-\sqrt{F^2{\left(\nu^2+\nu-2\right)}^2+4\lambda \nu\left(\nu+1\right)\left(-F^2+2\nu^2+2\nu\right)+4{\lambda}^2\nu^2{\left(\nu+1\right)}^2}\right)}{2\left(-F^2+\nu^4+2\nu^3+\nu^2\right)}.
\ea}
At $x=+\infty$, the two eigenvalues are
{\tiny
\ba
\label{eigenp}
\gamma_{1,+}(\lambda)=&\frac{F\nu\left(\nu+1\right)\left(F\nu+F\nu^2-2F\nu^3-2F\lambda \nu^2+\sqrt{F^2\nu^2{\left(-2\nu^2+\nu+1\right)}^2+4\lambda \nu\left(\nu+1\right)\left(-F^2\nu^2+2\nu+2\right)+4{\lambda}^2{\left(\nu+1\right)}^2}\right)}{2\left(-F^2\nu^4+\nu^2+2\nu+1\right)},\\ \gamma_{2,+}(\lambda)=&\frac{F\nu\left(\nu+1\right)\left(F\nu+F\nu^2-2F\nu^3-2F\lambda \nu^2-\sqrt{F^2\nu^2{\left(-2\nu^2+\nu+1\right)}^2+4\lambda \nu\left(\nu+1\right)\left(-F^2\nu^2+2\nu+2\right)+4{\lambda}^2{\left(\nu+1\right)}^2}\right)}{2\left(-F^2\nu^4+\nu^2+2\nu+1\right)}.
\ea}
It is easy to verify that in the domain \eqref{domainexistence} there holds
\ba 
\label{signgamma}
&\Re\gamma_{1,-}(\lambda)>0,\;\Re\gamma_{2,-}(\lambda)<0,\quad\text{for all $\Re\lambda>0,\;F<2,\;\nu>1$,}\\
&\Re\gamma_{1,+}(\lambda)>0,\;\Re\gamma_{2,+}(\lambda)>0,\quad\text{for all $\Re\lambda>0,\;\nu>\frac{1+\sqrt{1+4F}}{2F}$}.
\ea
\begin{definition}
We define the domain of consistent splitting $\Lambda$ as 
\be 
\label{consistent-splitting}
\Lambda:=\left\{\lambda:\Re\gamma_{1,-}(\lambda)>0,\;\Re\gamma_{2,-}(\lambda)<0,\;\Re\gamma_{1,+}(\lambda)>0,\;\Re\gamma_{2,+}(\lambda)>0\right\}.
\ee
By \eqref{signgamma}, we see $\{\lambda:\Re\lambda>0\}\subset\Lambda$. (See Figure \ref{consistentdomain} (a) for an example of domain of consistent splitting).
\end{definition}
\begin{figure}
\begin{center}
\includegraphics[scale=0.32]{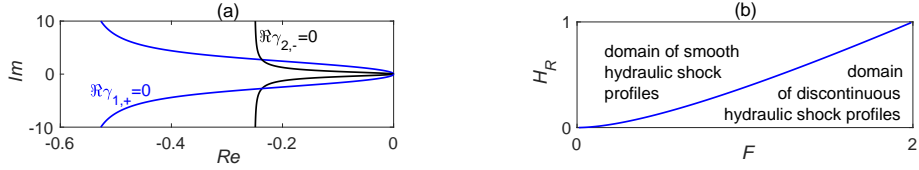}
\end{center}
\caption{(a) Domain of consistent splitting (region to the right of the blue and black curves for $F=1.5$, $H_R=0.2$); (b) Domain of smooth hydraulic shock profiles and domain of discontinuous hydraulic shock profiles (separating by curve $H_C=\frac{2F^2}{1+2F+\sqrt{1+4F}}$).}
\label{consistentdomain}
\end{figure}
By the conjugation lemma of \cite{MeZ}, there exist locally analytic coordinate changes $T_\pm(\lambda,x)$ ($T_+\equiv Id$) on $x\gtrless 0$, converging exponentially to $Id$ as $x\rightarrow\pm\infty$, such that $\check{v}=T_\pm z_\pm$, $A^{-1}(\check{I}_S(\lambda)+v_0)=T_\pm g$ reduce resolvent equation \ref{lineareqgoodLP}(i) to constant coefficients:
\be
\label{con-resolvent}
z_x=A_\pm^{-1}(E_\pm-\lambda I)z+g=\mathcal{A}_\pm(\lambda) z+g.
\ee
Letting $P_{1,2,\pm}(\lambda)$ be the eigenprojections of $\mathcal{A}_\pm(\lambda)$ associated with
eigenvalues $\gamma_{1,2,\pm}(\lambda)$, the solution of \eqref{lineareqgoodLP}(i) on $x\gtrless 0$ can be written as
\be 
\label{checkv}
\check{v}(\lambda,x)=\left\{\begin{aligned}
&T_-(\lambda,x)\Big(e^{\gamma_{1,-}(\lambda)x}P_{1,-}(\lambda)T_-^{-1}(\lambda,0^-)\check{v}(\lambda,0^-)\\
&+\int_{0^-}^x e^{\gamma_{1,-}(\lambda)(x-y)}P_{1,-}(\lambda)T_-^{-1}(\lambda,y)A^{-1}(y)\left(v_0(y)+\check{I}_S(\lambda,y)\right)dy\\&-\int_x^{-\infty}e^{\gamma_{2,-}(\lambda)(x-y)}P_{2,-}(\lambda)T_-^{-1}(\lambda,y)A^{-1}(y)\left(v_0(y)+\check{I}_S(\lambda,y)\right)dy\Big),&x<0,\\
&-\int_x^{+\infty}e^{\mathcal{A}_+(\lambda)(x-y)}A_+^{-1}\left(v_0(y)+\check{I}_S(\lambda,y)\right)dy,&x>0.\end{aligned}\right.
\ee
Here again $P_{1,-}(\lambda)$ is the projection onto the unstable subspace of $\mathcal{A}_-(\lambda)$ and  $P_{2,-}(\lambda)$ is the projection onto the stable subspace of $\mathcal{A}_-(\lambda)$. 
Setting $x=0^\pm$ in \eqref{checkv} yields
\ba 
\label{vpm}
\check{v}(\lambda,0^+)=&-\int_{0^+}^{+\infty}e^{-\mathcal{A}_+(\lambda)y}A_+^{-1}\left(v_0(y)+\check{I}_S(\lambda,y)\right)dy,\\
T_-^{-1}(\lambda,0^-)\check{v}(\lambda,0^-)=&P_{1,-}(\lambda)T_-^{-1}(\lambda,0^-)\check{v}(\lambda,0^-)\\
&-\int_{0^-}^{-\infty}e^{-\gamma_{2,-}(\lambda)y}P_{2,-}(\lambda)T_-^{-1}(\lambda,y)A^{-1}(y)\left(v_0(y)+\check{I}_S(\lambda,y)\right)dy,
\ea
which implies
\ba 
\check{v}(\lambda,0^+)=&-\int_{0^+}^{+\infty}e^{-\mathcal{A}_+(\lambda)y}A_+^{-1}\left(v_0(y)+\check{I}_S(\lambda,y)\right)dy,\\
P_{2,-}(\lambda)T_-^{-1}(\lambda,0^-)\check{v}(\lambda,0^-)=&-\int_{0^-}^{-\infty}e^{-\gamma_{2,-}(\lambda)y}P_{2,-}(\lambda)T_-^{-1}(\lambda,y)A^{-1}(y)\left(v_0(y)+\check{I}_S(\lambda,y)\right)dy.
\ea
Now set $P_{1,-}(\lambda)T_-^{-1}(\lambda,0^-)\check{v}(\lambda,0^-)=\alpha z_{1,-}(\lambda)$ with the scale of $z_{1,-}(\lambda)$ chosen such that
\be 
\label{scalev1m}
T_-(0,x)e^{\gamma_{1,-}(0)x}z_{1,-}(0)=\overline{W}'(x).
\ee
Then, $A(0^-)\check{v}(\lambda,0^-)$ can be written as
\ba
\label{AV0m}
&A(0^-)\check{v}(\lambda,0^-)\\
=&A(0^-)T_-(\lambda,0^-)\left(P_{1,-}(\lambda)+P_{2,-}(\lambda)\right)T_-^{-1}(\lambda,0^-)\check{v}(\lambda,0^-)\\
=&A(0^-)T_-(\lambda,0^-)\alpha z_{1,-}(\lambda)+A(0^-)T_-(\lambda,0^-)P_{2,-}(\lambda)T_-^{-1}(\lambda,0^-)\check{v}(\lambda,0^-)\\
=&A(0^-)T_-(\lambda,0^-)\alpha z_{1,-}(\lambda)\\
&-A(0^-)T_-(\lambda,0^-)\int_{0^-}^{-\infty}e^{-\gamma_{2,-}(\lambda)y}P_{2,-}(\lambda)T_-^{-1}(\lambda,y)A^{-1}(y)\left(v_0(y)+\check{I}_S(\lambda,y)\right)dy.
\ea
Plugging \eqref{AV0m} along with \eqref{vpm}(i) into the matching condition \eqref{lineareqgoodLP}(ii) implies
\ba 
\label{matchingcon}
\check{B}_S(\lambda)=&\check{\eta}[\lambda\overline{W}-R(\overline{W})]+A_+\check{v}(\lambda,0^+)-A(0^-)\check{v}(\lambda,0^-)\\
=&\check{\eta}[\lambda\overline{W}-R(\overline{W})]-\alpha A(0^-)T_-(\lambda,0^-) z_{1,-}(\lambda)\\
&-A_+\int_{0^+}^{+\infty}e^{-\mathcal{A}_+(\lambda)y}A_+^{-1}\left(v_0(y)+\check{I}_S(\lambda,y)\right)dy\\
&+A(0^-)T_-(\lambda,0^-)\int_{0^-}^{-\infty}e^{-\gamma_{2,-}(\lambda)y}P_{2,-}(\lambda)T_-^{-1}(\lambda,y)A^{-1}(y)\left(v_0(y)+\check{I}_S(\lambda,y)\right)dy.
\ea

\begin{definition}
Setting $M(\lambda):=\left[[\lambda W-R(W)]\;\Big|\;A(0^-)T_-(\lambda,0^-)z_{1,-}(\lambda)\right]$, on the domain of consistent splitting, we define the Evans-Lopatinsky determinant function $\Delta(\lambda)$ as
\be\label{lopatinsky}
\Delta(\lambda):=\det(M(\lambda)).
\ee
\end{definition}

By construction, the Evans-Lopatinsky function is analytic on the set of consistent splitting,
in particular on $\{ \lambda: \; \Re \lambda \geq 0\}\setminus \{0\}$.
Moreover, by separation of eigenvalues of $\mathcal{A}_-$ at $\lambda=0$, the associated eigenvectors and projections
may be extended analytically to a neighborhood of $\lambda=0$, allowing us to extend $\Delta$ analytically to
a neighborhood of $\{\lambda:\; \Re \lambda \geq 0\}$.
(For origins of this standard argument, see, e.g., \cite{PW,GZ,ZH}.)

\begin{definition}
Following \cite{Er1,JLW,Z1,Z2,GZ}, 
we say that a profile $\overline{W}$ is \emph{Evans-Lopatinsky stable} if $\Delta(\lambda)$
has no zeros on $\{\Re \lambda \geq 0\}$ save for a single, multiplicity-one root at $\lambda=0$.
\end{definition}

\br\label{genspecrmk}
Evidently, Evans-Lopatinsky stability is a generalized spectral stability condition
correponding with the usual notion of spectral stability on the set of consistent
splitting, namely, absence of eigenvalues, but also including information on
the embedded eigenvalue $\lambda=0$ lying on the boundary of the domain of consistent splitting.
\er

\section{Integral kernels and representation formula}\label{s:ker}
With the defined Evans-Lopatinsky determinant matrix $M(\lambda)$, equation \eqref{matchingcon} rewrites as
\ba 
&M(\lambda)\left(\begin{array}{c}\check{\eta}\\-\alpha\end{array}\right)\\=&\check{B}_S(\lambda)+A_+\int_{0^+}^{+\infty}e^{-\mathcal{A}_+(\lambda)y}A_+^{-1}\left(v_0(y)+\check{I}_S(\lambda,y)\right)dy\\&-A(0^-)T_-(\lambda,0^-)\int_{0^-}^{-\infty}e^{-\gamma_{2,-}(\lambda)y}P_{2,-}(\lambda)T_-^{-1}(\lambda,y)A^{-1}(y)\left(v_0(y)+\check{I}_S(\lambda,y)\right)dy.
\ea
When $M$ is invertible ($\Delta\neq 0$), solving for $\alpha$, $\check\eta$ yields solutions for equation \eqref{lineareqgood}:
{\small
\ba
\label{inverseresolvent}
\check{v}(\lambda,x)
=&\left\{\begin{aligned}&T_-(\lambda,x)\Bigg(e^{\gamma_{1,-}(\lambda)x}z_{1,-}(\lambda)\left(\begin{array}{rr}0 &-1\end{array}\right)M^{-1}(\lambda)\Big(\check{B}_S(\lambda)\\
&+A_+\int_{0^+}^{+\infty}e^{-\mathcal{A}_+(\lambda)y}A_+^{-1}\left(v_0(y)+\check{I}_S(\lambda,y)\right)dy\\&-A(0^-)T_-(\lambda,0^-)\int_{0^-}^{-\infty}e^{-\gamma_{2,-}(\lambda)y}P_{2,-}(\lambda)T_-^{-1}(\lambda,y)A^{-1}(y)\left(v_0(y)+\check{I}_S(\lambda,y)\right)dy\Big)\\
&+\int_0^xe^{\gamma_{1,-}(\lambda)(x-y)}P_{1,-}(\lambda)T_-^{-1}(\lambda,y)A^{-1}(y)\left(v_0(y)+\check{I}_S(\lambda,y)\right)dy\\
&-\int_x^{-\infty}e^{\gamma_{2,-}(\lambda)(x-y)}P_{2,-}(\lambda)T_-^{-1}(\lambda,y)A^{-1}(y)\left(v_0(y)+\check{I}_S(\lambda,y)\right)dy\Bigg),&x<0,\\&-\int_{x}^{+\infty}e^{\mathcal{A}_+(\lambda)(x-y)}A_+^{-1}\left(v_0(y)+\check{I}_S(\lambda,y)\right)dy,&x>0,\\\end{aligned}\right.\\
\check{\eta}(\lambda)=&\left(\begin{array}{rr}1 &0\end{array}\right)M^{-1}(\lambda)\Bigg(\check{B}_S(\lambda)+A_+\int_0^{+\infty}e^{-\mathcal{A}_+(\lambda)y}A_+^{-1}\left(v_0(y)+\check{I}_S(\lambda,y)\right)dy\\
&-A(0^-)T_-(\lambda,0^-)\int_0^{-\infty}e^{-\gamma_{2,-}(\lambda)y}P_{2,-}(\lambda)T_-^{-1}(\lambda,y)A^{-1}(y)\left(v_0(y)+\check{I}_S(\lambda,y)\right)dy\Bigg).
\ea}
Following the standard analysis in \cite{ZH}\cite{MZ}, we define the interior source resolvent kernel functions $\tilde{G}_\lambda$, $G_{1,\lambda}$, and $G_\lambda$ as follows.
\begin{definition} Setting $\check{B}_S(\lambda)=0$ in \eqref{inverseresolvent} and gathering terms in different $x$, $y$ locations, the interior source $\check{v}$-resolvent kernel $\tilde{G}_\lambda(x;y)$ is defined as
\ba
&\tilde{G}_\lambda(x;y):=\\
&\left\{\begin{aligned}
&-e^{\mathcal{A}_+(\lambda)(x-y)}A_+^{-1},
&0<x<y,\\
&0,&0<x,y<x,\\
&T_-(\lambda,x)e^{\gamma_{1,-}(\lambda)x}z_{1,-}(\lambda)\left(\begin{array}{rr}0 &-1\end{array}\right)M^{-1}(\lambda)A_+e^{-\mathcal{A}_+(\lambda)y}A_+^{-1},&x<0,y>0,\\
&-T_-(\lambda,x)e^{\gamma_{1,-}(\lambda)(x-y)}P_{1,-}(\lambda)T_-^{-1}(\lambda,y)A^{-1}(y)+T_-(\lambda,x)e^{\gamma_{1,-}(\lambda)x}z_{1,-}(\lambda)\times\\&\left(\begin{array}{rr}0 &-1\end{array}\right)M^{-1}(\lambda)A(0^-)T_-(\lambda,0^-)e^{-\gamma_{2,-}y}P_{2,-}(\lambda)T_-^{-1}(\lambda,y)A^{-1}(y),&x<y<0,\\
&T_-(\lambda,x)e^{\gamma_{2,-}(\lambda)(x-y)}P_{2,-}(\lambda)T_-^{-1}(\lambda,y)A^{-1}(y)+T_-(\lambda,x)e^{\gamma_{1,-}(\lambda)x}z_{1,-}(\lambda)\times\\&\left(\begin{array}{rr}0 &-1\end{array}\right)M^{-1}(\lambda)A(0^-)T_-(\lambda,0^-)e^{-\gamma_{2,-}y}P_{2,-}(\lambda)T_-^{-1}(\lambda,y)A^{-1}(y),&y<x<0,
\end{aligned}
\right.
\ea
and the interior source $\check{\eta}$-resolvent kernel $G_{1,\lambda}$ is defined as
\be \label{G1lambda}
G_{1,\lambda}(y):=\left\{\begin{aligned}&\left(\begin{array}{rr}1 &0\end{array}\right)M^{-1}(\lambda)A_+e^{-\mathcal{A}_+(\lambda)y}A_+^{-1}
,&y>0,\\
 &\left(\begin{array}{rr}1 &0\end{array}\right)M^{-1}(\lambda)A(0^-)T_-(\lambda,0^-)e^{-\gamma_{2,-}y}P_{2,-}(\lambda)T_-^{-1}(\lambda,y)A^{-1}(y),&y<0.\end{aligned}\right.
\ee
Let 
\ba 
G_\lambda(x;y):=\tilde{G}_\lambda(x;y)-\overline{W}'(x)G_{1,\lambda}(y),
\ea
and split $G_\lambda$ into two parts $G_\lambda=G^1_\lambda+G^2_\lambda$, where $G^1_\lambda$, $G^2_\lambda$ are defined as
\ba
\label{resolventkernel1}
G^1_\lambda(x;y):=&\left\{\begin{aligned}
&-e^{\mathcal{A}_+(\lambda)(x-y)}A_+^{-1},
&0<x<y,\\
&-T_-(\lambda,x)e^{\gamma_{1,-}(\lambda)(x-y)}P_{1,-}(\lambda)T_-^{-1}(\lambda,y)A^{-1}(y),
&x<y<0,\\
&T_-(\lambda,x)e^{\gamma_{2,-}(\lambda)(x-y)}P_{2,-}(\lambda)T_-^{-1}(\lambda,y)A^{-1}(y), &y<x<0,\\
&0, &otherwise,\\
\end{aligned}
\right.\\
G^2_\lambda(x;y):=&\left\{\begin{aligned}&T_-(\lambda,x)e^{\gamma_{1,-}(\lambda)x}z_{1,-}(\lambda)\left(\begin{array}{rr}0 &-1\end{array}\right)M^{-1}(\lambda)A_+e^{-\mathcal{A}_+(\lambda)y}A_+^{-1}\\
				&-T_-(0,x)e^{\gamma_{1,-}(0)x}z_{1,-}(0)\left(\begin{array}{rr}1 &0\end{array}\right)M^{-1}(\lambda)A_+e^{-\mathcal{A}_+(\lambda)y}A_+^{-1},&x<0,y>0,\\
    &T_-(\lambda,x)e^{\gamma_{1,-}(\lambda)x}z_{1,-}(\lambda)\left(\begin{array}{rr}0 &-1\end{array}\right)M^{-1}(\lambda)A(0^-)T_-(\lambda,0^-)\times\\&e^{-\gamma_{2,-}y}P_{2,-}(\lambda)T_-^{-1}(\lambda,y)A^{-1}(y)\\
    &-T_-(0,x)e^{\gamma_{1,-}(0)x}z_{1,-}(0)\left(\begin{array}{rr}1 &0\end{array}\right)M^{-1}(\lambda)A(0^-)T_-(\lambda,0^-)\times\\&e^{-\gamma_{2,-}y}P_{2,-}(\lambda)T_-^{-1}(\lambda,y)A^{-1}(y),&x<0,y<0,\\
&0,&x>0.
\end{aligned}\right.
\ea
These can be written alternatively as
{\scriptsize
\ba
\label{resolventkernel2}
G^1_\lambda(x;y):=&\left\{\begin{aligned}
&-\mathcal{F}_{\lambda}^{y\rightarrow x}A_+^{-1},
&0<x<y,\\
&-\mathcal{F}_{\lambda}^{y\rightarrow x}\Pi_{\lambda,s}(y)A^{-1}(y),&x<y<0,\\
&\mathcal{F}_{\lambda}^{y\rightarrow x}\Pi_{\lambda,u}(y)A^{-1}(y),&y<x<0,\\
&0,&otherwise,\\
\end{aligned}
\right.\\
G^2_\lambda(x;y):=&\left\{\begin{aligned}&-\left(\mathcal{F}_\lambda^{0^-\rightarrow x}\Pi_{\lambda,s}(0^-)\left(\begin{array}{rr}0 &1\end{array}\right)+\overline{W}'(x)\left(\begin{array}{rr}1 &0\end{array}\right)\right)M^{-1}(\lambda)A_+\mathcal{F}_\lambda^{y\rightarrow 0^+}A_+^{-1}
,&x<0,y>0,\\
 &-\left(\mathcal{F}_\lambda^{0^-\rightarrow x}\Pi_{\lambda,s}(0^-)\left(\begin{array}{rr}0 &1\end{array}\right)+
\overline{W}'(x)\left(\begin{array}{rr}1 &0\end{array}\right)\right)M^{-1}(\lambda)A(0^-)\mathcal{F}_\lambda^{y\rightarrow 0^-}\Pi_{\lambda,u}(y)A^{-1}(y),&x<0,y<0,\\
&0,&x>0,\end{aligned}\right.
\ea}
where $\mathcal{F}_\lambda^{y\rightarrow x}$ is the solution operator from $y$ to $x$ of eigenvalue equation \eqref{eigen-eq} and $\Pi_{\lambda,s}$ ($\Pi_{\lambda,u}$) is the projection onto the stable (unstable) flow as $x\rightarrow-\infty$.
\end{definition}

In addition to these interior source kernels analogous to those of the smooth profile case \cite{ZH}\cite{MZ}, 
we define the boundary source 
$\check{v}$-, $\check{\eta}$-resolvent kernel functions $\tilde{K}_\lambda$, $K_{1,\lambda}$ as follows.
\begin{definition} \label{bsourceGkernel}
Setting $\check{I}_S(\lambda,y)=0$, $v_0(y)=0$ in \eqref{inverseresolvent} and gathering terms in different $x$ locations, 
the boundary source $\check{v}$-resolvent kernel $\tilde{K}_\lambda(x)$ is defined as
\be 
\tilde{K}_\lambda(x)=\left\{\begin{aligned}&0,&x>0,\\
									      &T_-(\lambda,x)e^{\gamma_{1,-}(\lambda)x}z_{1,-}(\lambda)\left(\begin{array}{rr}0 &-1\end{array}\right)M^{-1}(\lambda),&x<0,
\end{aligned}\right.
\ee
and the boundary source $\check{\eta}$-resolvent kernel $K_{1,\lambda}$ is defined as
\be
\label{K1lambda}
K_{1,\lambda}=\left(\begin{array}{rr}1&0\end{array}\right)M^{-1}(\lambda),
\ee
and we set
\be 
K_\lambda(x):=\tilde{K}_\lambda(x)-\overline{W}'(x)K_{1,\lambda}.
\ee
\end{definition}

\begin{lemma}
\label{GKanalytic} $G_\lambda$ and $K_\lambda$ are analytic near $\lambda=0$.
\end{lemma}

\begin{proof}
It suffices to show that $\left(\begin{array}{rr}1 &1\end{array}\right)M^{-1}(\lambda)$ is analytic at $0$. 
\be
\left(\begin{array}{rr}1 &1\end{array}\right)M^{-1}(\lambda)=\left(\begin{array}{rr}1 &1\end{array}\right)\frac{1}{\det(M(\lambda))}\left(\begin{array}{rr}(Aw_{1,-}(\lambda,0^-))_2&-(Aw_{1,-}(\lambda,0^-))_1\\ {(-\lambda Q+(H-\frac{Q^2}{H^2}))}&{(\lambda H)}\end{array}\right).
\ee
Since $0$ is a simple root of $\det(M(\lambda))$, $0$ will not be a pole of $\left(\begin{array}{rr}1 &1\end{array}\right)M^{-1}(\lambda)$ if
\ba
&\left(\begin{array}{rr}1 &1\end{array}\right)\left(\begin{array}{rr}(A\overline{W}'(0^-))_2&-(A\overline{W}'(0^-))_1\\ {(H-\frac{Q^2}{H^2})}&0\end{array}\right)\\
=&\left(\begin{array}{rr}(A\overline{W}'(0^-))_2 -R(\overline{W}(0^-))_2+R(\overline{W}(0^+))_2&-(A\overline{W}'(0^-))_1\end{array}\right)
\ea
vanishes. But, it does vanish because $\overline{W}$ is a traveling wave solution to \eqref{sv}.
\end{proof}

\begin{definition}\label{kerdef}
The corresponding interior/boundary source Green kernels are defined as
\ba
\label{Greenkernel}
\tilde{G}(x,t;y):=&\frac{1}{2\pi i}P.V.\int_{a-i\infty}^{a+i\infty}e^{\lambda t}\tilde{G}_\lambda(x;y)d\lambda,&G_1(t;y):=&\frac{1}{2\pi i}P.V.\int_{a-i\infty}^{a+i\infty}e^{\lambda t}G_{1,\lambda}(y)d\lambda,\\
G(x,t;y):=&\frac{1}{2\pi i}P.V.\int_{a-i\infty}^{a+i\infty}e^{\lambda t}G_\lambda(x;y)d\lambda
,&G^{1,2}(x,t;y):=&\frac{1}{2\pi i}P.V.\int_{a-i\infty}^{a+i\infty}e^{\lambda t}G^{1,2}_\lambda(x;y)d\lambda,\\
\tilde{K}(x,t):=&\frac{1}{2\pi i}P.V.\int_{a-i\infty}^{a+i\infty}e^{\lambda t}\tilde{K}_\lambda(x)d\lambda,&K_1(t):=&\frac{1}{2\pi i}P.V.\int_{a-i\infty}^{a+i\infty}e^{\lambda t}K_{1,\lambda}d\lambda,\\
K(x,t):=&\frac{1}{2\pi i}P.V.\int_{a-i\infty}^{a+i\infty}e^{\lambda t}K_\lambda(x)d\lambda,
\ea
where $a$ is a sufficiently large number.
\end{definition}

\begin{proposition}
\label{kernelrelation}
The interior/boundary source Green kernels satisfy 
\be
\tilde{K}(x,t)-\overline{W}'(x)K_1(t)=K(x,t),\quad \tilde{G}(x,t;y)-\overline{W}'(x)G_1(t;y)=G(x,t;y).
\ee
\end{proposition}
With these definitions, equations \eqref{inverseresolvent} can be rewritten in the concise form 
\ba 
\check{v}(\lambda,x)=&\tilde{K}_\lambda(x)\check{B}_S(\lambda)+\int_{-\infty}^\infty \tilde{G}_\lambda(x;y)\left(v_0(y)+\check{I}_S(\lambda,y)\right)dy,\\
\check{\eta}(\lambda)=&K_{1,\lambda}\check{B}_S(\lambda)+\int_{-\infty}^\infty G_{1,\lambda}(y)\left(v_0(y)+\check{I}_S(\lambda,y)\right)dy.
\ea 
Formally exchanging the order of integration in the inverse Laplace tranform formula \eqref{ILT},
we get finally, a formal description of the solution to \eqref{lineareqgood} as
\ba 
\label{tildesol}
\tilde{\tilde{v}}(x,t)=&\int_0^t\tilde{K}(x,t-s)B_S(s)ds+\int_{-\infty}^\infty \tilde{G}(x,t;y)v_0(y)dy+\int_0^t\int_{-\infty}^\infty \tilde{G}(x,t-s;y)I_S(s,y)dyds,\\
\tilde{\eta}(t)=&\int_0^t K_1(t-s)B_S(s)ds+\int_{-\infty}^\infty G_1(t;y)v_0(y)dy+\int_0^t\int_{-\infty}^\infty G_1(t-s;y)I_S(s,y)dyds
.\ea 
Translating from good unknowns back to original coordinates
and validating rigorously the formal exchange of integration,
we consolidate our results in the following integral representation.

\begin{proposition}\label{p:integral_rep}
For $v$ uniformly bounded in $H^2$, the solution of \eqref{lineareq} may be written as
\ba 
\label{originalsol}
v(x,t)=&\int_0^tK(x,t-s)B_S(s)ds+\int_{-\infty}^\infty G(x,t;y)v_0(y)dy+\int_0^t\int_{-\infty}^\infty G(x,t-s;y)I_S(s,y)dyds,\\
\eta(t)=&\eta_0+\int_0^t K_1(t-s)B_S(s)ds+\int_{-\infty}^\infty G_1(t;y)v_0(y)dy+\int_0^t\int_{-\infty}^\infty G_1(t-s;y)I_S(s,y)dyds,
\ea 
where $K$, $G$, $K_1$, and $G_1$ defined in \eqref{Greenkernel} are distributions of order at most two, i.e., expressible 
as the sum of at most second-order derivatives of measurable functions.\footnote{In fact as we show in the following 
section, they are precisely of order one.}
\end{proposition}

\begin{proof}
Using $\tilde{\tilde{v}}-\overline{W}'\tilde{\eta}=v$ and Proposition \ref{kernelrelation}, 
\eqref{originalsol} follows formally by subtracting $\overline{W}'$ times \eqref{tildesol}(ii) from \eqref{tildesol}(i). 
Thus, the issue is to show that, interpreted in the sense of distributions,
the order of integration may be exchanged in the double-integral terms of \eqref{ILT} expanded as
\ba\label{ILTvexpand}
\tilde{\tilde{v}}(x,t)&= \frac{1}{2\pi i}P.V.\int_{a-i\infty}^{a+i\infty}e^{\lambda t} \check v(x,\lambda) d\lambda\\
&= \frac{1}{2\pi i}P.V.\int_{a-i\infty}^{a+i\infty}e^{\lambda t} \tilde{K}_\lambda(x)\check B_S(\lambda)  d\lambda
+
\frac{1}{2\pi i}P.V.\int_{a-i\infty}^{a+i\infty}e^{\lambda t} 
\int_{-\infty}^{+\infty} \tilde{G}_\lambda(x,y) v_0(y) dy\, d\lambda\\
&\quad+
\frac{1}{2\pi i}P.V.\int_{a-i\infty}^{a+i\infty}e^{\lambda t} 
\int_{-\infty}^{+\infty} \tilde{G}_\lambda(x,y)\check I_S(y,\lambda) dy\, d\lambda,
\ea
and
\ba\label{ILTetaexpand}
\tilde \eta(x,t)&= \frac{1}{2\pi i}P.V.\int_{a-i\infty}^{a+i\infty}e^{\lambda t} \check \eta(x,\lambda) d\lambda\\
&= \frac{1}{2\pi i}P.V.\int_{a-i\infty}^{a+i\infty}e^{\lambda t} K_{1,\lambda} \check B_S(\lambda)  d\lambda
+
\frac{1}{2\pi i}P.V.\int_{a-i\infty}^{a+i\infty}e^{\lambda t} 
\int_{-\infty}^{+\infty} G_{1,\lambda}(y) v_0(y) dy\, d\lambda\\
&\quad+
\frac{1}{2\pi i}P.V.\int_{a-i\infty}^{a+i\infty}e^{\lambda t} 
\int_{-\infty}^{+\infty} G_{1,\lambda}(y)\check I_S(y,\lambda) dy\, d\lambda,
\ea
the single-integral terms being treatable by the standard property that the inverse transform of a product is the 
convolution of inverse transforms of its factors.

The double-integral terms may be treated similarly as in \cite{ZH,MZ,MZ2}
by a standard device used in semigroup theory to validate the inverse Laplace transform representation 
of the solution operator \cite[\S 1.7, pp. 28-29]{Pa}, adapted to the context of integral kernels.
Namely, applying the resolvent kernel identity $ \tilde G_\lambda= (L\tilde G_\lambda + \delta_y)/\lambda$
deriving from the defining property $(\lambda-L) \tilde G_\lambda= \delta_y$ of the interior resolvent kernel $\tilde G_\lambda$,
we may factor
$$
\tilde G_\lambda= L^2 \tilde G_\lambda/\lambda^2  + L\delta_y/\lambda^2 + \delta_y/\lambda.
$$

By the crude high-frequency bound 
\be\label{crudebd}
(d/dx)^k \tilde G_\lambda(x,y)\leq Ce^{-\eta|x-y|}
\ee
for $k\geq 0$ and $\Re \lambda \geq \alpha$, $\alpha$ sufficiently large, carried out in Section \ref{HighFrequency},
we have therefore that term
$
\frac{1}{2\pi i}P.V.\int_{a-i\infty}^{a+i\infty}e^{\lambda t} 
\int_{-\infty}^{+\infty} \tilde G_\lambda(x,y)\check I_S(y,\lambda) dy \,d\lambda
$
in \eqref{ILTvexpand} may be expanded as $L^2$ applied to the integral
$ \frac{1}{2\pi i}P.V.\int_{a-i\infty}^{a+i\infty}e^{\lambda t} 
\int_{-\infty}^{+\infty} \tilde G_\lambda(x,y)v_0(y)/\lambda^2 dy\, d\lambda $
plus two explicitly evaluable terms.

Observing for $\Re \lambda=a$ fixed that the integrand $e^{\lambda t}  \tilde G_\lambda(\cdot,y) v_0(y)/\lambda^2$
is absolutely integrable in $(y, \lambda)$, we have by Fubini's theorem that we may switch the order of integration to 
obtain instead $L^2$ applied to the limit
$
\frac{1}{2\pi i} \int_{-\infty}^{+\infty} P.V.\int_{a-i\infty}^{a+i\infty}e^{\lambda t} 
\tilde G_\lambda(x,y)  v_0(y)/\lambda^2  d\lambda\, dy,
$
which, since limits and derivatives of distributions freely exchange, is equal to 
$$
\frac{1}{2\pi i} \int_{-\infty}^{+\infty} P.V.\int_{a-i\infty}^{a+i\infty}e^{\lambda t} 
L^2 \tilde G_\lambda(x,y)v_0(y)/\lambda^2  d\lambda\, dy.
$$
We find in passing that the result is a distribution of at most order $2$, since it is expressible as
the second-order derivative operator $L^2$ applied to a measurable function.

Likewise, we find by standard inverse Laplace transform computations that the order of integration may be exchanged in
$
\frac{1}{2\pi i}P.V.\int_{a-i\infty}^{a+i\infty}e^{\lambda t} \int_{-\infty}^{+\infty} 
\delta_y/\lambda^2 dy\, d\lambda
$
and
$
\frac{1}{2\pi i}P.V.\int_{a-i\infty}^{a+i\infty}e^{\lambda t} \int_{-\infty}^{+\infty} 
\delta_y/\lambda \,dy\, d\lambda,
$
validating the exchange in order for the entire $\tilde G$-term in \eqref{ILTvexpand}.
The first term is at most order $2$ since expressible as the first-order operator applied to an order-$1$ distribution
(the delta-function), while the second is order $1$.  Thus, the entire term is at most of order $2$.

The $G_1$ term in \eqref{ILTetaexpand} goes similarly, using the defining relation $(\lambda-L)G_{1,\lambda}=0$.
Thus, the order of integration may be exchanged also for double-integral terms of
\eqref{ILTetaexpand} may be expanded as $L^2$ applied to the integral
\eqref{originalsol}(ii), justifying \eqref{originalsol}(ii); at the same time this shows that $G_1$ is a distribution
of at most order $2$.
(Alternatively, observing that the terms in the representation of $\tilde \eta$ are expressible as functions of $\tilde{\tilde {v}}(0,t)$, 
we may conclude \eqref{tildesol}(ii) directly from \eqref{tildesol}(i).)

Similarly, using the property $\tilde K_\lambda=L \tilde K_{\lambda}/\lambda$, and the uniform bound $|\tilde K_\lambda|_{H^s}\leq C$ for
$\Re \lambda$ sufficiently large obtained in Section \ref{HighFrequency}, we find that 
$$
\tilde K(x,t):= \frac{1}{2\pi i}P.V.\int_{a-i\infty}^{a+i\infty}e^{\lambda t} \tilde K_\lambda(x) d\lambda
=
L^2\frac{1}{2\pi i}P.V.\int_{a-i\infty}^{a+i\infty}e^{\lambda t} \tilde K_\lambda(x)/\lambda^2 d\lambda
$$
factors as $L^2$ applied to an $H^s$ function defined by the absolutely convergent integral of 
$$e^{\lambda t} \tilde K_\lambda(x)/\lambda^2 d\lambda=O(1/|\lambda|^2),$$ so is a distribution of order at most $2$.
Finally, using the large-$|\lambda|$ bound $K_{1,\lambda}=V_h/\lambda + O(1/|\lambda|^2)$ obtained in 
\eqref{lambdaK1lambda}, Section \ref{HighFrequency},
we find that $K_1(t):= \frac{1}{2\pi i}P.V.\int_{a-i\infty}^{a+i\infty}e^{\lambda t} K_{1,\lambda} d\lambda$ decomposes
into the sum of an explicitly evaluable, constant term 
$$V_h \frac{1}{2\pi i}P.V.\int_{a-i\infty}^{a+i\infty}e^{\lambda t}\lambda^{-1}d\lambda=V_h$$
and an absolutely convergent integral
$\frac{1}{2\pi i}P.V.\int_{a-i\infty}^{a+i\infty}e^{\lambda t}O(|\lambda|^{-2})d\lambda $, hence is a $C^0$ function 
with respect to $t$. 
\end{proof}

\br\label{movermk}
Noting (Section \ref{HighFrequency}) that the crude
high-frequency estimate \eqref{crudebd} holds for $\Re \lambda \geq -b$
and $|\lambda|\geq R$ for $b>0$ sufficiently small and $R>0$ sufficiently large, we find by the same analysis
used to justify exchange of integration order in the proof of Proposition \ref{p:integral_rep}
that the contour $P.V. \int_{a-i\infty}^{a+i\infty}$ in \eqref{Greenkernel}, Definition \ref{kerdef}
(interpreted in distributional sense) may be deformed to
\be\label{deformed}
\lim_{M\to \infty}\Big(
\int_{-b-iM}^{-b-iR}
+\int_{-b-iR}^{a-iR}
+\int_{a-iR}^{a+iR}
+\int_{a+iR}^{-b+iR}
+\int_{-b+iR}^{-b+iM}
\Big)
\ee
for $b>0$ sufficiently small and $R>0$ sufficiently large.
This simplifies somewhat the corresponding analysis of \cite{MZ} based on more detailed bounds.
\er

\section{Resolvent estimates}\label{s:resolvent}
We now derive bounds on the various resolvent kernels, on the crucial large- and small-$|\lambda|$ regimes,
corresponding via the usual frequency/temporal duality for the Laplace transform to small- and large-$t$ behavior of the associated time-evolutionary Green kernels.
These are obtained with no a priori assumption of spectral stability, that is, we establish in the course of our analysis
{\it rigorous high- and low-frequency Evans-Lopatinsky stability}.
Intermediate frequencies $1/C\leq |\lambda|\leq C$ for $C>0$ yield by construction immediately {\it uniform exponential estimates}
\be\label{intest}
|\tilde G_\lambda(x,y)|\leq Ce^{-\bar \eta |x-y|}, \quad \bar \eta>0,
\ee
and etc., provided the Evans-Lopatinsky condition is satisfied, hence their analysis is trivial in this sense.
On the other hand, the verification of the Evans-Lopatinsky condition appears to be quite complicated in this regime,
and we find it necessary to carry this out numerically (see Section \ref{s:numerics}).

\subsection{High Frequency analysis} \label{HighFrequency}
We now study behavior of system \eqref{eigen-eq}(i) in the high frequency regime. Denote $w=\check{v}$ and write \eqref{eigen-eq}(i) as
\be
\label{originalsystem}
w_x=-\lambda A^{-1}w+A^{-1}(E-A_x)w
\ee
and perform two diagonalizations $U:=\tilde{R}R^{-1}w$ similar to procedures in High Frequency analysis in \cite{JNRYZ}, we reach a 2 by 2 system in which $U$ satisfies
\be 
\label{diageigen}
U'=\left(\lambda\left(\begin{array}{rr}\mu_1(x)&0\\0&\mu_2(x)\end{array}\right)+\left(\begin{array}{rr}M_{11}(x)&0\\0&M_{22}(x)\end{array}\right)+\frac{1}{\lambda}N(\lambda,x)\right)U:=\left(\Lambda(\lambda,x)+\frac{1}{\lambda}N(\lambda,x)\right)U
\ee
where 
\ba 
\label{originRdef}
&-A^{-1}=R\left(\begin{array}{rr}\mu_1&0\\0&\mu_2\end{array}\right)R^{-1},\quad M=R^{-1}(A^{-1}E-A^{-1}A_x)R-R^{-1}R_x,\\
&\mu_{1,2}=\frac{FH(\sqrt{H_R}+1)}{FH_R\pm H^{\frac{3}{2}}(\sqrt{H_R}+1)},\quad \tilde{R}=Id+\left(\begin{array}{rr}-\frac{M_{12}M_{21}}{(\mu_1-\mu_2)^2\lambda^2}&\frac{M_{12}}{\mu_1-\mu_2)\lambda}\\-\frac{M_{21}}{(\mu_1-\mu_2)\lambda}&0\end{array}\right),\\
&R=\left(\begin{array}{cc} -FH & FH\\ H^{3/2}-FH\sqrt{H_R}-\frac{F(H-H_R)}{\sqrt{H_R}+1} & H^{3/2}+FH\sqrt{H_R}+\frac{F(H-H_R)}{\sqrt{H_R}+1} \end{array}\right),
\ea
and $|N(\lambda,x)|< C(F,H_R)$ uniformly in $|\lambda|>1,x\gtrless0$.

\begin{lemma}\label{cone}
Let $U(x)=\left(\begin{array}{rr}U_1(x)&U_2(x)\end{array}\right)^T$ be stable flow (as $x\rightarrow-\infty$) of \eqref{diageigen} and define $\Phi_1=U_2/U_1$. Let $\tilde{U}(x)=\left(\begin{array}{rr}\tilde{U}_1(x)&\tilde{U}_2(x)\end{array}\right)^T$ be unstable flow (as $x\rightarrow-\infty$) of \eqref{diageigen} and define $\Phi_2=\tilde{U}_1/\tilde{U}_2$.  For $\Re \lambda>-\bar\eta$ $(\bar\eta>0,\;\bar\eta$ sufficiently small$)$ and  $|\lambda|$ sufficiently large, we then have $\Phi_{1,2}(x,\lambda)=O(1/|\lambda|)$ uniformly in $x<0$.
\end{lemma}

\begin{proof}
We find after a brief calculation that 
\be
\label{Phi1eq}
\Phi_1'=(\Lambda_{22}-\Lambda_{11})\Phi_1+\frac{1}{\lambda}\left(N_{21}+N_{22}\Phi_1-N_{11}\Phi_1-N_{12}\Phi_1^2\right).
\ee
For $x<0$, let $\tilde{\mathcal{F}}_\lambda$ denotes the flow of equation $\Phi'=(\Lambda_{22}(\lambda)-\Lambda_{11}(\lambda))\Phi$. For $\Re\lambda\ge-\bar{\eta}$, 
\ba
\Re(\Lambda_{22}(\lambda)-\Lambda_{11}(\lambda))=&\left(\mu_2-\mu_1\right)\Re\lambda+M_{22}-M_{11}\\
\le&\frac{2FH^{5/2}{\left(\sqrt{H_R}+1\right)}^2}{H^3H_R+H^3-F^2{H_R}^2+2H^3\sqrt{H_R}}\bar\eta+M_{22}-M_{11}\\
\le&\frac{M_{22}-M_{11}}{2}<-c<0.
\ea
Define bounded operator $\mathcal{T}_\lambda$ on Banach space $\mathcal{B}=C^b((-\infty,0],\;\mathbb{C})$\footnote{Here $C^b((-\infty,0],\;\mathbb{C})$ is the space of bounded continuous function on $(-\infty,0]$ associated with the sup norm.} by
\ba 
\label{operatorPhi1}
(\mathcal{T}_\lambda\Phi)(x):=\int_{-\infty}^x\tilde{\mathcal{F}}_\lambda^{y\rightarrow x}\frac{1}{\lambda}\left(N_{21}(y)+N_{22}(y)\Phi(y)-N_{11}(y)\Phi(y)-N_{12}(y)\Phi^2(y)\right)dy.
\ea
Claim one: For $L>0$, the operator $\mathcal{T}_\lambda$ is a contraction map on $\{\Phi:||\Phi||_\infty\le L,\Phi\in\mathcal{B}\}$ provided that $|\lambda|\ge\max\{\frac{C(1+L)^2}{cL},\frac{4C(1+L)}{c}\}$. This follows from inequalities
\ba
&|(\mathcal{T}_\lambda\Phi)(x)|\le\int_{-\infty}^xe^{-c(x-y)}\frac{C+2CL+CL^2}{|\lambda|}dy=\frac{C(1+L)^2}{c|\lambda|}\le L,\\
&|(\mathcal{T}_\lambda\Phi-\mathcal{T}_\lambda\tilde{\Phi})(x)|\le\int_{-\infty}^xe^{-c(x-y)}\frac{2C(1+L)||\Phi-\tilde{\Phi}||_\infty}{|\lambda|}dy\le\frac{1}{2}||\Phi-\tilde{\Phi}||.
\ea
Claim two: For $|\lambda|>\frac{8C}{c}$, $\mathcal{T}_\lambda$ is a contraction map for $L:=\frac{4C}{c|\lambda|}<\frac{1}{2}$. This is because
$$
\frac{C(1+L)^2}{cL}\le\frac{C2^2}{c\frac{4C}{c|\lambda|}}=|\lambda|,\quad \frac{4C(1+L)}{c}<\frac{8C}{c}<|\lambda|.
$$
Claim two then follows from Claim one.

The unique solution to \eqref{Phi1eq} guaranteed by the contraction mapping theorem will be in the ball of radius $L=\frac{4C}{c|\lambda|}$, which is of $O(1/|\lambda|)$.
On the other hand 
\be
\label{Phi2eq}
\Phi_2'=(\Lambda_{11}-\Lambda_{22})\Phi_2+\frac{1}{\lambda}\left(N_{12}+N_{11}\Phi_2-N_{22}\Phi_2-N_{21}\Phi_2^2\right).
\ee
Let $\tilde{\mathcal{F}}_\lambda$ now denote the flow of equation $\Phi'=(\Lambda_{11}(\lambda)-\Lambda_{22}(\lambda))\Phi$. For $\Re\lambda\ge-\eta$, $\Re(\Lambda_{11}(\lambda)-\Lambda_{22}(\lambda))>c>0$,
we define bounded operator $\mathcal{T}_\lambda$ on Banach space $C^b((-\infty,0],\;\mathbb{C})$
\ba 
\label{operatorPhi2}
(\mathcal{T}_\lambda\Phi)(x):=\int_{0}^x\tilde{\mathcal{F}}_\lambda^{y\rightarrow x}\frac{1}{\lambda}\left(N_{12}(y)+N_{11}(y)\Phi(y)-N_{22}(y)\Phi(y)-N_{21}(y)\Phi^2(y)\right)dy.
\ea
Again inequalities
\ba
&|(\mathcal{T}_\lambda\Phi)(x)|\le\int_{x}^0e^{c(x-y)}\frac{C+2CL+CL^2}{|\lambda|}dy=\frac{C(1+L)^2}{c|\lambda|}\le L,\\
&|(\mathcal{T}_\lambda\Phi-\mathcal{T}_\lambda\tilde{\Phi})(x)|\le\int_{x}^0e^{c(x-y)}\frac{2C(1+L)||\Phi-\tilde{\Phi}||_\infty}{|\lambda|}dy\le\frac{1}{2}||\Phi-\tilde{\Phi}||
\ea
yield claims one and two, completing the lemma.
\end{proof}

\begin{lemma}\label{orihflemma}
Writing $R$ in \eqref{originRdef} as $R=\left(\begin{array}{rr}R_1&R_2\end{array}\right)$ and setting $\left(\begin{array}{c}L_1\\L_2\end{array}\right)=\left(\begin{array}{rr}R_1&R_2\end{array}\right)^{-1}$, for $\Re\lambda>-\bar\eta$ $(\bar\eta>0,\;\bar\eta$ sufficiently small$)$, $|\lambda|$ sufficiently large, the solution operator $\mathcal{F}_\lambda^{y\rightarrow x}$ of system \eqref{originalsystem} on $x>0$ is
\ba
\label{solutionoperator}
\mathcal{F}_\lambda^{y\rightarrow x}=&e^{\left(\lambda\mu_{1,+}+M_{11,+}+\frac{1}{\lambda}N_{11,+}(\lambda)+\frac{1}{\lambda}N_{12,+}(\lambda)\right)(x-y)}\left(R_{1,+}L_{1,+}+O(1/|\lambda|)\right)\\
+&e^{\left(\lambda\mu_{2,+}+M_{22,+}+\frac{1}{\lambda}N_{22,+}(\lambda)+\frac{1}{\lambda}N_{21,+}(\lambda)\right)(x-y)}\left(R_{2,+}L_{2,+}+O(1/|\lambda|)\right),&0<x<y.
\ea
Moreover, the stable and unstable flow operators of system (\eqref{originalsystem}) on $x<0$ are
\ba
\label{stableunstablefloww}
\mathcal{F}_\lambda^{y\rightarrow x}\Pi_{\lambda,s}(y)=&e^{\int_y^x\left(\lambda\mu_1(z)+M_{11}(z)+\frac{1}{\lambda}N_{11}(\lambda,z)+\frac{1}{\lambda}N_{12}(\lambda,z)\right)dz}\left(R_1(x)L_1(y)+O(1/|\lambda|)\right),&x<y<0,\\
\mathcal{F}_\lambda^{y\rightarrow x}\Pi_{\lambda,u}(y)=&e^{\int_y^x\left(\lambda\mu_2(z)+M_{22}(z)+\frac{1}{\lambda}N_{22}(\lambda,z)+\frac{1}{\lambda}N_{21}(\lambda,z)\right)dz}\left(R_2(x)L_2(y)+O(1/|\lambda|)\right),&y<x<0,\\
\ea
where $\mu_{1,2}$, $M$ as in \eqref{originRdef} are independent of $\lambda$, $|N(\lambda,x)|< C(F,H_R)$ uniformly in $|\lambda|>1,x\gtrless0$, and the bound $O(1/|\lambda|)$ is independent of $x$, $y$. 
\begin{proof}
By lemma \ref{cone}, the stable flow of \eqref{diageigen} may be written as $U=\left(\begin{array}{rr}U_1&\Phi_1 U_1\end{array}\right)^T$ with $\Phi_1=O(1/|\lambda|)$. The equation for $U_1$ then reads $U_1'=\left(\Lambda_{11}+\frac{1}{\lambda}N_{11}+\frac{1}{\lambda}N_{12}\Phi_1\right)U_1$. Integrating from $0$ to $x$ yields solution
\be
U_1(\lambda,x)=e^{\int_0^x\left(\Lambda_{11}(\lambda,y)+\frac{1}{\lambda}N_{11}(\lambda,y)+\frac{1}{\lambda}N_{12}(\lambda,y)\Phi_1(\lambda,y)\right)dy}U_1(0).
\ee
Hence the full solution to \eqref{diageigen} is
\be
\left(\begin{array}{c}U_1\\U_2\end{array}\right)=\left(\begin{array}{c}U_1\\\Phi_1U_1\end{array}\right)=e^{\int_0^x\left(\Lambda_{11}(\lambda,y)+\frac{1}{\lambda}N_{11}(\lambda,y)+\frac{1}{\lambda}N_{12}(\lambda,y)\Phi_1(\lambda,y)\right)dy}\left(\begin{array}{c}1\\\Phi_1(\lambda,x)\end{array}\right)U_1(0).
\ee
Transforming back to $w$ coordinates by $w=R\tilde{R}^{-1}U$ and using estimate $\Phi_1(\lambda,x)=O(1/|\lambda|)$, $\tilde{R}^{-1}=Id+O(1/|\lambda|)$, 
we obtain 
\be
\label{stablesolution}
w(\lambda,x)=e^{\int_0^x\left(\Lambda_{11}(\lambda,y)+\frac{1}{\lambda}N_{11}(\lambda,y)+\frac{1}{\lambda}N_{12}(\lambda,y)\Phi_1(\lambda,y)\right)dy}\left(R_1(x)+O(1/|\lambda|)\right).
\ee
The projection onto the stable manifold $\Pi_{\lambda,s}(y)$ is approximately $R_1(y)L_1(y)+O(1/|\lambda|)$; following the flow from $y$ to $x$ 
thus yields \eqref{stableunstablefloww}(i). The unstable flow operator \eqref{stableunstablefloww}(ii) can be derived similarly.
\end{proof}
\end{lemma}

\subsection{Pointwise estimates on resolvent kernels}

\subsubsection{Large $\lambda\sim$ small time}
\begin{proposition}\label{highestimates}
For $\Re\lambda>-\bar\eta$ $(\bar\eta>0,\;\bar\eta$ sufficiently small$)$ and $|\lambda|$ sufficiently large, $G^1_\lambda$, $G^2_\lambda$, and $K_\lambda$ can be written as 
\small
\ba
\label{highresolvent1}
&G^1_\lambda=\\&\left\{\begin{aligned}&-e^{\left(\lambda\mu_{1,+}+M_{11,+}+\frac{1}{\lambda}N_{11,+}(\lambda)+\frac{1}{\lambda}N_{12,+}(\lambda)\Phi_{1,+}(\lambda)\right)(x-y)}\left(R_{1,+}L_{1,+}+O(1/|\lambda|)\right)A^{-1}_+\\
&-e^{\left(\lambda\mu_{2,+}+M_{22,+}+\frac{1}{\lambda}N_{22,+}(\lambda)+\frac{1}{\lambda}N_{21,+}(\lambda)\Phi_{2,+}(\lambda)\right)(x-y)}\left(R_{2,+}L_{2,+}+O(1/|\lambda|)\right)A^{-1}_+,&0<x<y,\\&-e^{\int_y^x\left(\lambda\mu_1(z)+M_{11}(z)+\frac{1}{\lambda}N_{11}(\lambda,z)+\frac{1}{\lambda}N_{12}(\lambda,z)\Phi_1(\lambda,z)\right)dz}\left(R_1(x)L_1(y)+O(1/|\lambda|)\right)A^{-1}(y),&x<y<0,\\
&e^{\int_y^{x}\left(\lambda\mu_2(z)+M_{22}(z)+\frac{1}{\lambda}N_{22}(\lambda,z)+\frac{1}{\lambda}N_{21}(\lambda,z)\Phi_2(\lambda,z)\right)dz}\left(R_2(x)L_2(y)+O(1/|\lambda|)\right)A^{-1}(y),&y<x<0,
\end{aligned}\right.
\ea
\ba
\label{highresolvent2}
&G^2_\lambda=\\
&\left\{\begin{aligned}&-\Bigg(e^{\int_0^x\left(\lambda\mu_1(z)+M_{11}(z)+\frac{1}{\lambda}N_{11}(\lambda,z)+\frac{1}{\lambda}N_{12}(\lambda,z)\Phi_1(\lambda,z)\right)dz}\left(R_1(x)L_1(0^-)
+O(1/|\lambda|)\right)V+\\&\overline{W}'(x)O(1/|\lambda|)\Bigg)A_+\Bigg(e^{-\left(\lambda\mu_{1,+}+M_{11,+}+\frac{1}{\lambda}N_{11,+}(\lambda)+\frac{1}{\lambda}N_{12,+}(\lambda)\Phi_{1,+}(\lambda)\right)y}\left(R_{1,+}L_{1,+}+O(1/|\lambda|)\right)\\
&+e^{-\left(\lambda\mu_{2,+}+M_{22,+}+\frac{1}{\lambda}N_{22,+}(\lambda)+\frac{1}{\lambda}N_{21,+}(\lambda)\Phi_{2,+}(\lambda)\right)y}\left(R_{2,+}L_{2,+}+O(1/|\lambda|)\right)\Bigg)A^{-1}_+,&x<0,y>0,\\
&-\Bigg(e^{\int_0^x\left(\lambda\mu_1(z)+M_{11}(z)+\frac{1}{\lambda}N_{11}(\lambda,z)+\frac{1}{\lambda}N_{12}(\lambda,z)\Phi_1(\lambda,z)\right)dz}\left(R_1(x)L_1(0^-)
+O(1/|\lambda|)\right)V+\\&\overline{W}'(x)O(1/|\lambda|)\Bigg)A(0^-)\times\\
&e^{\int_y^0\left(\lambda\mu_2(z)+M_{22}(z)+\frac{1}{\lambda}N_{22}(\lambda,z)+\frac{1}{\lambda}N_{21}(\lambda,z)\Phi_2(\lambda,z)\right)dz}\left(R_2(0^-)L_2(y)+O(1/|\lambda|)\right)A^{-1}(y),&x<0,y<0,\end{aligned}
\right.
\ea
\ba
&K_\lambda=\\&\left\{\begin{aligned}&0,&x>0,\\&-e^{\int_0^x\left(\lambda\mu_1(z)+M_{11}(z)+\frac{1}{\lambda}N_{11}(\lambda,z)+\frac{1}{\lambda}N_{12}(\lambda,z)\Phi_1(\lambda,z)\right)dz}\left(R_1(x)L_1(0^-)
+O(1/|\lambda|)\right)V+\overline{W}'(x)O(1/|\lambda|),&x<0,\end{aligned}\right.
\ea
\normalsize
where $\mu_{1,2}$, $R$, $L$, $M_{11}$, $M_{22}$ as in \eqref{originRdef}, $V$ as in \eqref{defV} are explicitly calculable and independent of $\lambda$, $\Phi_{1,2}(\lambda,x)$ as in Lemma \ref{cone} are $O(1/|\lambda|)$ terms uniformly in $x$, $y$. Moreover, they can be decomposed as
\be
G^1_\lambda=H^1_\lambda+(G^1_\lambda-H^1_\lambda),\quad G^2_\lambda=H^2_\lambda+(G^2_\lambda-H^2_\lambda),\quad K_\lambda=H_{K,\lambda}+(K_\lambda-H_{K,\lambda})
\ee
where $H^{1,2}_\lambda$, $H_{K,\lambda}$ are their corresponding lowest order terms defined by
\ba
\label{highresolventlow}
H^1_\lambda=&\left\{\begin{aligned}&-e^{\left(\lambda\mu_{1,+}+M_{11,+}\right)(x-y)}R_{1,+}L_{1,+}A^{-1}_+\\
&-e^{\left(\lambda\mu_{2,+}+M_{22,+}\right)(x-y)}R_{2,+}L_{2,+}A^{-1}_+,&0<x<y,\\
&-e^{\int_y^x\left(\lambda\mu_1(z)+M_{11}(z)\right)dz}R_1(x)L_1(y)A^{-1}(y),&x<y<0,\\
&e^{\int_y^{x}\left(\lambda\mu_2(z)+M_{22}(z)\right)dz}R_2(x)L_2(y)A^{-1}(y),&y<x<0,
\end{aligned}\right.\\
H^2_\lambda=&\left\{\begin{aligned}&-e^{\int_0^x\left(\lambda\mu_1(z)+M_{11}(z)\right)dz}R_1(x)L_1(0^-)
VA_+\\
&\times\left(e^{-(\lambda\mu_{1,+}+M_{11,+})y}R_{1,+}L_{1,+}+e^{-(\lambda\mu_{2,+}+M_{22,+})y}R_{2,+}L_{2,+}\right)A_+^{-1},&x<0,y>0,\\
&-e^{\int_0^x\left(\lambda\mu_1(z)+M_{11}(z)\right)dz}R_1(x)L_1(0^-)
VA(0^-)\\
&\times e^{\int_y^0\left(\lambda\mu_2(z)+M_{22}(z)\right)dz}R_2(0^-)L_2(y)A^{-1}(y),&x<0,y<0,\end{aligned}
\right.\\
H_\lambda=&\left\{\begin{aligned}&0,&x>0,\\&-e^{\int_0^x\left(\lambda\mu_1(z)+M_{11}(z)\right)dz}R_1(x)L_1(0^-)
V,&x<0,\end{aligned}\right.
\ea
and $G^{1,2}_\lambda-H^{1,2}_\lambda$, $K_\lambda-H_\lambda$ are $O(1/|\lambda|)$ terms.
\normalsize

\end{proposition}
\begin{proof}
As consequences of Lemma \ref{orihflemma} and using either \eqref{resolventkernel1} or \eqref{resolventkernel2}, the $G_\lambda^1$ part \eqref{highresolvent1} then follows. As for the $G_\lambda^2$ part, explicit calculation shows that in the high frequency regime
\ba
\label{defV}
\left(\begin{array}{rr}1 &0\end{array}\right)M^{-1}(\lambda)=&O(1/|\lambda|)\\
\left(\begin{array}{rr}0 &1\end{array}\right)M^{-1}(\lambda)=&\left(\begin{array}{cc} \frac{FH_*\left(H_R+\sqrt{H_R}+1\right)}{{\left({H_*}^{3/2}+\sqrt{H_R}{H_*}^{3/2}+FH_R\right)}^2} & -\frac{FH_*\left(\sqrt{H_R}+1\right)}{{\left({H_*}^{3/2}+\sqrt{H_R}{H_*}^{3/2}+FH_R\right)}^2} \end{array}\right)+O(1/|\lambda|)\\
:=&V+O(1/|\lambda|).
\ea
Equation \eqref{highresolvent2} then follows. To estimate the error terms, we find for $x<y<0$
\ba 
\label{taylorestimation}
&\left|e^{\int_y^x\left(\lambda\mu_1(z)+M_{11}(z)+\frac{1}{\lambda}N_{11}(\lambda,z)+\frac{1}{\lambda}N_{12}(\lambda,z)\Phi_1(\lambda,z)\right)dz}-e^{\int_y^x\left(\lambda\mu_1(z)+M_{11}(z)\right)dz}\right|\\
=&\left|\sum_{n=1}^\infty e^{\int_y^x\left(\lambda\mu_1(z)+M_{11}(z)\right)dz}\frac{1}{n!}\frac{\left(\int_y^x\left(N_{11}(\lambda,z)+N_{12}(\lambda,z)\Phi_1(\lambda,z)\right)dz\right)^n}{\lambda^n}\right|\\
\le&\sum_{n=1}^\infty e^{c(x-y)}\frac{(C(x-y))^n}{n!|\lambda|^n}=\sum_{n=1}^\infty \frac{e^{c(x-y)}C(x-y)}{n|\lambda|}\frac{(C(x-y))^{n-1}}{(n-1)!|\lambda|^{n-1}}\\
\lesssim&\sum_{n=1}^\infty \frac{e^{c(x-y)/2}}{|\lambda|}\frac{(C(x-y))^{n-1}}{(n-1)!|\lambda|^{n-1}}=\frac{1}{|\lambda|}e^{\frac{c(x-y)}{2}+\frac{C(x-y)}{|\lambda|}}=O(1/|\lambda|).
\ea
Thus $G^1_\lambda-H^1_\lambda$ is an $O(1/|\lambda|)$ term on $y<x<0$. The other parts can be similarly estimated.
\end{proof}

Desingularizing $\check{\eta}$-resolvent kernels $G_{1,\lambda}$, $K_{1,\lambda}$ by multiplying by a factor $\lambda$, 
we have the following estimates on $\lambda G_{1,\lambda}$, $\lambda K_{1,\lambda}$ in the high frequency regime.

\begin{proposition} \label{highestimates1}
For $\Re\lambda>-\bar\eta$ $(\bar\eta>0,\;\bar\eta$ sufficiently small$)$ and $|\lambda|$ sufficiently large, $\lambda G_{1,\lambda}$, $\lambda K_{1,\lambda}$ can be written as 
{\small
\ba
\label{lambdaG1lambda}
&\lambda G_{1,\lambda}=\\
&\left\{\begin{aligned}&e^{-\left(\lambda\mu_{1,+}+M_{11,+}+\frac{1}{\lambda}N_{11,+}(\lambda)+\frac{1}{\lambda}N_{12,+}(\lambda)\Phi_{1,+}(\lambda)\right)y}\left(V_h+O(1/|\lambda|)\right)A_+\left(R_{1,+}L_{1,+}+O(1/|\lambda|)\right)A^{-1}_+\\
&e^{-\left(\lambda\mu_{2,+}+M_{22,+}+\frac{1}{\lambda}N_{22,+}(\lambda)+\frac{1}{\lambda}N_{21,+}(\lambda)\Phi_{2,+}(\lambda)\right)y}\left(V_h+O(1/|\lambda|)\right)A_+\left(R_{2,+}L_{2,+}+O(1/|\lambda|)\right)A^{-1}_+,&0<y,\\
&e^{\int_y^{0^-}\left(\lambda\mu_2(z)+M_{22}(z)+\frac{1}{\lambda}N_{22}(\lambda,z)+\frac{1}{\lambda}N_{21}(\lambda,z)\Phi_2(\lambda,z)\right)dz}\left(V_h+O(1/|\lambda|)\right)\times\\
&A(0^-)\left(R_2(0^-)L_2(y)+O(1/|\lambda|)\right)A^{-1}(y),&y<0,\end{aligned}\right.
\ea}
\be
\label{lambdaK1lambda}
\lambda K_{1,\lambda}=V_h+O(1/|\lambda|)
\ee
\normalsize
where $\mu_{1,2}$, $R$, $L$, $M_{11}$, $M_{22}$ as in \eqref{originRdef} $V_h$ as in \eqref{defVh} are explicitly calculable and independent of $\lambda$, $\Phi_{1,2}(\lambda,x)$ as in Lemma \ref{cone} are $O(1/|\lambda|)$ terms uniformly in $x$, $y$. Moreover, $\lambda G_{1,\lambda}$ can be decomposed as
\be
\lambda G_{1,\lambda}=H_{1,\lambda}+(G_{1,\lambda}-H_{1,\lambda}),
\ee
where $H_{1,\lambda}$ is its corresponding lowest order term defined by
\ba
\label{highresolventlow1}
H_{1,\lambda}=&\left\{\begin{aligned}&e^{-\left(\lambda\mu_{1,+}+M_{11,+}\right)y}V_hA_+R_{1,+}L_{1,+}A^{-1}_+\\
&e^{-\left(\lambda\mu_{2,+}+M_{22,+}\right)y}V_hA_+R_{2,+}L_{2,+}A^{-1}_+,&0<x<y,\\
&e^{\int_y^{0^-}\left(\lambda\mu_2(z)+M_{22}(z)\right)dz}V_hA(0^-)R_2(0^-)L_2(y)A^{-1}(y),&y<x<0,
\end{aligned}\right.\\
\ea
and $\lambda G_{1,\lambda}-H_{1,\lambda}$ is an $O(1/|\lambda|)$ term.
\end{proposition}
\normalsize
\begin{proof}
By definition of $K_{1,\lambda}$ \eqref{K1lambda} and equations \eqref{secondcolumn} \eqref{hfexpansion}, equation \eqref{lambdaK1lambda} follows from the calculation
\ba
\label{defVh}
\lambda K_{1,\lambda}=&\left(\begin{array}{rr}1&0\end{array}\right)\lambda M^{-1}(\lambda)\\
=&\frac{\lambda}{\Delta(\lambda)}\left(\begin{array}{rr}1&0\end{array}\right)\left(\begin{array}{rr}-\frac{1}{\mu_1(0^-)}(R_{1}(0^-))_2+O(1/|\lambda|)&\frac{1}{\mu_1(0^-)}(R_{1}(0^-))_1+O(1/|\lambda|)\\-\left(\lambda \overline{W}-R(\overline{W})\right)_2&\left(\lambda \overline{W}-R(\overline{W})\right)_1\end{array}\right)\\
=&(H_*-H_R)\left(FH_R+{H_*}^{3/2}(1+\sqrt{H_R})\right)\times\\
 &\left(\begin{array}{cc} -H_*^{3/2}(\sqrt{H_R}+1)+F(H_*-H_R+H_*H_R+H_*\sqrt{H_R})) & -FH_*\left(\sqrt{H_R}+1\right) \end{array}\right)+O(1/|\lambda|)\\
:=&V_h+O(1/|\lambda|).
\ea
By definition of $G_{1,\lambda}$ \eqref{G1lambda} and using Lemma \ref{orihflemma}, in the high frequency regime, 
\ba
\label{highG1lambda}
&G_{1,\lambda}=\\
&\left\{\begin{aligned}&\left(\begin{array}{rr}1 &0\end{array}\right)M^{-1}(\lambda)A_+\Bigg(e^{-\left(\lambda\mu_{1,+}+M_{11,+}+\frac{1}{\lambda}N_{11,+}(\lambda)+\frac{1}{\lambda}N_{12,+}(\lambda)\Phi_{1,+}(\lambda)\right)y}\left(R_{1,+}L_{1,+}+O(1/|\lambda|)\right)\\
&+e^{-\left(\lambda\mu_{2,+}+M_{22,+}+\frac{1}{\lambda}N_{22,+}(\lambda)+\frac{1}{\lambda}N_{21,+}(\lambda)\Phi_{2,+}(\lambda)\right)y}\left(R_{2,+}L_{2,+}+O(1/|\lambda|)\right)\Bigg)A^{-1}_+,\quad y>0,\\
&\left(\begin{array}{rr}1 &0\end{array}\right)M^{-1}(\lambda)A(0^-)e^{\int_y^0\left(\lambda\mu_2(z)+M_{22}(z)+\frac{1}{\lambda}N_{22}(\lambda,z)+\frac{1}{\lambda}N_{21}(\lambda,z)\Phi_2(\lambda,z)\right)dz}\\
&\left(R_2(0^-)L_2(y)+O(1/|\lambda|)\right)A^{-1}(y),\quad y<0,\end{aligned}
\right.
\ea
By equation \eqref{defVh}, the $\lambda G_{1,\lambda}$ part \eqref{lambdaG1lambda} then follows. 
Following similar calculation as in \eqref{taylorestimation}, 
we find that $\lambda G_{1,\lambda}-H_{1,\lambda}$ is an $O(1/|\lambda|)$ term.
\end{proof}
\begin{proposition}
\label{highG1y}
For $\Re\lambda>-\bar\eta$ $(\bar\eta>0,\;\bar\eta$ sufficiently small$)$ and $|\lambda|$ sufficiently large, the $y$ derivative of $G_{1,\lambda}$ can be decomposed as 

\be
\partial_y G_{1,\lambda}=HY_{\lambda}+(\partial_y G_{1,\lambda}-HY_{\lambda}),
\ee
where $HY_{\lambda}$ is its corresponding lowest order term defined by
\ba
\label{HYterm}
HY_{\lambda}=&\left\{\begin{aligned}&-\mu_{1,+}e^{-\left(\lambda\mu_{1,+}+M_{11,+}\right)y}V_hA_+R_{1,+}L_{1,+}A^{-1}_+\\
&-\mu_{2,+}e^{-\left(\lambda\mu_{2,+}+M_{22,+}\right)y}V_hA_+R_{2,+}L_{2,+}A^{-1}_+,&0<x<y,\\
&-\mu_2(y)e^{\int_y^{0^-}\left(\lambda\mu_2(z)+M_{22}(z)\right)dz}V_hA(0^-)R_2(0^-)L_2(y)A^{-1}(y),&y<x<0,
\end{aligned}\right.\\
\ea
and $\partial_y G_{1,\lambda}-HY_{\lambda}$ is a $O(1/|\lambda|)$ term.
\end{proposition}
\begin{proof}
By taking $y$ derivative of $G_{1,\lambda}$ using \eqref{highG1lambda}, we see when the $y$-derivative falls on the exponential terms, the exponent $-\lambda \mu_{1,+}$ then gives another factor of $\lambda$ that cancel the $\frac{1}{\lambda}$ factor in $\left(\begin{array}{rr}1 &0\end{array}\right)M^{-1}(\lambda)$ \eqref{defV}, giving $HY_{\lambda}$ term. When the $y$-derivative falls on other terms, it results in terms of order $O(\frac{1}{|\lambda|})$.
\end{proof}
\subsubsection{Small $\lambda\sim$ large time}
Expanding \eqref{eigenm}\eqref{eigenp} near $\lambda=0$ yields
\ba 
\label{lowgamma12ex}
\gamma_{1,-}(\lambda)=&c^0_{1,-}+c^1_{1,-}\lambda+O(\lambda^2),&
\gamma_{2,-}(\lambda)=&-c^1_{2,-}\lambda+c^2_{2,-}\lambda^2+O(\lambda^3):=\tilde{\gamma}_{2,-}(\lambda)+O(\lambda^3),\\
\gamma_{2,+}(\lambda)=&c^0_{2,+}+c^1_{2,+}\lambda+O(\lambda^2),&\gamma_{1,+}(\lambda)=&c^1_{1,+}\lambda-c^2_{1,+}\lambda^2+O(\lambda^3):=\tilde{\gamma}_{1,+}(\lambda)+O(\lambda^3),
\ea
where $c^i_{1,2,\pm}$ are positive constant explicitly calculable as functions of $F$, $H_R$.
Since $\mathcal{A}_+(\lambda)$ ($\mathcal{A}_-(\lambda)$) has distinct eigenvalues $\gamma_{1,2,+}$ ($\gamma_{1,2,-}$) near $\lambda=0$, we have the following proposition.

\begin{proposition}\label{extendtildeG}The resolvent kernels $G^1_\lambda(x;y)$, $G^2_\lambda(x;y)$, and $K_\lambda(x)$ can be extended holomorphically to $B(0,r)$ for sufficiently small $r$. Moreover $G^1_\lambda(x;y)$ can be decomposed as
\be
\label{decomG1}
G^1_\lambda=S^1_\lambda+R^1_\lambda,
\ee
where 
\be
\label{scatter1lambda}
S^1_\lambda(x;y):=\left\{\begin{aligned}
&-e^{\tilde{\gamma}_{1,+}(\lambda)(x-y)}P_{1,+}(0)A_+^{-1}
,&0<x<y,\\
&e^{\tilde{\gamma}_{2,-}(\lambda)(x-y)}P_{2,-}(0)A^{-1}_-,&y<x<0,\\
&0,&\text{otherwise},
\end{aligned}
\right.
\ee
and $R^1_\lambda$ is a faster-decaying residual
\ba
&R^1_\lambda(x;y):=\\
&\left\{\begin{aligned}
&-e^{\gamma_{2,+}(\lambda)(x-y)}P_{2,+}(\lambda)A_+^{-1}+\left(e^{\tilde{\gamma}_{1,+}(\lambda)(x-y)}P_{1,+}(0)-e^{\gamma_{1,+}(\lambda)(x-y)}P_{1,+}(\lambda)\right)A_+^{-1},
&0<x<y,\\
&-e^{\gamma_{1,-}(\lambda)(x-y)}T_-(\lambda,x)P_{1,-}(\lambda)T_-^{-1}(\lambda,y)A^{-1}(y),&x<y<0,\\
&-e^{\tilde{\gamma}_{2,-}(\lambda)(x-y)}P_{2,-}(0)A_-^{-1}+e^{\gamma_{2,-}(\lambda)(x-y)}T_-(\lambda,x)P_{2,-}(\lambda)T_-^{-1}(\lambda,y)A^{-1}(y),&y<x<0,\\
&0,&\text{otherwise}.
\end{aligned}
\right.
\ea
Further, $G^2_\lambda(x;y)$ is a faster decaying term which can be estimated as 
\ba 
\label{G2low}
|G_\lambda^2|=&\left|O(e^{-\tilde{r}_{1,+}y-\theta'|x|})\right|,\quad x<0,y>0,\\
|G_\lambda^2|=&\left|O(e^{-\tilde{r}_{2,-}y-\theta'|x|})\right|,\quad x<0,y<0,
\ea 
and $K_\lambda(x)$ is a faster decaying term which can be estimated as 
\be
\label{Klow}
|K_\lambda|=\left|O(e^{-\theta'|x|})\right|,\quad x<0.
\ee
\end{proposition}

\begin{proposition}\label{extendlambdaGK}The desingularized resolvent kernels $\lambda G_{1,\lambda}(y)$, $\lambda K_{1,\lambda}$ can be extended holomorphically to $B(0,r)$ for sufficiently small $r$. Moreover, defining 
\be 
\label{defVl}
V_l=\lim_{\lambda\rightarrow 0}\lambda\left(\begin{array}{rr}1 &0\end{array}\right)M^{-1}(\lambda),
\ee
there holds
\be 
\label{lambdaK1lambdalow}
\lambda K_{1,\lambda}=V_l+O(|\lambda|),
\ee 
and $\lambda G_{1,\lambda}(y)$ can be decomposed as
\be
\label{decomlambdaG1lambda}
\lambda G_{1,\lambda}=S_{1,\lambda}+R_{1,\lambda},
\ee
where 
\be
\label{scatter1lambda1}
S_{1,\lambda}(y):=\left\{\begin{aligned}
&e^{-\tilde{\gamma}_{1,+}(\lambda)y}V_lA_+P_{1,+}(0)A_+^{-1}
,&0<y,\\
&e^{-\tilde{\gamma}_{2,-}(\lambda)y}V_lA(0^-)P_{2,-}(0)A^{-1}_-,&y<0,
\end{aligned}
\right.
\ee
and $R_{1,\lambda}$ is a faster-decaying residual.
\end{proposition}
\begin{proof}
By the definition of $G_{1,\lambda}$ \eqref{G1lambda}
\be 
\label{G1lambdalow}
G_{1,\lambda}=\left\{\begin{aligned}&\left(\begin{array}{rr}1 &0\end{array}\right)M^{-1}(\lambda)A_+\left(e^{-\gamma_{1,+}(\lambda)y}P_{1,+}(\lambda)+e^{-\gamma_{2,+}(\lambda)y}P_{2,+}(\lambda)\right)A_+^{-1},&y>0,\\
								&\left(\begin{array}{rr}1 &0\end{array}\right)M^{-1}(\lambda)A(0^-)T_-(\lambda,0^-)e^{-\gamma_{2,-}y}P_{2,-}(\lambda)T_-^{-1}(\lambda,y)A^{-1}(y),&y<0.
\end{aligned}\right.
\ee 
We see in the neighborhood of the origin $\left(\begin{array}{rr}1 &0\end{array}\right)M^{-1}(\lambda)$ is desingularized by the extra $\lambda$ factor in $\lambda G_{1,\lambda}$. The proposition then follows.
\end{proof}
\begin{proposition}
\label{lowG1y}
The $y$-derivative of $G_{1,\lambda}$ can be decomposed as
\be
\label{decomlambdaG1ylambda}
\partial_y G_{1,\lambda}=SY_{1,\lambda}+SY_{2,\lambda}+RY_{\lambda},
\ee
where 
\ba 
\label{scatter1ylambda1}
SY_{1,\lambda}:=&\left\{\begin{aligned}
&-c_{1,+}^1e^{-\tilde{\gamma}_{1,+}(\lambda)y}V_lA_+P_{1,+}(0)A_+^{-1}
,&0<y,\\
&-c_{2,-}^1e^{-\tilde{\gamma}_{2,-}(\lambda)y}V_lA(0^-)P_{2,-}(0)A^{-1}_-,&y<0,
\end{aligned}
\right.\\
SY_{2,\lambda}:=&
\left\{\begin{aligned}&-\left(\begin{array}{rr}1 &0\end{array}\right)M^{-1}(\lambda)A_+\gamma_{2,+}e^{-\gamma_{2,+}(\lambda)y}P_{2,+}(\lambda)A_+^{-1},&y>0,\\
				    &\left(\begin{array}{rr}1 &0\end{array}\right)M^{-1}(\lambda)A(0^-)T_-(\lambda,0^-)e^{-\gamma_{2,-}y}P_{2,-}(\lambda)\frac{T_-^{-1}(\lambda,y)A^{-1}(y)}{\partial y},&y<0,
\end{aligned}\right.
\ea 
and $RY_{\lambda}$ is a faster-decaying residual that is of order $O(|\lambda\; SY_{1,\lambda}|)$. The term $SY_{2,\lambda}$ has a simple pole at the origin and in $y$ it is of order $O(e^{-\theta|y|})$.
\end{proposition}

\begin{proof}
Taking $y$-derivative of \eqref{G1lambdalow} yields 
\ba 
\label{G1lambdalowy}
&\partial_yG_{1,\lambda}=\\
&\left\{\begin{aligned}&-\left(\begin{array}{rr}1 &0\end{array}\right)M^{-1}(\lambda)A_+\left(\gamma_{1,+}e^{-\gamma_{1,+}(\lambda)y}P_{1,+}(\lambda)+\gamma_{2,+}e^{-\gamma_{2,+}(\lambda)y}P_{2,+}(\lambda)\right)A_+^{-1},&y>0,\\
&-\left(\begin{array}{rr}1 &0\end{array}\right)M^{-1}(\lambda)A(0^-)T_-(\lambda,0^-)\gamma_{2,-}e^{-\gamma_{2,-}y}P_{2,-}(\lambda)T_-^{-1}(\lambda,y)A^{-1}(y)\\
&+\left(\begin{array}{rr}1 &0\end{array}\right)M^{-1}(\lambda)A(0^-)T_-(\lambda,0^-)e^{-\gamma_{2,-}y}P_{2,-}(\lambda)\frac{T_-^{-1}(\lambda,y)A^{-1}(y)}{\partial y},&y<0.
\end{aligned}\right.
\ea 
By \eqref{lowgamma12ex}, we see the terms contain factor $\gamma_{1,+}(\lambda),\gamma_{2,-}(\lambda)$ will be desingularized, giving the $SY_{1,\lambda}$ and $RY_{\lambda}$ term. The remaining terms are defined to be $SY_{2,\lambda}$.
\end{proof}
\begin{proposition}
\label{G1lambdav}
The term $G_{1,\lambda}\left(\begin{array}{c}0\\1\end{array}\right)$ can be decomposed as
\be
\label{decomlambdaG1vlambda}
 G_{1,\lambda}\left(\begin{array}{c}0\\1\end{array}\right)=SV_{1,\lambda}+SV_{2,\lambda}+RV_{\lambda},
\ee
where 
\ba
\label{scatter1vlambda1}
SV_{1,\lambda}:=&\left\{\begin{aligned}
&e^{-\tilde{\gamma}_{1,+}(\lambda)y}V_lA_+\big(\partial_\lambda P_{1,+}\big)(0)A_+^{-1}\left(\begin{array}{c}0\\1\end{array}\right)
,&0<y,\\
 &e^{-\tilde{\gamma}_{2,-}(\lambda)y}V_lA(0^-)\big(\partial_\lambda P_{2,-}\big)(0)A^{-1}_-\left(\begin{array}{c}0\\1\end{array}\right),&y<0,
\end{aligned}
\right.\\
SV_{2,\lambda}:=&
\left\{\begin{aligned}&\left(\begin{array}{rr}1 &0\end{array}\right)M^{-1}(\lambda)A_+e^{-\gamma_{2,+}(\lambda)y}P_{2,+}(\lambda)A_+^{-1}\left(\begin{array}{c}0\\1\end{array}\right),&y>0,\\
				    &\left(\begin{array}{rr}1 &0\end{array}\right)M^{-1}(\lambda)A(0^-)T_-(\lambda,0^-)\\
	       &\times e^{-\gamma_{2,-}y}P_{2,-}(\lambda)\Big(T_-^{-1}(\lambda,y)A^{-1}(y)-A^{-1}_-\Big)\left(\begin{array}{c}0\\1\end{array}\right),&y<0,
\end{aligned}\right.
\ea 
and $RV_{\lambda}$ is a faster-decaying residual that is of order $O(|\lambda\; SV_{1,\lambda}|)$. The term $SV_{2,\lambda}$ has a simple pole at the origin and in $y$ it is of order $O(e^{-\theta|y|})$.
\end{proposition}
\begin{proof}
By expansion $$P_{1,+}(\lambda)=P_{1,+}(0)+\big(\partial_\lambda P_{1,+}\big)(0)\lambda+O(|\lambda|^2),\quad P_{2,-}(\lambda)=P_{2,-}(0)+\big(\partial_\lambda P_{2,-}\big)(0)\lambda+O(|\lambda|^2),$$ and the special structure on $P_{1,+}(0)A^{-1}_+$ and $P_{2,-}(0)A^{-1}_-$ from Observation \ref{apecialstructiononS}, the proposition follows by a similar argument as for Proposition \ref{lowG1y}.
\end{proof}

\section{Pointwise estimates on Green kernels}\label{s:5}
With the above preparations, we are now ready to carry out our main linear estimates,
obtaining detailed pointwise bounds on the Green kernels of the time-evolution problem.

\begin{theorem}\label{pointwiseintvGreen}
The interior source $v$-Green kernel function $G$ defined in \eqref{Greenkernel} may be decomposed as
\be
\label{tildeGdecompose}
G=H^1+H^2+S^1+R,
\ee
where, assuming Evans-Lopatinsky stability, 
\ba 
\label{H1term}
&H^1(x,t;y):=\\&\left\{\begin{aligned}
&-e^{-\bar\eta t+\left(-\bar\eta\mu_{1,+}+M_{11,+}\right)(x-y)}R_{1,+}L_{1,+}A^{-1}_+ \delta\big(t+\mu_{1,+}(x-y)\big)\\
&-e^{-\bar\eta t+\left(-\bar\eta\mu_{2,+}+M_{22,+}\right)(x-y)}R_{2,+}L_{2,+}A^{-1}_+ \delta\big(t+\mu_{2,+}(x-y)\big),&0<x<y,\\&
-e^{-\bar\eta t+\int_y^x\left(-\bar\eta\mu_1(z)+M_{11}(z)\right)dz}R_1(x)L_1(y)A^{-1}(y) \delta\big(t+\int_y^x\mu_1(z)dz\big),&x<y<0,\\
&e^{-\bar\eta t+\int_y^x\left(-\bar\eta\mu_2(z)+M_{22}(z)\right)dz}R_2(x)L_2(y)A^{-1}(y) \delta\big(t+\int_y^x\mu_2(z)dz\big),&y<x<0,
\end{aligned}\right.
\ea
{\small
\ba
\label{H2term}
&H^2(x,t;y):=\\
&\left\{\begin{aligned}&-e^{-\bar\eta t+\int_0^x\left(-\bar\eta\mu_1(z)+M_{11}(z)\right) dz}R_1(x)L_1(0^-)VA_+\Bigg(\delta\left(t-\mu_{1,+}y+\int_0^x\mu_1(z)dz\right)R_{1,+}L_{1,+}\\
&\times e^{-(-\bar{\eta}\mu_{1,+}+M_{11,+})y}+\delta\left(t-\mu_{2,+}y+\int_0^x\mu_1(z)dz\right)R_{2,+}L_{2,+}e^{-(-\bar{\eta}\mu_{2,+}+M_{22,+})y}\Bigg)A_+^{-1},&x<0,y>0,\\
&-e^{-\bar\eta t+\int_0^x\left(-\bar\eta\mu_1(z)+M_{11}(z) \right)dz+\int_y^0\left(-\bar\eta\mu_2(z)+M_{22}(z)\right)dz}R_1(x)L_1(0^-)VA(0^-)R_2(0^-)L_2(y)\\
&\times A^{-1}(y)\delta\left(t+\int_0^x\mu_1(z)dz+\int_y^0\mu_2(z)dz\right),& x<0,y<0,
\end{aligned}
\right.
\ea
}
\be
\label{S1term}
S^1(x,t;y):=\left\{\begin{aligned}&\chi_{t\ge 1}\frac{-\sqrt{c^1_{1,+}}}{\sqrt{4c^2_{1,+}\pi t}}e^{-\frac{c^1_{1,+}(t+c^1_{1,+}(x-y))^2}{4c^2_{1,+}t}}P_{1,+}(0)A_+^{-1},&0<x<y,\\&\chi_{t\ge 1}\frac{\sqrt{c^1_{2,-}}}{\sqrt{4c^2_{2,-}\pi t}}e^{-\frac{c^1_{2,-}(t-c^1_{2,-}(x-y))^2}{4c^2_{2,-}t}}P_{2,-}(0)A_-^{-1},&y<x<0,\\
&0,&otherwise,\\
\end{aligned}\right.
\ee
\be 
\label{S1yterm}
S^1_y(x,t;y)=\chi_{{}_{t\ge 1,\, 0<x<y}}e^{-\frac{(t+c^1_{1,+}(x-y))^2}{Mt}}O\left(\frac{1}{t}\right)+\chi_{{}_{t\ge 1,\, y<x<0}}e^{-\frac{(t-c^1_{2,-}(x-y))^2}{Mt}}O\left(\frac{1}{t}\right),
\ee
\ba 
\label{Rbound}
R(x,t;y)=&\chi_{{}_{t\ge 1,\, 0<x<y}}e^{-\frac{(t+c^1_{1,+}(x-y))^2}{Mt}}O\left(\frac{1}{t}\right)+\chi_{{}_{t\ge 1,\, y<x<0}}e^{-\frac{(t-c^1_{2,-}(x-y))^2}{Mt}}O\left(\frac{1}{t}+\frac{1}{\sqrt{t}}e^{-\theta|x|}\right)\\
&+\chi_{{}_{t\ge 1,\, x<0}}\frac{1}{\sqrt{t}}e^{-\theta'|x|}O\left(\chi_{{}_{0<y}}e^{-\frac{(t-c^1_{1,+}y)^2}{Mt}}+\chi_{{}_{y<0}}e^{-\frac{(t+c^1_{2,-}y)^2}{Mt}}\right)+O(e^{-\bar\eta(|x-y|+t)}),
\ea 
\ba \label{Rybound}
R_y(x,t;y)=&\chi_{{}_{t\ge 1,\, 0<x<y}}e^{-\frac{(t+c^1_{1,+}(x-y))^2}{Mt}}O\left(\frac{1}{t^{\frac{3}{2}}}\right)+\chi_{{}_{t\ge 1,\, y<x<0}}e^{-\frac{(t-c^1_{2,-}(x-y))^2}{Mt}}O\left(\frac{1}{t^{\frac{3}{2}}}+\frac{1}{t}e^{-\theta|x|}\right)\\
&+\chi_{{}_{t\ge 1,\, x<0}}\frac{1}{t}e^{-\theta'|x|}O\left(\chi_{{}_{0<y}}e^{-\frac{(t-c^1_{1,+}y)^2}{Mt}}+\chi_{{}_{y<0}}e^{-\frac{(t+c^1_{2,-}y)^2}{Mt}}\right)+O(e^{-\bar\eta(|x-y|+t)}),
\ea
\be 
\label{s1second}
S^1(x,t;y)\left(\begin{array}{c}0\\1\end{array}\right)=0,
\ee 
and moreover
\ba 
\label{R2bound}
&R(x,t;y)\left(\begin{array}{c}0\\1\end{array}\right)\\
=&\chi_{{}_{t\ge 1,\, 0<x<y}}e^{-\frac{(t+c^1_{1,+}(x-y))^2}{Mt}}O\left(\frac{1}{t^{\frac{3}{2}}}\right)+\chi_{{}_{t\ge 1,\, y<x<0}}e^{-\frac{(t-c^1_{2,-}(x-y))^2}{Mt}}O\left(\frac{1}{t^{\frac{3}{2}}}\right)\\
&+\chi_{{}_{t\ge 1,\, x<0}}\frac{1}{t}e^{-\theta'|x|}O\left(\chi_{{}_{0<y}}e^{-\frac{(t-c^1_{1,+}y)^2}{Mt}}+\chi_{{}_{y<0}}e^{-\frac{(t+c^1_{2,-}y)^2}{Mt}}\right)+O(e^{-\bar\eta(|x-y|+t)}),
\ea
where $M$ is some sufficiently big constant and $\bar\eta$ is a sufficiently small positive constant.
\end{theorem}

The interior source kernel estimates of Theorem \ref{pointwiseintvGreen} may be recognized as essentially those of
the smooth profile case \cite{MZ,MZ2}.  Namely, as displayed in Figure \ref{cartoonfig}, the principal high-frequency component consists of time-decaying delta-functions moving along hyperbolic characteristics of \eqref{sv}
and refracting/reflecting from the shock, while the principal low-frequency component consists of 
time-algebraically decaying Gaussian signals moving along characteristics of the reduced, equilibrium 
system \eqref{CE}.

\begin{figure}[htbp]
\begin{center}
\includegraphics[scale=0.32]{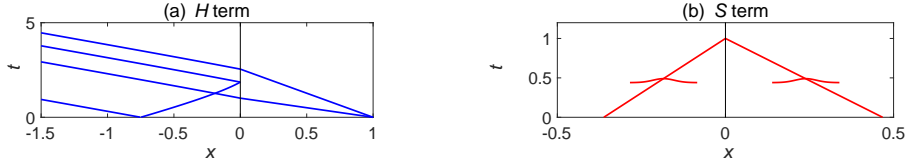}
\end{center}
\caption{Schematic of low- and high-frequency parts of the Green function $G$: a) numerically-computed 
support of high-frequency hyperbolic ($H$) terms in the $x$-$t$ plane for fixed initial $y$, showing  
transport along hyperbolic characteristics and refraction/reflection at the subshock. b) centers of Gaussians
making up Low-frequency scattering ($S$) terms, indicating approximate propagaion along characteristics of 
equilibrium system \eqref{CE}, with values frozen at end states $\overline{W}_\pm$.
}
\label{cartoonfig}
\end{figure}

The behavior of additional, boundary kernels in the discontinuous (subshock) case is similar.

\begin{theorem}\label{pointwisebouvGreen}
For $x>0$, the boundary source $v$-Green kernel function $K(x,t)$ defined in \eqref{Greenkernel} is identically 0. 
For $x<0$, it may be decomposed as
\be 
\label{Kdecompose}
K(x,t)=H_K+R_K,
\ee
where, assuming Evans-Lopatinsky stability,
\ba
H_K(x,t):=&-e^{-\bar\eta t+\int_0^x\left(-\bar\eta\mu_1(z)+M_{11}(z)\right)dz}R_1(x)L_1(0^-)V\delta\left(t+\int_0^x\mu_1(z)dz\right),& x<0,\\
R_K(x,t)=&O(e^{-\bar\eta(|x|+t)}),\quad x<0.
\ea
\end{theorem}

\begin{theorem}\label{pointwiseinttimeetaGreen}
The time derivative of the interior source $\eta$-Green kernel function $G_{1}(t;y)$ defined in \eqref{Greenkernel} may be decomposed as
\be 
\label{G1tdecompose}
G_{1t}=H_1+S_1+R_1,
\ee
where, assuming Evans-Lopatinsky stability,
\be 
H_1(t;y):=\left\{\begin{aligned}&e^{-\bar\eta t-\left(-\bar\eta\mu_{1,+}+M_{11,+}\right)y}V_hA_+R_{1,+}L_{1,+}A^{-1}_+ \delta\big(t-\mu_{1,+}y\big)\\
&+e^{-\bar\eta t-\left(-\bar\eta\mu_{2,+}+M_{22,+}\right)y}V_hA_+R_{2,+}L_{2,+}A^{-1}_+ \delta\big(t-\mu_{2,+}y\big),&0<y,\\
&e^{-\bar\eta t+\int_y^0\left(-\bar\eta\mu_2(z)+M_{22}(z)\right)dz}V_hA(0^-)R_2(0^-)L_2(y)A^{-1}(y) \delta\big(t+\int_y^0\mu_2(z)dz\big),&y<0,\end{aligned}\right.
\ee
\be
S_1(t;y):=\left\{\begin{aligned}&\chi_{t\ge 1}\frac{\sqrt{c^1_{1,+}}}{\sqrt{4c^2_{1,+}\pi t}}e^{-\frac{c^1_{1,+}(t-c^1_{1,+}y)^2}{4c^2_{1,+}t}}V_lA_+P_{1,+}(0)A_+^{-1},&0<y,\\&\chi_{t\ge 1}\frac{\sqrt{c^1_{2,-}}}{\sqrt{4c^2_{2,-}\pi t}}e^{-\frac{c^1_{2,-}(t+c^1_{2,-}y)^2}{4c^2_{2,-}t}}V_lA(0^-)P_{2,-}(0)A_-^{-1},&y<0,
\end{aligned}\right.
\ee
\be 
S_{1y}(t;y)=\chi_{{}_{t\ge 1,\, 0<y}}e^{-\frac{(t-c^1_{1,+}y)^2}{Mt}}O\left(\frac{1}{t}\right)+\chi_{{}_{t\ge 1,\, y<0}}e^{-\frac{(t+c^1_{2,-}y)^2}{Mt}}O\left(\frac{1}{t}\right),
\ee
\be
R_1(t;y)=\chi_{{}_{t\ge 1,\, 0<y}}e^{-\frac{(t-c^1_{1,+}y)^2}{Mt}}O\left(\frac{1}{t}\right)+\chi_{{}_{t\ge 1,\, y<0}}e^{-\frac{(t+c^1_{2,-}y)^2}{Mt}}O\left(\frac{1}{t}\right)+O(e^{-\bar\eta(|y|+t)}),
\ee
\be
R_{1,y}(t;y)=\chi_{{}_{t\ge 1,\, 0<y}}e^{-\frac{(t-c^1_{1,+}y)^2}{Mt}}O\left(\frac{1}{t^{\frac{3}{2}}}\right)+\chi_{{}_{t\ge 1,\, y<0}}e^{-\frac{(t+c^1_{2,-}y)^2}{Mt}}O\left(\frac{1}{t^{\frac{3}{2}}}\right)+O(e^{-\bar\eta(|y|+t)}),
\ee
\be 
\label{s_1second}
S_1(t;y)\left(\begin{array}{c}0\\1\end{array}\right)=0,
\ee 
and moreover
\be
R_1(t;y)\left(\begin{array}{c}0\\1\end{array}\right)=\chi_{{}_{t\ge 1,\, 0<y}}e^{-\frac{(t-c^1_{1,+}y)^2}{Mt}}O\left(\frac{1}{t^{\frac{3}{2}}}\right)+\chi_{{}_{t\ge 1,\, y<0}}e^{-\frac{(t+c^1_{2,-}y)^2}{Mt}}O\left(\frac{1}{t^{\frac{3}{2}}}\right)+O(e^{-\bar\eta(|y|+t)}),
\ee
where $M$ is some sufficiently big constant, $\bar\eta$ is a sufficiently small positive constant, and $V_h$, $V_l$ defined in \eqref{defVh}, \eqref{defVl} are constant vectors.
\end{theorem}
\begin{theorem}\label{pointwiseboutimeetaGreen}
The time derivative of boundary source $\eta$-Green kernel function $K_{1}(t)$ defined in \eqref{Greenkernel}may be decomposed as
\be 
\label{estimateK1t}
K_{1t}=H_{K_1}+R_{K_1},
\ee
where, assuming Evans-Lopatinsky stability, 
\be\label{K1tests}
\hbox{\rm $H_{K_1}(t)=V_h\delta(t)$ and $R_{K_1}=O(e^{-\bar{\eta}t})$.}
\ee
\end{theorem}

\begin{obs}[Special structure on $P_{2,-}(0)A_-^{-1}$ and $P_{1,+}(0)A_+^{-1}$ ]\label{apecialstructiononS} 
The matrices $P_{2,-}(0)A_-^{-1}$ and $P_{1,+}(0)A_+^{-1}$ can be computed symbolically to be
{\small
$$
P_{2,-}(0)A_-^{-1}=\left(\begin{array}{cc} \frac{2+2\sqrt{H_R}}{1-2H_R+\sqrt{H_R}} & 0\\ -\frac{9H_R-3+6{H_R}^{3/2}}{1-5H_R+4{H_R}^2} & 0 \end{array}\right),\quad
P_{1,+}(0)A_+^{-1}=\left(\begin{array}{cc} \frac{2+2\sqrt{H_R}}{H_R-2+\sqrt{H_R}} & 0\\ -\frac{9H_R-3{H_R}^2+6\sqrt{H_R}}{4-5H_R+{H_R}^2} & 0 \end{array}\right).
$$}
In particular, the second columns vanish and equations \eqref{s1second} \eqref{s_1second} follow.
\end{obs}

\begin{proof}[Proof of Theorem \ref{pointwiseintvGreen}]

\textbf{Case I. $|x-y|/t$ sufficiently large}
Following \cite{MZ}, we note, for $|x-y|/t>S$, $S$ sufficiently large, that $G=0$.
Taking $a$ sufficiently large in \eqref{Greenkernel}, we can use Proposition \ref{highestimates} to estimate $G^{1,2}(x,t;y)$. For example on $x<y<0$, 
\ba
\label{vanishingG}
|G^1(x,t;y)|=&\left|\frac{1}{2\pi}P.V.\int_{a-i\infty}^{a+i\infty}e^{\lambda t}e^{\int_y^x\left(\lambda\mu_1(z)+M_{11}(z)\right) dz}\left(R_1(x)L_1(y)A^{-1}(y)+O(1/|\lambda|)\right)d\lambda\right|\\
\le&\left|\frac{1}{2\pi}P.V.\int_{a-i\infty}^{a+i\infty}e^{\lambda t}e^{\int_y^x\left(\lambda\mu_1(z)+M_{11}(z)\right) dz}R_1(x)L_1(y)A^{-1}(y)d\lambda\right|\\
&+\left|\frac{1}{2\pi}P.V.\int_{a-i\infty}^{a+i\infty}e^{\lambda t}e^{\int_y^x\left(\lambda\mu_1(z)+M_{11}(z)\right) dz}O(1/|\lambda|)d\lambda\right|\\
\lesssim&e^{a(t+\int_y^x\mu_1(z)dz)+\int_y^xM_{11}(z)dz}\left|P.V.\int_{-\infty}^\infty e^{i\xi\big(t+\int_y^x\mu_1(z)dz\big)}d\xi\right|\\
&+e^{a(t+\int_y^x\mu_1(z)dz)+\int_y^xM_{11}(z)dz}\int_{-\infty}^{\infty}\frac{1}{\sqrt{a^2+\xi^2}}d\xi\\
=&e^{a\big(t+\int_y^x\mu_1(z)dz\big)+\int_y^xM_{11}(z)dz}\delta\Big(t+\int_y^x\mu_1(z)dz\Big)\\
&+e^{a\big(t+\int_y^x\mu_1(z)dz\big)+\int_y^xM_{11}(z)dz}\int_{-\infty}^{\infty}\frac{1}{\sqrt{a^2+\xi^2}}d\xi.
\ea
In "$\lesssim$" line, the integral can be explicitly computed because $\mu_1,M_{11},R_1,L_1$, and $A$ are independent of $\lambda$. And, on the next line, using the triangle inequality for the integral yields the bound. Since there are $c$, $C$ depending only on $F$, $H_R$ such that 
$$-C<\mu_2(H(z))<-c<0<c<\mu_1(z)<C,$$
$\delta\Big(t+\int_y^x\mu_1(z)dz\Big)$ will be $0$ provided that $\frac{|x-y|}{t}<\frac{1}{C}$ or $\frac{|x-y|}{t}>C$ for some $C$ sufficiently large. As for the term in the last row of \eqref{vanishingG}, for $|x-y|/t$ sufficiently large, $t+\int_y^x\mu_1(z)dz$ will be a negative number. Thus by sending $a$ to $+\infty$ this term also vanishes. The same result holds on $y<x<0$ and $0<x<y$. Similarly, $G^2(x,t;y)$ also vanishes.

\medskip

\textbf{Case II. $|x-y|/t$ bounded}. First, observe that $|x-y|\leq Ct$ yields for $\theta>0$ that
$$
e^{-\theta t}\leq e^{-\theta_1 (t+|x-y|)}
$$
for some $\theta_1>0$, a contribution absorbable in error term $R$.  Thus, in this regime, it is enough to show that terms 
are time-exponentially small in order to verify that they are absorbable in $R$.

By our construction of resolvent kernels, they are meromorphic on the set of consistent splitting, with poles
precisely at zeros of the Evans-Lopatinsky function $\Delta$ \eqref{lopatinsky}.
Function $\Delta$ is nonvanishing on $\{\lambda:\Re\lambda\ge -a,|\lambda|>r\}$ by a combination of Proposition \ref{hfprop} and the assumed Evans-Lopatinsky stablity, that is, $M$ is invertible on $\{\lambda:\Re\lambda\ge -a,|\lambda|>r\}$. 
As observed in Remark \ref{movermk}, we can deform the contour of integration in \eqref{Greenkernel} to 
\eqref{deformed}.
Since by Lemma \ref{GKanalytic}, $G_\lambda$ is holomorphic in a small neighborhood of the origin, we can further deform the contour to the left of the origin and obtain 
\ba
\label{lmh}
G(x,t;y) &=\frac{1}{2\pi i}\int_{-\bar\eta-ir}^{-\bar\eta+ir}e^{\lambda t}G_\lambda d\lambda+\frac{1}{2\pi i}\left(\int_{-\bar\eta-iR}^{-\bar\eta-ir}+\int_{-\bar\eta+ir}^{-\bar\eta+iR}\right)e^{\lambda t}G_\lambda d\lambda\\
&+
P.V.  \frac{1}{2\pi i}
\left(\int_{-\bar\eta-i\infty}^{-\bar\eta-iR}+\int_{-\bar\eta+iR}^{-\bar\eta+i\infty}\right)e^{\lambda t}G_\lambda d\lambda
:=I+II+III
\ea
for $\bar\eta>0$, $\bar\eta$ sufficiently small. We will use superscript $1,2$ to denote contributions from $G_\lambda^{1,2}$ to $G$.

\textbf{Intermediate frequency contribution $II$.} 
For $\lambda$ in the intermediate frequency regime $[-\bar\eta-iR,-\bar\eta-ir]$ and $[-\bar\eta+ir,-\bar\eta+iR]$, the resolvent kernel is bounded. Therefore term $II$ is time-exponentially small of order $e^{-\bar\eta t}$ and absorbable in $R$.

\textbf{High frequency contribution $III$.}
In this regime, we can again use Proposition $\ref{highestimates}$. The term $III^1$ can be written as
\ba
III^1=&\frac{1}{2\pi i}P.V.\int_{-\bar\eta-i\infty}^{-\bar\eta+i\infty}e^{\lambda t}H^1_\lambda d\lambda-\frac{1}{2\pi i}\int_{-\bar\eta-iR}^{-\bar\eta+iR}e^{\lambda t}H^1_\lambda d\lambda\\
&+\frac{1}{2\pi i}P.V.\left(\int_{-\bar\eta-i\infty}^{-\bar\eta-iR}+\int_{-\bar\eta+iR}^{-\bar\eta+i\infty}\right)e^{\lambda t}(G_\lambda^1-H^1_\lambda)d\lambda\\
:=&III^1_{a}+III^1_{b}+III^1_{c}.
\ea
The term $III^1_{a}$ can be explicitly computed to be
\ba
III^1_{a}=&-e^{-\bar\eta t+\int_y^x\left(-\bar\eta\mu_1(z)+M_{11}(z)\right)dz}R_1(x)L_1(y)A^{-1}(y) \frac{1}{2\pi }P.V.\int_{-\infty}^\infty e^{i\xi\big(t+\int_y^x\mu_1(z)dz\big)}d\xi\\
=&-e^{-\bar\eta t+\int_y^x\left(-\bar\eta\mu_1(z)+M_{11}(z)\right)dz}R_1(x)L_1(y)A^{-1}(y) \delta\big(t+\int_y^x\mu_1(z)dz\big),
\ea
which gives contribution $H^1$ on $x<y<0$. As for the term $III^1_b$, it can be bounded by
\ba
|III^1_{b}|=&\left|e^{-\bar\eta t+\int_y^x\left(-\bar\eta \mu_1(z)+M_{11}(z)\right)dz}R_1(x)L_1(y)A^{-1}(y) \frac{1}{2\pi }\int_{-R}^R e^{i\xi\big(t+\int_y^x\mu_1(z)dz\big)}d\xi\right|\\
\lesssim&e^{-\bar\eta t+\int_y^x\left(-\bar\eta\mu_1(z)+M_{11}(z)\right)dz}\\ 
\le&e^{-\bar\eta\left(t+(y-x)\right)}
\ea
in which the last inequality follow by $-\bar\eta\mu_1(z)+M_{11}(z)>c>\bar\eta>0$ for $\bar\eta$ sufficiently small and all $z<0$. Hence $III^1_b$ is absorbable in $R$. By Proposition \ref{highestimates}, $G_\lambda^1-H^1_\lambda$ 
expands as $1/\lambda$ times a bounded function $h(x,y)$ plus an error term of
order $O(\frac{1}{|\lambda|^2})$ on the contour of integral $III^1_c$. Thus,
\ba
|III^1_{c}|\lesssim 
&
e^{-\bar\eta t}h(x,y) P.V. (\int_{R}^{\infty}+\int_{-\infty}^R)\lambda^{-1} d\lambda
+ e^{-\bar\eta t}\int_{R}^{\infty}\frac{1}{\bar\eta^2+\xi^2}d\xi\lesssim e^{-\bar\eta t},
\ea
which again is absorbable in $R$. Similar analysis can be carried out on $y<x<0$ and for $G^2_\lambda$.\\

\textbf{Low frequency contribution $I$.} 

({\it Case $t\le 1$}). Estimates in the short-time regime $t\le 1$ are trivial. Since then $e^{\lambda t}G_\lambda$ is uniformly bounded on the compact set $[-\bar\eta-ir,-\bar\eta+ir]$, we have $|I|\lesssim e^{-\bar\eta t}$ is absorbable in $R$.

({\it Case $t\ge 1$}). Next, consider $I^1$, $I^2$ on the critical regime $t\ge 1$ and $y<x<0$.

\textbf{$I^1$:} Decompose $G^1_\lambda=S^1_\lambda+R^1_\lambda$ and write $I^1$ as
\be 
I^1=\frac{1}{2\pi i}\int_{-\bar\eta-ir}^{-\bar\eta+ir}e^{\lambda t}S^1_\lambda d\lambda+\frac{1}{2\pi i}\int_{-\bar\eta-ir}^{-\bar\eta+ir}e^{\lambda t}R^1_\lambda d\lambda:=I^1_S+I^1_R.
\ee
We then analyze $I^1_S,I^1_R$ separately.

\textbf{$I^1_S$:} Deform the integral to write $I^1_S$ as
\be 
I^1_S=\frac{1}{2\pi i}\left(\int_{-\bar\eta-ir}^{\eta_*-ir}+\int_{\eta_*-ir}^{\eta_*+ir}+\int_{\eta_*+ir}^{-\bar\eta+ir}\right)e^{\lambda t}S_\lambda d\lambda:=I^1_{S1}+I^1_{S2}+I^1_{S3},
\ee
where the saddle point $\eta_*(x,y,t)$ is defined as 
\be
\eta_*(x,y,t):=\left\{\begin{aligned}
&\frac{\bar\alpha}{p}, &if\quad \left|\frac{\bar\alpha}{p}\right|\le \eps,\\
&\pm\eps, &if\quad\frac{\bar\alpha}{p}\gtrless\pm\eps,
\end{aligned}
\right.
\ee
with
\be 
\bar\alpha:=\frac{x-y-\frac{1}{c^1_{2,-}}t}{2t},\quad p:=\frac{c^2_{2,-}(x-y)}{c^1_{2,-}t}.
\ee 
A key observation is:
When $\left|\frac{\bar\alpha}{p}\right|\le \eps$, $\eps$ sufficiently small, $\frac{t}{c^1_{2,-}},x-y$ are comparable, that is we have
\be 
\label{comparable}
\frac{1}{2}(x-y)<(1-\frac{2c^2_{2,-}\eps}{c^1_{2,-}})(x-y)<\frac{t}{c^1_{2,-}}<(1+\frac{2 c^2_{2,-}\eps}{c^1_{2,-}})(x-y)<2(x-y).
\ee 

\begin{obs}\label{expyestimate} Assuming the comparability condition \eqref{comparable} and $y<x<0$, $e^{-\theta'|y|}$ is time-exponentially decaying.
\end{obs}

\begin{proof}
By the comparability condition, $y<x-\frac{t}{2c_{2,-}^1}<-\frac{t}{2c_{2,-}^1}$, we have
$e^{-\theta'|y|}<e^{-\frac{\theta'}{2c_{2,-}^1}t}$.
\end{proof}

\textbf{$I^1_{S2}:$} i. When $\left|\frac{\bar\alpha}{p}\right|\le\eps$, $I^1_{S2}$ can be explicitly computed to be
\ba 
I^1_{S2}=&\frac{1}{2\pi}e^{-\frac{\left(t-c^1_{2,-}(x-y)\right)^2}{4c^2_{2,-}\left(x-y\right)}}\int_{-r}^re^{-c^2_{2,-}(x-y)\xi^2}d\xi P_{2,-}(0)A_-^{-1}\\
=&\frac{1}{\sqrt{4c^2_{2,-}\pi(x-y)}}e^{-\frac{(t-c^1_{2,-}(x-y))^2}{4c^2_{2,-}(x-y)}}P_{2,-}(0)A_-^{-1}\\
&-\frac{1}{2\pi}e^{-\frac{\left(t-c^1_{2,-}(x-y)\right)^2}{4c^2_{2,-}\left(x-y\right)}}\left(\int_{-\infty}^{-r}+\int_r^\infty\right)e^{-c^2_{2,-}(x-y)\xi^2}d\xi P_{2,-}(0)A_-^{-1}\\
=&
\frac{\sqrt{c^1_{2,-}}}{\sqrt{4c^2_{2,-}\pi t}}e^{-\frac{c^1_{2,-}(t-c^1_{2,-}(x-y))^2}{4c^2_{2,-}t}}P_{2,-}(0)A_-^{-1}\\&+\left(-\frac{\sqrt{c^1_{2,-}}}{\sqrt{4c^2_{2,-}\pi t}}e^{-\frac{c^1_{2,-}(t-c^1_{2,-}(x-y))^2}{4c^2_{2,-}t}}+\frac{1}{\sqrt{4c^2_{2,-}\pi(x-y)}}e^{-\frac{(t-c^1_{2,-}(x-y))^2}{4c^2_{2,-}(x-y)}}\right)P_{2,-}(0)A_-^{-1}\\
&-\frac{1}{2\pi}e^{-\frac{\left(t-c^1_{2,-}(x-y)\right)^2}{4c^2_{2,-}\left(x-y\right)}}\left(\int_{-\infty}^{-r}+\int_r^\infty\right)e^{-c^2_{2,-}(x-y)\xi^2}d\xi P_{2,-}(0)A_-^{-1}\\
:=&S^1+I^1_{S2Ri}+I^1_{S2Rii},
\ea
where $S^1$ gives contribution $S^1$ in \eqref{tildeGdecompose} and $I^1_{S2Ri}$, $I^1_{S2Rii}$ are shown in Appendix \ref{ApEstimate} to be absorbable in $R$.

$S^1_y$: By direct calculation,
\ba 
|S^1_y|=\left|P_{2,-}(0)A_-^{-1}\frac{\sqrt{c^1_{2,-}}}{\sqrt{4c^2_{2,-}\pi t}}e^{-\frac{c^1_{2,-}(t-c^1_{2,-}(x-y))^2}{4c^2_{2,-}t}}\frac{{c^1_{2,-}}(t-c^1_{2,-}(x-y))}{2c^2_{2,-}t}\right|\lesssim\frac{1}{t}e^{-\frac{c^1_{2,-}(t-c^1_{2,-}(x-y))^2}{8c^2_{2,-}t}}.
\ea
ii. When $\frac{\bar\alpha}{p}> \eps$, the term $I^1_{S2}$ can be bounded by
\ba 
|I^1_{S2}|\lesssim& e^{\eps(t-c^1_{2,-}(x-y))+\eps^2 c^2_{2,-}(x-y)}\int_{-r}^re^{-c^2_{2,-}(x-y)\xi^2}d\xi\\
\lesssim& e^{\eps(t-c^1_{2,-}(x-y)+\eps c^2_{2,-}(x-y))}\\
\le&e^{\frac{1}{2}(t-c^1_{2,-}(x-y))\eps}\le e^{-\eps't},
\ea
Hence it is absorbable in $R$. Similarly when $\frac{\bar\alpha}{p}<-\eps$, $I_{S_2}$ is also time-exponentially small.

\textbf{$I^1_{S1}$ and $I^1_{S3}:$} i. When $\left|\frac{\bar\alpha}{p}\right|\le\eps$, the term $I^1_{S1}$ and $I^1_{S3}$ can be estimated as
\ba
|I^1_{S1}|, |I^1_{S3}|\lesssim e^{-c^2_{2,-}r^2(x-y)}\left|\int_{-\bar\eta}^{\eta_*}e^{\left(t-c^1_{2,-}(x-y)\right)\xi+c^2_{2,-}(x-y)\xi^2}d\xi\right|.
\ea
Since $\eta_*$ is the critical point of quadratic function $\left(t-c^1_{2,-}(x-y)\right)\xi+c^2_{2,-}(x-y)\xi^2$, we then have 
\be
|I^1_{S1}|, |I^1_{S3}|\lesssim e^{-c^2_{2,-}r^2(x-y)}e^{-\left(t-c^1_{2,-}(x-y)\right)\bar\eta+c^2_{2,-}(x-y)\bar\eta^2}\left|\eta_*+\bar\eta\right|.
\ee
Choosing $\bar\eta$ sufficiently small with respect to $r^2$ and using comparability of $\frac{t}{c^1_{2,-}}$ and $x-y$, $I^1_{S1}, I^1_{S3}$ is then time-exponentially decaying.

ii. When $\left|\frac{\bar\alpha}{p}\right|>\eps$, we have
\ba
&|I^1_{S1}|,|I^1_{S3}|\\\lesssim& e^{-c^2_{2,-}r^2(x-y)}\left|\int_{-\bar\eta}^{\pm\eps}e^{\left(t-c^1_{2,-}(x-y)\right)\xi+c^2_{2,-}(x-y)\xi^2}d\xi\right|\\
\lesssim&e^{-c^2_{2,-}r^2(x-y)}\left(e^{-\left(t-c^1_{2,-}(x-y)\right)\bar\eta+c^2_{2,-}(x-y)\bar\eta^2}+e^{-\left(t-c^1_{2,-}(x-y)\right)\eps+c^2_{2,-}(x-y)\eps^2}\right)(\eps+\bar\eta)\\
=&\left(e^{-\bar\eta t+(c^1_{2,-}\bar\eta+c^2_{2,-}\bar\eta^2-c^2_{2,-}r^2)(x-y)}+e^{-\eps t+(c^1_{2,-}\eps+c^2_{2,-}\eps^2-c^2_{2,-}r^2)(x-y)}\right)(\eps+\bar\eta).
\ea
Again choosing $\bar\eta$, $\eps$ sufficiently small with respect to $r$ yields that $I_{S1}$, $I_{S3}$ are time-exponentially small.

\textbf{$I^1_R$:} Using $P_{2,-}(\lambda)=P_{2,-}(0)+O(|\lambda|)$, $T_-(\lambda,x)=Id+O(e^{-\theta'|x|})$, $T_-^{-1}(\lambda,y)=Id+O(e^{-\theta'|y|})$, $\gamma_{2,-}(\lambda)=\tilde{\gamma}_{2,-}(\lambda)+O(|\lambda|^3)$, and $A^{-1}(y)=A^{-1}_-+O(e^{-\theta y})$, $R^1_\lambda$ can be estimated as
\ba 
\label{IR1decom}
R^1_\lambda=&-e^{\tilde{\gamma}_{2,-}(\lambda)(x-y)}P_{2,-}(0)A_-^{-1}+e^{\gamma_{2,-}(\lambda)(x-y)}T_-(\lambda,x)P_{2,-}(\lambda)T_-^{-1}(\lambda,y)A^{-1}(y)\\
=&-e^{\tilde{\gamma}_{2,-}(\lambda)(x-y)}P_{2,-}(0)A_-^{-1}+e^{\tilde{\gamma}_{2,-}(\lambda)(x-y)}\left(1+O(|\lambda^3(x-y)|)\right)\left(Id+O(e^{-\theta'|x|})\right)\\
&\times\left(P_{2,-}(0)+O(|\lambda|)\right)\left(Id+O(e^{-\theta'|y|})\right)\left(A^{-1}_-+O(e^{-\theta y})\right)\\
=&e^{\tilde{\gamma}_{2,-}(\lambda)(x-y)}\left(O(|\lambda^3(x-y)|)+O(e^{-\theta |x|})\right)P_{2,-}(0)A_-^{-1}\\
&+e^{\tilde{\gamma}_{2,-}(\lambda)(x-y)}\left(O(|\lambda|)A_-^{-1}+P_{2,-}(0)O(e^{-\theta'|y|})A_-^{-1}+P_{2,-}(0)O(e^{-\theta|y|})\right)\\
:=&\#_1+\#_2.
\ea
Deform the contour as before to write $I^1_R$ as
\be 
I^1_R=\frac{1}{2\pi i}\left(\int_{-\bar\eta-ir}^{\eta_*-ir}+\int_{\eta_*-ir}^{\eta_*+ir}+\int_{\eta_*+ir}^{-\bar\eta+ir}\right)e^{\lambda t}R^1_\lambda d\lambda:=I^1_{R1}+I^1_{R2}+I^1_{R3}.
\ee
\textbf{$I^1_{R2}$:} On the contour $[\eta_*-ir,\eta_*+ir]$, we notice that 
\be 
O(|\lambda|)=O(|\eta_*|)+O(|\xi|),\quad O(|\lambda^3(x-y)|)=\sum_{i=0}^3 O(|\eta_*|^i||\xi|^{3-i}|x-y|) .
\ee

i. When $\left|\frac{\bar\alpha}{p}\right|\le \eps$, $I^1_{R2}$ can be estimated as
\ba
\label{IR2}
|I^1_{R2}|\lesssim& e^{-\frac{\left(t-c^1_{2,-}(x-y)\right)^2}{4c^2_{2,-}\left(x-y\right)}}\int_{-r}^re^{-\xi^2c^2_{2,-}(x-y)}O\left(|\eta_*|+|\xi|+\sum_{i=0}^3 |\eta_*|^i|\xi|^{3-i}|x-y|+e^{-\theta'|x|}+e^{-\theta'|y|}\right) d\xi\\
\lesssim &e^{-\frac{\left(t-c^1_{2,-}(x-y)\right)^2}{4c^2_{2,-}\left(x-y\right)}}\int_{-r}^re^{-\xi^2c^2_{2,-}(x-y)}O\left(|\eta_*|+|\xi|+\sum_{i=0}^3 |\eta_*|^i|\xi|^{3-i}|x-y|\right)d\xi\\
&+\frac{1}{\sqrt{c^2_{2,-}(x-y)}}e^{-\frac{\left(t-c^1_{2,-}(x-y)\right)^2}{4c^2_{2,-}\left(x-y\right)}}O(e^{-\theta'|x|}+e^{-\theta'|y|})\\
:=&I^1_{R2i}+I^1_{R2ii}.
\ea
The $I^1_{R2ii}$ term in the last line of \eqref{IR2} is absorbable in $R$ because $\frac{t}{c^1_{2,-}}$ and $x-y$ are comparable so $\frac{1}{\sqrt{c^2_{2,-}(x-y)}}e^{-\frac{\left(t-c^1_{2,-}(x-y)\right)^2}{4c^2_{2,-}\left(x-y\right)}}O(e^{-\theta'|x|})$ is absorbable in $R$ \eqref{Rbound} and by Observation \ref{expyestimate} the term $\frac{1}{\sqrt{c^2_{2,-}(x-y)}}e^{-\frac{\left(t-c^1_{2,-}(x-y)\right)^2}{4c^2_{2,-}\left(x-y\right)}}O(e^{-\theta'|y|})$ is time-exponentially small hence also absorbable. The term $I^1_{R2i}$ is shown in Appendix \ref{ApEstimate} to be absorbable in $R$.

ii. When $\frac{\bar\alpha}{p}>\eps$, following part ii above in the estimation of $I^1_{S2}$, we find
that $I^1_{R2}$ is also time-exponentially decaying.
%And for $I^1_{R1}$ and $I^1_{R3}$, 
Using \eqref{IR1decom} and imitating the way of estimating $I^1_{S1}$ and $I^1_{S3}$,
we find that $I^1_{R1}$ and $I^1_{R3}$ are also time-exponentially decaying.

In the regime $t\ge 1$ and $x<y<0$, by Proposition \ref{extendtildeG}, $G^1_\lambda=R^1_\lambda$. Because $\Re\gamma_{1,-}(\lambda)>c>0$ in a small neighborhood of the origin, $G^1_\lambda(x;y)$ is then uniformly bounded in $x<y<0$, 
and so the term $I^1$ in \eqref{lmh} is time-exponentially decaying.

Following the way of estimating $I^1_R$ and using estimates \eqref{G2low}, $I^2$ can be estimated in a similar way and absorbed in $R$.

\textbf{$R\left(\begin{array}{c}0\\1\end{array}\right)$:} In $R$\footnote{Refer to Appendix \ref{Decompositionmap} equation \eqref{decomposition2}(iii) for decomposition of $R$.}, the terms $II$, $III_{b,c}^{1,2}$, and $\chi_{t\le 1}I$ are time-exponentially small, hence absorbable in \eqref{R2bound}. By Observation \ref{apecialstructiononS}, any terms in $R$ that has  a labeling ``$S$" will become $0$ when right multiplied by $\left(\begin{array}{c}0\\1\end{array}\right)$. The only term 
remaining to be analyzed is $I_R^1$. By \eqref{IR1decom}, $\#_1\left(\begin{array}{c}0\\1\end{array}\right)=0$, the other term $\#_2\left(\begin{array}{c}0\\1\end{array}\right)$ is absorbable in \eqref{R2bound} by Observation \ref{expyestimate}.

Finally estimation of $R_y$ can be done by estimating $y$-derivatives of terms in \eqref{decomposition2}(iii) separately. That is for
\begin{itemize}
\item $II$:  The $y$-derivative of the resolvent kernel is bounded on the intermediate frequency regime, hence $II_y$ is time-exponentially small;
\item $III_{b,c}^{1,2}$: Direct computation shows they are time-exponentially small;
\item $\chi_{t\le 1}I$: $e^{\lambda t}G_{\lambda,y}$ uniformly bounded, so time-exponentially small;
\item $I^1_{S1,S3},\chi_{\frac{\bar{\alpha}}{p}>\eps}I^1_{S2}$: When the $y$-derivative hits the exponential term, 
	this will bring down only the order-one exponential rate, with no improvement due to differentiation.
	But, this term is already uniformly bounded for low frequencies. So, these terms again are time-exponentially 
	small by following the estimates in the undifferentiated case;
\item $\chi_{\frac{\bar{\alpha}}{p}\le\eps}I_{S2Ri}^1$: See Appendix \ref{ApEstimate};
\item $\chi_{\frac{\bar{\alpha}}{p}\le\eps}I_{S2Rii}^1$: See Appendix \ref{ApEstimate};
\item $I_R^1$: Again we demonstrate how to estimate 
$\frac{\partial I_R^1}{\partial y}$ on the critical regime $y<x<0$. By direct computation,
\ba
&\frac{\partial}{\partial y}\left(-e^{\tilde{\gamma}_{2,-}(\lambda)(x-y)}P_{2,-}(0)A_-^{-1}+e^{\gamma_{2,-}(\lambda)(x-y)}T_-(\lambda,x)P_{2,-}(\lambda)T_-^{-1}(\lambda,y)A^{-1}(y)\right)\\=&\tilde{\gamma}_{2,-}(\lambda)e^{\tilde{\gamma}_{2,-}(\lambda)(x-y)}P_{2,-}(0)A_-^{-1}-\gamma_{2,-}e^{\gamma_{2,-}(\lambda)(x-y)}T_-(\lambda,x)P_{2,-}(\lambda)T_-^{-1}(\lambda,y)A^{-1}(y)\\
&+e^{\gamma_{2,-}(\lambda)(x-y)}T_-(\lambda,x)P_{2,-}(\lambda)\frac{\partial}{\partial y}\left(T_-^{-1}(\lambda,y)A^{-1}(y)\right)\\
:=&\#_3+\#_4,
\ea
and
\ba
\#_3=&\tilde{\gamma}_{2,-}(\lambda)e^{\tilde{\gamma}_{2,-}(\lambda)(x-y)}P_{2,-}(0)A_-^{-1}-\left(\tilde{\gamma}_{2,-}(\lambda)+O(|\lambda|^3)\right)e^{\tilde{\gamma}_{2,-}(\lambda)(x-y)}\left(1+O(|\lambda^3(x-y)|)\right)\\
&\times\left(Id+O(e^{-\theta' |x|})\right)\left(P_{2,-}(0)+O(|\lambda|)\right)\left(Id+O(e^{-\theta'|y|})\right)\left(A_-^{-1}+O(e^{-\theta |y|})\right)\\
=&\left(O(|\lambda|^3)e^{\tilde{\gamma}_{2,-}(\lambda)(x-y)}+\tilde{\gamma}_{2,-}(\lambda)e^{\tilde{\gamma}_{2,-}(\lambda)(x-y)}\left(O(|\lambda^3(x-y)|+O(e^{-\theta' |x|})\right)\right)P_{2,-}(0)A_-^{-1}\\
&+\tilde{\gamma}_{2,-}(\lambda)e^{\tilde{\gamma}_{2,-}(\lambda)(x-y)}\left(O(|\lambda|)A_-^{-1}+P_{2,-}(0)O(e^{-\theta'|y|})A_-^{-1}+P_{2,-}(0)O(e^{-\theta|y|})\right)\\
=&\left(O(|\lambda|^3)e^{\tilde{\gamma}_{2,-}(\lambda)(x-y)}+O(|\lambda|)e^{\tilde{\gamma}_{2,-}(\lambda)(x-y)}\left(O(|\lambda^3(x-y)|+O(e^{-\theta' |x|})\right)\right)P_{2,-}(0)A_-^{-1}\\
&+O(|\lambda|)e^{\tilde{\gamma}_{2,-}(\lambda)(x-y)}\left(O(|\lambda|)A_-^{-1}+P_{2,-}(0)O(e^{-\theta'|y|})A_-^{-1}+P_{2,-}(0)O(e^{-\theta|y|})\right).
\ea
We then see 
\be
|\#_3|=|e^{\tilde{\gamma}_{2,-}(\lambda)(x-y)}|O\left(|\lambda|^2+|\lambda|^4|x-y|+|\lambda|e^{-\theta|x|}+|\lambda|e^{-\theta'|y|}\right).
\ee
The estimation is then similar to that of $I^1_R$ and we gain another $\frac{1}{\sqrt{t}}$ from the extra $|\lambda|$. Because
\be
\#_4=e^{\tilde{\gamma}_{2,-}(\lambda)(x-y)}P_{2,-}(0)O(e^{-\theta'|y|}).
\ee
The estimation is again similar to that of $I^1_R$ in particular by Observation \ref{expyestimate}, this term is time-exponentially small.
\end{itemize}
\end{proof}

\begin{proof}[Proof of Theorem \ref{pointwisebouvGreen}]
Following a similar analysis as for {\it High frequency contribution III} in the estimation 
of $G$ above and noting the bound \eqref{Klow} at low frequencies, we find that
the boundary source $v$-resolvent kernel $K_\lambda$ is exponentially decaying in $|x|$,
hence the boundary source $v$-Green kernel $K$ decomposes into only an $H_K$ part and an $R$ part 
as written in \eqref{Kdecompose}. 
\end{proof}

\begin{proof}[Proof of Theorems \ref{pointwiseinttimeetaGreen} and \ref{pointwiseboutimeetaGreen}]
Taking time derivatives of $K_1$, $G_1$ in \eqref{Greenkernel}, we get an extra $\lambda$ which removes the singularity of $M^{-1}$ at $\lambda=0$. The decompositions for $K_{1t}$ \eqref{estimateK1t}, $G_{1t}$ \eqref{G1tdecompose} 
then follow by shifting the contour to $\{\lambda:\Re \lambda =-\bar\eta\}$ and using 
Propositions \ref{highestimates1} and \ref{extendlambdaGK}. 
\end{proof}

\section{Linear stability}\label{s:linstab}
From the pointwise estimates of Theorem \ref{pointwiseintvGreen} we 
obtain the following linear $L^q\to L^p$ stability estimates, from which we will ultimately derive nonlinear
stability and asymptotic orbital stability.
%phase-asymptotic orbital stability.

\subsection{Linear orbital stability estimates}\label{s:linorb}

\begin{lemma}\label{etadotlem}
The time derivative of equation \eqref{originalsol}(ii) is
\ba 
\label{timeprimeoriginalsol2}
\dot \eta(t)=&K_1(0)B_S(t)+\int_0^t K_{1t}(t-s)B_S(s)ds+\int_{-\infty}^\infty G_{1t}(t;y)v_0(y)dy\\
&+\int_{-\infty}^\infty G_{1}(0;y)I_S(t;y)dy+\int_0^t\int_{-\infty}^\infty G_{1t}(t-s;y)I_S(s,y)dyds.
\ea 
\end{lemma}
\begin{proof}
Directly by taking time derivative of \eqref{originalsol}(ii).
\end{proof}
Taking $B_S=0$, $I_S=0$ in \eqref{originalsol} and taking the time derivative of the $\eta$- equation yields
the linearized integral equations:
\be
\label{linearintegral}
v(x,t)=\int_{-\infty}^\infty G(x,t;y)v_0(y)dy,\quad \dot \eta(t)=\int_{-\infty}^\infty G_{1t}(t;y)v_0(y)dy.
\ee
Linear asymptotic orbital stability follows immediately from Theorem \ref{pointwiseintvGreen}. Splitting $G$ \eqref{tildeGdecompose} into singular part $H:=H^1+H^2$ and regular part $\bar{G}:=S^1+R$, $K$ \eqref{Kdecompose} into singular part $H_K$ and regular part $R_K$, and $G_{1t}$ \eqref{G1tdecompose} into singular part $H_1$ and regular part $\bar{G}_{1t}:=S_1+R_1$, we then have the following lemmas.

\begin{lemma}\label{Glem}
Assuming Evans-Lopatinsky stability, for the splitting $G=\bar G+H$, there hold:
\be
\left|\int_{-\infty}^{+\infty}\bar{G}(\cdot\;,t;y)f(y)dy\right|_{L^p}\le C(1+t)^{-\frac{1}{2}(1-1/r)}|f|_{L^q}+Ce^{-\bar{\eta}t}|f|_{L^p},
\ee
\be
\left|\int_{-\infty}^{+\infty}\bar{G}(\cdot\;,t;y)\left(\begin{array}{c}0\\f(y)\end{array}\right)dy\right|_{L^p}\le C(1+t)^{-\frac{1}{2}(2-1/r)}|f|_{L^q}+Ce^{-\bar{\eta}t}|f|_{L^p},
\ee
\be
\left|\int_{-\infty}^{+\infty}\bar{G}_y(\cdot\;,t;y)f(y)dy\right|_{L^p}\le C(1+t)^{-\frac{1}{2}(2-1/r)}|f|_{L^q}+Ce^{-\bar{\eta}t}|f|_{L^p},
\ee
and
\be\label{GlemH}
\left|\int_{-\infty}^{+\infty}H(\cdot\;,t;y)f(y)dy\right|_{L^p}\le Ce^{-\bar{\eta}t}|f|_{L^p},
\ee
for all $t\ge 0$, some $C,\bar\eta>0$, for any $1\le r\le p$ and $f\in L^q$ (resp. $L^p$), where $1/r+1/q=1+1/p$.
\end{lemma}

\begin{lemma}\label{Klem}
Assuming Evans-Lopatinsky stability, for the splitting $K=H_K+R_K$, there hold:
\be
\left| \int_0^t H_K(\cdot\;,t-s)g(s)ds\right|_{L^p}\le C|e^{-\bar \eta (t-\cdot)}g|_{L^p(0,t)}
\ee
and
\be
\left| \int_0^t R_K(\cdot\;,t-s)g(s)ds\right|_{L^p}\le C|e^{-\bar \eta (t-\cdot)}g|_{L^1(0,t)},
\ee
for all $t\ge 0$, some $C,\bar\eta>0$, for any $1\le p$ and $g\in L^p$ (resp. $L^1$).
\end{lemma}
\begin{lemma}\label{G1lem}
	Assuming Evans-Lopatinsky stability, for the splitting $G_{1t}=\bar{G}_{1t}+H_1$, there hold:
\be 
\label{barG1t}
\left|\int_{-\infty}^{\infty}\bar{G}_{1t}(t;y)f(y)dy\right|\le C(1+t)^{-\frac{1}{2q}}|f|_{L^q},
\ee
\be 
\left|\int_{-\infty}^{\infty}\bar{G}_{1t}(t;y)\left(\begin{array}{c}0\\f(y)\end{array}\right)dy\right|\le C(1+t)^{-\frac{1}{2}-\frac{1}{2q}}|f|_{L^q},
\ee
\be 
\left|\int_{-\infty}^{\infty}\bar{G}_{1ty}(t;y)f(y)dy\right|\le C(1+t)^{-\frac{1}{2}-\frac{1}{2q}}|f|_{L^q},
\ee
and
\be 
\left|\int_{-\infty}^{\infty}H_1(t;y)f(y)dy\right|\le Ce^{-\bar\eta t}|f|_{L^\infty},
\ee
for all $t\ge 0$, some $C,\bar\eta>0$, for any $1\le q$ and $f\in L^q$ (resp. $L^\infty$).
\end{lemma}

\begin{lemma}\label{K1lem}
For $K_{1t}$, there holds
\be
\left| \int_0^t  {K_{1t}}(t-s)g(s)ds\right|= V_h g(t)+O(\int_0^t e^{-\bar \eta (t-s)}|g(s)|ds) 
\le C\sup_{t/2\leq s\leq t}|g(s)| + e^{-\bar \eta t/2}|g|_{L^\infty(0,t)},
\ee
for all $t\ge 0$, some $C,\bar\eta>0$, and $g\in L^\infty(0,t)$.
\end{lemma}

\begin{proof}[Proof of Lemmas \ref{Glem}--\ref{K1lem}]
These follow by direct calculation from our more detailed pointwise Green function bounds,
exactly as in the proof of \cite[Lemma 7.1]{MZ}.
\end{proof}

\subsection{Linear phase estimates}\label{s:phase}

\begin{lemma}\label{intK1lem}
Assuming Evans-Lopatinsky stability, for $K_1$, there hold:
\be\label{K1bd}
K_1(t) =K_1(0)+V_h + O(1-e^{-\bar \eta t}), 
\ee
\be\label{intK1eq}
 \left| \int_0^t  {K_{1}}(t-s)g(s)ds\right| \le 
 C|g|_{L^1(0,t)},
 \ee
for all $t\ge 0$, some $C,\bar\eta>0$, and any $g\in L^1$.
\end{lemma}

\begin{proof}
\eqref{K1bd} follows by direct computation using Theorem \ref{pointwiseboutimeetaGreen}:
$$
K_1(t)=K_1(0)+\int_0^t K_{1t}(s)ds= K_1(0)+V_h +\int_0^t O(e^{-\bar \eta (t-s)})ds.
$$
Likewise, \eqref{intK1eq} follows immediately from \eqref{K1bd}, using
$|K_1(0)+V_h + O(1-e^{-\bar \eta t})|\leq C$. 
\end{proof}

\begin{lemma}\label{intG1lem}
	Assuming Evans-Lopatinsky stability, for $G_{1}$, there hold:
\be \label{G1bounded}
 |G_{1}(t;y)|\le C,\quad \text{where $C$ is independent of $t,y$,}
 \ee
\be \label{G1y}
 \left|\int_{-\infty}^\infty \partial_y G_{1}(t;z)f(z)dz\right|\le C(1+t)^{-\frac{1}{2q}}|f|_{L^q}+e^{-\bar{\eta}t}|f|_{L^\infty}+
 %|\psi(y)f(\cdot)|_{L^1},
|e^{-\theta|\cdot|}f(\cdot)|_{L^1},
 \ee
 and
 \be \label{G1v}
 \left|\int_{-\infty}^\infty G_{1}(t;z)\left(\begin{array}{c}0\\f(z)\end{array}\right)dz\right|\le C(1+t)^{-\frac{1}{2q}}|f|_{L^q}+
	% |\tilde{\psi}(y)f(\cdot)|_{L^1},
|e^{-\theta|\cdot|}f(\cdot)|_{L^1},
 \ee
 for all $t\ge 0$, some $C,\bar\eta>0$, for any $1\le q$ and $f\in L^q\cap L^\infty$. 
\end{lemma}

\begin{proof}
By equations \eqref{highG1lambda} and \eqref{defV}, in the high frequency regime $G_{1,\lambda}=O(\frac{1}{|\lambda|})$, 
hence there will be no $\delta$-function contribution to $G_{1}(t;y)$. 
Integrating the term $G_{1,\lambda}$ in the high frequency regime will contribute to $G_1(t;y)$ a time and space exponentially decaying term, which is harmless. In the low frequency regime, by shifting the integral contour to the left of the origin and applying the Residue theorem, we get a residue term that is independent of time and decaying in space plus a time exponentially small term. Therefore \eqref{G1bounded} follows.

By Proposition \ref{highG1y}, integrating the term $HY_\lambda$ in the high frequency regime will contribute to $\partial_yG_1(t;y)$ a term containing a $\delta$-function like $H_1(t;y)$ in Theorem \ref{pointwiseinttimeetaGreen}. 
Integrating such a term in space with $f(y)$ gives a contribution that
can be controlled by $e^{-\bar{\eta}t}|f|_{L^\infty}$. Integrating the term $(\partial_y G_{1,\lambda}-HY_\lambda)$ in the high frequency regime will contribute to $\partial_yG_1(t;y)$ a time and space exponentially decaying term, which is harmless. By Proposition \ref{lowG1y}, integrating the term $SY_{1,\lambda}$ in the low frequency regime will contribute to $G_1(t;y)$ a scattering term like $S_1(t;y)$ in Theorem \ref{pointwiseinttimeetaGreen}. Integrating such scattering term with $f(z)$ can be bound similarly as in \eqref{barG1t}. Integrating the term $SY_{2,\lambda}$ in the low frequency regime will result in a residue term 
%$\psi(y)$ 
that is decaying exponentially in space and independent of time. Hence, integrating this term with $f(z)$
gives a contribution that can be controlled by $|e^{-\theta|\cdot|}f(\cdot)|_{L^1}$.

By equation \eqref{highG1lambda}, in the high frequency regime, $G_{1,\lambda}=O(\frac{1}{|\lambda|})$, 
hence there will be no $\delta$-function contribution to $G_{1}(t;y)\left(\begin{array}{c}0\\f(z)\end{array}\right)$. 
The low frequency contribution can be estimated similarly as \eqref{G1y} by applying Proposition \ref{G1lambdav}.
\end{proof}

\subsection{Auxiliary estimates}\label{s:aux}
\begin{lemma}\label{Glemaux}
Assuming Evans-Lopatinsky stability, for the splitting $G=\bar G+H$, there hold:
\be\label{aux1}
\int_{0}^{t} \left| \int_{-\infty}^{+\infty} {G}(x,s;y)f(y)dy\right| \, ds \leq C |f|_{L^1\cap L^\infty},
\ee
\ba
\label{aux2}
&\int_0^t (1+s)^{-\frac{1}{2}}\left|\int_0^s    \int_{-\infty}^{\infty}\bar{G}(x, s-\tau;y)\left(\begin{array}{c}0\\f(y,\tau)\end{array}  \right)  dy  \, 
	d\tau \right|\, ds\\ 
\le&	C\int_0^t  (1+s)^{-\frac{1}{2}+\upsilon } |f(\cdot, s)|_{L^{2}} ds,
\ea
\ba
\label{aux3}
\int_0^t (1+s)^{-\frac{1}{2}}\left|\int_0^s \int_{-\infty}^{+\infty}\bar{G}_y(x,s-\tau ;y)f(y,\tau)dy  \, d\tau\right| \,ds
\leq C\int_0^t  (1+s)^{-\frac{1}{2} +\upsilon } |f(\cdot, s)|_{L^{2}} ds,
\ea
\ba
\label{aux4}
\int_0^t (1+s)^{-\frac{1}{2}}\left|\int_0^s \bar{G}(x,s-\tau ;0^\pm)f(\tau) \, d\tau\right|\, ds 
\leq C\int_0^t  (1+s)^{-\frac{1}{2} } |f( s)| ds,
\ea
\ba
\label{aux5}
\int_0^t (1+s)^{-\frac{1}{2}}\left|\int_0^s K(x,s-\tau )f(\tau) \, d\tau\right|\, ds
\leq C\int_0^t  (1+s)^{-\frac{1}{2} } |f( s)| ds,
\ea
and
\ba
\label{aux6}
\int_0^t (1+s)^{-1/2}\left|\int_0^s  \int_{-\infty}^{+\infty}H(x,s-\tau ;y)f(y,\tau)dy\, d\tau\right| \, ds
\leq C\int_0^t  (1+s)^{-\frac{1}{2} } |f(\cdot, s)|_{L^\infty } ds,
\ea
for all $t\ge 0$ and $x\gtrless 0$, any $\upsilon>0$, and some $C=C(\upsilon)>0$. 
\end{lemma}

\begin{proof}
	For estimate \eqref{aux1}, we estimate separately the contributions of terms $\bar G$ and $H$ in the splitting $G=\bar G+H$.
For $\bar G$, changing order of integration, and using the fact that $\int_{0}^{t}  |\bar{G}(x,s;y)| \, ds\leq C$, we obtain
$$
\begin{aligned}
	\int_{0}^{t} \left| \int_{-\infty}^{+\infty} \bar{G}(x,s;y)f(y)dy\right| \, ds &\leq
\int_{0}^{t}  \int_{-\infty}^{+\infty} |\bar{G}(x,s;y)||f(y)| dy \, ds \\
&= 
\int_{-\infty}^{+\infty}  |f(y)| \int_{0}^{t}  |\bar{G}(x,s;y)| \, ds \, dy 
 \leq C |f|_{L^1},
\end{aligned}
$$
In turn, bound $\int_{0}^{t}  |\bar{G}(x,s;y)| \, ds\leq C$ may be verified 
as in \cite{TZ}, by integrating a Gaussian moving with 
nonzero speed, hence passing transversally to the vertical line $y\equiv \const$.
Namely, denoting by $\theta(z,s)$ the moving Gaussian $\theta(z,s):=(s)^{-1/2}e^{-(z-as)^2/bs}$ with $z=x-y$,
	and $\Theta=\int_{-\infty}^z \theta$ its bounded (error function) antiderivative,
	we have 
	$$
	\Theta_s= - a\Theta_z + b/4 \Theta_{zz}= -a\theta + b/4\theta_z,
	$$
	from which we may bound 
	$$
	\begin{aligned}
	&\int_0^t |\bar G(x,s;y)|ds\leq C\int_0^t e^{-\bar\eta(|z|+s)}ds+ C\int_0^{t}\chi_{s\ge 1}\theta(z,s) ds\le C+ C\int_0^{t}\chi_{s\ge 1}\theta(z,s) ds\\
	&\int_0^{t}\chi_{s\ge 1}\theta(z,s) ds= -\frac{1}{a}\int_0^{t} \chi_{s\ge 1}\Theta_s(z,s)ds
	+\frac{b}{4a}\int_0^{t} \chi_{s\ge 1}\theta_z(z,s)ds.
	\end{aligned}
	$$
	The term $\int_0^{t} \chi_{s\ge 1}\Theta_s(z,s)ds= \Theta(z,t)-\Theta(z,1)$ is uniformly bounded.
	The term $\int_0^{t}\chi_{s\ge 1} \theta_z(z,s)ds$ may be reduced by a similar argument to a constant times
	$\Theta_z(z,t)-\Theta_z(z,1)= \theta(z,t)-\theta(z,1)$ plus a multiple of $\int_0^t \chi_{s\ge 1}\theta_{zz}(z,s)ds$; the latter
	can be bounded as
	$ |\int_0^t \chi_{s\ge 1}\theta_{zz}(z,s)ds|\leq \int_1^t (1+s)^{-3/2}ds\leq C_2 $. 
	The $H$ part can be directly obtained from taking $p=\infty$ in estimate \eqref{GlemH} and then integrating in time. This completes the proof
	of \eqref{aux1}.

	Estimates \eqref{aux2} and \eqref{aux3} are more delicate, based on the estimate
	%\ba\label{test}
	%\int_1^{t}(1+\tau+s )^{-1/2} & |\theta_y|(y,s) \, ds\leq \\
													   %&C\Big( (1+\tau+t)^{-1/2}\theta(y,t) + (1+\tau)^{-1/2} \theta(y,1) +(1+\tau+|y|)^{-1/2})
	%(1+|y|)^{-1/2}\chi_{[a, at]}(y)\Big)
	%\ea
	%SIMPLER:
	\ba
	\left|\label{test}
	\int_1^{t}s^{-1-\upsilon} e^{-\frac{(as-|\cdot|)^2}{bs}} \, ds\right|_{L^2}\leq C, 
	\ea
	for $a\neq 0$ and some $C=C(\upsilon)$ independent of $t$. To prove \eqref{test}, by symmetricity, it suffices to show
	\ba
	\label{test1}
	\left(
	\int_0^\infty\left|\int_1^{t}s^{-1-\upsilon} e^{-\frac{(as-z)^2}{bs}} \, ds\right|^2\,dz\right)^{\frac{1}{2}}\leq C. 
	\ea
	Based on the bound on the integrand function
	$$
	s^{-1-\upsilon} e^{-\frac{(as-z)^2}{bs}}\le e^{-\frac{(as-z)^2}{bs}}\frac{2|as-z|}{bs^{1.5+\upsilon}}+s^{-1-\upsilon} \chi_{{}_{2|as-z|<\sqrt{s}b}}
	=s^{-\upsilon}|\theta_z|(z,s)+s^{-1-\upsilon}\chi_{{}_{2|as-z|<\sqrt{s}b}},
	$$
	\eqref{test1} follows from 
	\be 
	\label{test2}
	\left|\int_1^{t} s^{-\upsilon}|\theta_z|(z,s) \, ds\right|_{L^2}\le C, \quad\text{and}\quad	\left(\int_0^\infty\left|\int_1^{t}\chi_{{}_{2|z-as|<\sqrt{s}b}}s^{-1-\upsilon}  ds\right|^2\,dz\right)^{\frac{1}{2}}\leq C.
	\ee

	Estimate \eqref{test2}[i] follows readily for $z\not \in [a,at]$ by rewriting
	$$\int_1^{t} s^{-\upsilon}|\theta_z|(z,s) \, ds\le\int_1^{t} |\theta_z|(z,s) \, ds = \Big|\int_1^{t} \theta_z(z,s) \, ds\Big|,$$
	then using $\theta_s=-a\theta_z+b/4\theta_{zz}$ to obtain
	$$
	\int_1^{t} \theta_z(z,s) \, ds=\int_1^{t}\left(-\frac{\theta_s}{a}+\frac{b}{4a}\theta_{zz}\right)\,ds=\frac{1}{a}\theta(z,1)-\frac{1}{a}\theta(z,t)+\int_1^{t}\frac{b}{4a}\theta_{zz}(z,s)\,ds
	$$
	where we see $|\theta(\cdot,t)|_{L^2},|\theta(\cdot,1)|_{L^2}\le C$ and 
	$$
	\left|\int_1^{t}\theta_{zz}(\cdot,s)\,ds\right|_{L^2}\le C\left|\int_1^{t}s^{-\frac{3}{2}}e^{-\frac{{\left(\cdot-as\right)}^2}{2bs}}\,ds\right|_{L^2}\le C\int_1^{t}\left|s^{-\frac{3}{2}}e^{-\frac{{\left(\cdot-as\right)}^2}{2bs}}\right|_{L^2}\,ds\le C\int_1^t s^{-\frac{5}{4}}ds\le C.
	$$
	For $z\in [a,at]$, setting 
	$\theta_c(z,s):=(s)^{-1/2}e^{-(z-as)^2/bcs}$, $1<c$ to be a Gaussian with larger support moving at the same speed as $\theta$, we may estimate
$|s^{-\upsilon}\partial_z \theta|\leq C |\partial_z \theta_c^{1+2\upsilon}|$, for $\upsilon<(c-1)/2$.
Calculation shows
$$
\partial_s (\theta_c^{1+2\upsilon})= - a(\theta_c^{1+2 \upsilon})_z + \frac{bc}{4}(\theta_c^{1+ 2\upsilon })_{zz} 
- \frac{bc\upsilon \Big((\theta_c^{1+2\upsilon})_z\Big)^2}{2(1+2\upsilon)\theta_c^{1+2\upsilon}}=:- a(\theta_c^{1+2 \upsilon})_z+F(z,s).
$$
Observing that
$$
\int_1^{t} |(\theta_c^{1+ 2\upsilon})_z|(z,s) \, ds = \left|\Big(\int_1^{z/a} -\int_{z/a}^t \Big)
(\theta_c^{1+2\upsilon})_z(z,s)  \, ds\right|,
$$
hence, integrating the principal contribution $-1/a\partial_s(\theta_c^{1+2\upsilon})$ results in terms bounded by
$$C\Big(\theta_c(z,1)^{1+2\upsilon} +   (1+|z|)^{-1/2-\upsilon}  + \theta_c(z,t)^{1+2\upsilon}\Big)\in L^2(z).$$
As for integrating the $F$ term, calculation shows
$$
F(z,s)=\frac{e^{-\frac{{\left(z-as\right)}^2(1+2\upsilon)}{bcs}}\left(1+2\upsilon\right)\left((z-as)^2-bcs\right)}{2bcs^{5/2+\upsilon}}=O(s^{-\frac{3}{2}}e^{-\frac{{\left(z-as\right)}^2}{2bcs}}).
$$
Therefore, by the triangle inequality, also
$$
\left|\Big(\int_1^{z/a} -\int_{z/a}^t \Big)
F(\cdot,s)  \, ds\right|_{L^2}\le C.
$$
Estimate \eqref{test2}[ii] follows by direct computation
$$
\begin{aligned}
\int_0^\infty\left|\int_1^{t}\chi_{{}_{2|z-as|<\sqrt{s}b}}s^{-1-\upsilon}  ds\right|^2\,dz\le&\int_0^\infty\left|\int_1^{\infty}\chi_{{}_{2|z-as|<\sqrt{s}b}}s^{-1-\upsilon} ds\right|^2\,dz\\
\le& C\int_0^\infty\left( \frac{s_1(z)-s_2(z)}{(1+s_2(z))^{1+\upsilon}}\right)^2\,dz\le C
\end{aligned}
$$
where $s_{1,2}=(8az+b^2\pm \sqrt{16azb^2+b^4})/(8a^2)$ are the two roots of equation $2|z-as|=\sqrt{s}b$ and the last inequality follows from 
$$\left( \frac{s_1(z)-s_2(z)}{(1+s_2(z))^{1+\upsilon}}\right)^2\sim z^{-1-2\upsilon}\quad \text{for $z\gg 1$}.
$$

By the decomposition of $\bar G_y=S^1_y+R_y$ and bounds on $S^1_y$, $R_y$ from Theorem \ref{pointwiseintvGreen}, we obtain
\ba 
\label{terminL2}
&\int_\tau^{t}(1+s)^{-1/2} |\bar G_y(x,s -\tau;y)| \, ds
\leq (1+\tau)^{-1/2+\upsilon} \int_0^{t-\tau} (1+\tau+\sigma)^{-\upsilon} |\bar G_y(x,\sigma; y)| \, d\sigma \\
\leq& C(1+\tau)^{-1/2+\upsilon} \int_1^{t-\tau} (1+\tau+\sigma)^{-\upsilon}\sigma^{-1}e^{-\frac{(a\sigma-|z|)^2}{b\sigma}}  \, d\sigma \leq C(1+\tau)^{-1/2+\upsilon} \int_1^{t-\tau} \sigma^{-1-\upsilon}e^{-\frac{(a\sigma-|z|)^2}{b\sigma}}  \, d\sigma.
\ea
Switching orders of integration and applying \eqref{terminL2}, Holder's inequality, and \eqref{test} yields
$$
\begin{aligned}
	&\int_{0}^{t} (1+s)^{-1/2}\left|\int_0^s  \int_{-\infty}^{+\infty} \bar{G}_y(x,s-\tau;y)f(y,\tau)\, dy\,d\tau\right| \, ds \\
\leq  & \int_{0}^{t} (1+s)^{-1/2}\int_0^s \int_{-\infty}^{+\infty} |\bar{G}_y(x,s-\tau;y)| \, 
|f(y,\tau)| \, dy\,d\tau\,ds \\
=& \int_{0}^{t} \int_{-\infty}^{+\infty}|f(y,\tau)| \Big(\int_\tau^t   (1+s)^{-1/2}|\bar{G}_y(x,s-\tau;y)| \, 
ds \Big) \, dy\, d\tau \\
\leq& C \int_{0}^{t} (1+\tau)^{-1/2+\upsilon}\int_{-\infty}^{+\infty}|f(x-z,\tau)| 
\int_1^{t-\tau} \sigma^{-1-\upsilon}e^{-\frac{(a\sigma-|z|)^2}{b\sigma}}  \, d\sigma \, dz\, d\tau   \\
\le&C \int_{0}^{t}(1+\tau)^{-1/2+\upsilon} |f(\cdot,\tau)|_{L^2}\, d\tau\,,\\
%%%%%%%%%%
\end{aligned}
$$
verifying \eqref{aux3}. By equation \eqref{s1second} and estimate \eqref{R2bound}, estimate \eqref{aux2} follows similarly.

Using $\int_0^t|\bar G(x,t;y)|\leq C$, we have, switching the order of integration,
$$
\begin{aligned}
	&\int_0^t (1+s)^{-1/2}\left|\int_0^s  \bar{G}(x,s-\tau ;0^\pm)f(\tau) \, d\tau\right|\, ds \leq
	\int_0^t \int_0^s  (1+s)^{-1/2}\left|\bar{G}(x,s-\tau ;0^\pm)f(\tau)\right| \, d\tau\, ds \\
	=&
	C\int_0^t(1+\tau)^{-1/2} |f(\tau)|\int_\tau^t \left|\bar{G}(x,s-\tau ;0^\pm)\right| \, ds\, d\tau\\
	\leq&
	C\int_0^t(1+\tau)^{-1/2} |f(\tau)|\int_0^{t-\tau} \left|\bar{G}(x,\sigma ;0^\pm)\right| \, d\sigma\, d\tau\leq
	C\int_0^t(1+\tau)^{-1/2} |f(\tau)|\, d\tau \\
\end{aligned}
$$
verifying \eqref{aux4}.  
The proof of \eqref{aux5} goes similarly by applying Lemma \ref{Klem}.  
Likewise, applying \eqref{GlemH} with $p=\infty$, we have
$$
\int_0^t (1+s)^{-1/2}\int_0^s \left| \int_{-\infty}^{+\infty}H(x,s-\tau ;y)f(y,\tau)dy\right|\, d\tau \, ds\leq
C\int_0^t (1+s)^{-1/2}\int_0^s e^{-\bar \eta(s-\tau)} |f(\cdot,\tau)|_{L^\infty} \, d\tau \, ds,
$$
which, switching the order of integration, yields bound
$$
\begin{aligned}
	&\int_0^t (1+s)^{-1/2}\int_0^s e^{-\bar \eta(s-\tau)} |f(\cdot,\tau)|_{L^\infty} \, d\tau \, ds=\int_0^t \int_\tau^t (1+s)^{-1/2} e^{-\bar \eta(s-\tau)} |f(\cdot,\tau)|_{L^\infty} \, ds \, d\tau\\=& \int_0^t  (1+\tau)^{-1/2} |f(\cdot,\tau)|_{L^\infty} 
\int_\tau^t e^{-\bar \eta(s-\tau)}  \, ds  \, d\tau\leq  C \int_0^t  (1+\tau)^{-\frac{1}{2} } |f(\cdot, \tau)|_{L^\infty } d\tau,
\end{aligned}
$$
verifying \eqref{aux6}. 
\end{proof}

\br\label{transversermk}
The above ``Strichartz-type'' bounds make crucial use of transverse propatation and pointwise bounds.
By contrast, a straightforward estimation by Holder's inequality 
$$
	\int_{0}^{t} \left| \int_{-\infty}^{+\infty} \bar{G}(x,s;y)f(y)dy\right| \, ds \leq 
\int_{0}^{t} |\bar{G}(x,s;y)|_{L^\infty} |f|_{L^1} \, ds \leq 
(\int_{0}^{t} (1+s)^{-1/2} \, ds) |f|_{L^1}
$$
yields bound $C(1+t)^{1/2}|f|_{L^1}$, poorer by factor
$(1+t)^{1/2}$ than \eqref{aux1}.
Similarly for the cases \eqref{aux2}--\eqref{aux3}, straightforward estimates by Holder's inequality yield bounds poorer by factor $(1+t)^{1/4-\upsilon}$. 
\er

\section{Short-time existence and nonlinear damping estimate}\label{s:damping}
We next establish nonlinear existence and damping estimates, obtained by Kreiss symmetrizer and Kawashima type energy
estimates, respectively.
As noted in \cite{JLW}, short time existence theory may be concluded by the analysis of shock
stability carried out by Kreiss symmetrizer techniques in \cite{Ma,Me}.
Denote by $\tilde R_{a}$ the punctured real line $(-\infty,a)\cup (a ,+\infty)$, and $\tilde R$
the symmetric version $\tilde R_0$.
We obtain by the results of \cite{Me} immediately the following short time existence theory.\footnote{
Note that we correct a minor typo in \cite[Thm. 4.1.5]{Me}, which requires data $v_0$ only in $H^s$ rather than $H^{s+1/2}$.
(This is not necessary for our later analysis, but only sharpens our initial regularity assumptions.)
}

\bpr\label{existence}
For $0<F<2$ and $0<H_R< H_L \frac{2F^2}{1+2F+\sqrt{1+4F}}$, let $\overline{W}=(H,Q)$ be a hydraulic shock \eqref{prof},
and $v_0$ be a perturbation supported away from the subshock discontinuity of $\overline{W}$ and lying
%TODO: typo in Me.
%in $H^{s+1/2}(\tilde R_{\eta_0})$, $s\geq 2$. 
in $H^{s}(\tilde R))$, $s\geq 2$. 
Moreover, assume that $\overline{W}$ is \emph {spectrally stable} in the sense of the Evans-Lopatinsky condition
defined in Section \ref{s:lop}. 
Then, for initial data $\tilde W_0:=\overline{W}+v_0$, there exists a unique solution of \eqref{sv} defined for $0\leq t\leq T$, 
for some $T>0$,  with a single shock located at $ct-\eta(t)$, and $H^s$ to either side of the shock, such that
for $v(x,t):= \tilde W(x+ct -\eta(t),t)-\overline{W}(x)$,
\be\label{shortreg}
(v,\eta)\in C^0\left([0,T]; H^s(\tilde \R)\right) \times C^{s+1}([0,t]).
\ee
Moreover, the maximal time of existence $T_*$ defined as the supremum of $T>0$ for which the solution is defined
is either $+\infty$ or satisfies 
%NOTE: $\lim_{t\to T_*^-}|v|_{W^{1,\infty}(\tilde R_\eta)} =+\infty.$
$\lim_{t\to T_*^-}|v|_{W^{1,\infty}} =+\infty.$
Finally, if $v_0^r\to v_0$, with $v_0^r\in H^{r+1/2}$, $r\geq s$, then the corresponding solutions $(v^r,\eta^r)$
converge to $(v,\eta)$ in $C^0\left([0,T]; H^s(\tilde \R)\right) \times C^{s+1}([0,t])$.
\epr

\begin{proof}
Noting that the assumption that the perturbation is supported away from the subshock implies compatibility to
all orders, we have that the first two assertions follow from Theorems 4.15 and 4.16 of \cite{Me}, 
provided that the subshock satisfies the (shock) Lopatinsky condition of Majda \cite{Ma}.  
But (see Remark \ref{subloprmk}), Lopatinsky stability of
the component subshock is implied in the high-frequency limit
by the Evans-Lopatinsky condition for the full shock profile.
The third assertion, though not explicitly stated in \cite[Thm. 4.1.5]{Me}, is established in the course of its proof.
\end{proof}

\br
In fact, the subshock can be seen directly to satisfy Majda's Lopatinsky condition, 
independent of Evans-Lopatinsky stability of the associated relaxation shock profile, 
by the fact \cite{Ma} that shock waves of isentropic gas dynamics are stable, 
since the shock Lopatinsky condition depends only on the first-order part of \eqref{sv}.
\er

Our main effort will be devoted to proving the following nonlinear damping estimate generalizing
the one proved for smooth relaxation profiles in  \cite[Prop. 1.4]{MZ2}. 

\bpr\label{damping}
Under the assumptions of Proposition \ref{existence}, 
suppose that, for $0\le t\le T$,
$|v(\cdot,s)|_{H^s(\tilde R)}$ and $|\dot \eta|$ are bounded by a sufficiently small constant $\zeta>0$.
Then, for all $0\le t\le T$ and some $\theta>0$,
\be
\label{dampest}
 |v|_{{}_{H^s}}^2(t)\le Ce^{-\theta t} |v_0|_{{}_{H^s}}^2+
 C\int_0^te^{-\theta(t-\tau)} \big( |v|_{{L^2}}^2 + |\dot \eta|^2 \big)(\tau)d\tau.
\ee
\epr

\br\label{removermk}
In the course of the proof, we show using the Rankine-Hugoniot conditions at the subshock
that $|\dot \eta|^2$ is controlled by a bounded linear sum of trace terms $|v(0^\pm)|^2$ at $\xi=0$.
By one-dimensional Sobolev embedding, these in turn are controlled by a lower-order term
$C|v|_{H^1}^2$ absorbable in the estimates \eqref{tdiffdamp} from which \eqref{dampest} is obtained, hence \eqref{dampest} could be improved to
\be \label{idampest}
 |v|_{{}_{H^s}}^2(t)\le Ce^{-\theta t} |v_0|_{{}_{H^s}}^2+
 C\int_0^te^{-\theta(t-\tau)} |v|_{{L^2}}^2 (\tau)d\tau,
\ee
slightly improving the estimate of the smooth case \cite{MZ2}. 
\er

\br\label{engen}
For clarity, we carry out the proof of Proposition \ref{damping} for shock profiles of
\eqref{sv}; however, the argument applies more generally to profiles of general relaxation systems
of the class considered in \cite{MZ2}, provided they contain a \emph{single subshock}.
(This allows the freedom to initialize the symmetrizer $A^0$ arbitrarily at the endstates of the subshock,
as we use crucially to arrange maximal dissipativity with respect to $A^0$ of the 
associated Rankine-Hugoniot conditions.)
A very interesting open problem would be to develop corresponding
damping estimates in multi-dimensions, perhaps by Kreiss symmetrizer techniques \cite{Kr};
see Remark \ref{multidrmk} for related discussion.
\er

\subsection{Preliminaries}\label{s:prelim}

Under the assumptions of Proposition \ref{damping}, the equations \eqref{sv} and profiles $\overline{W}$ satisfy
the structural assumptions made for general relaxation systems in \cite{MZ2} along the smooth
portions of $W$, i.e., everywhere except at the subshock at $x=0$.
Thus, we have the following results of \cite{MZ2}, denoting by $\Re M:=\frac12(M+M^*)$ the symmetric
part of a matrix $M$.

\begin{lemma}\label{lem:expdecay}
Under the assumptions of Proposition \ref{damping},
for some $\theta>0$, and all $k\geq 0$,
\begin{equation}
	\hbox{\rm 
	$|(d/dx)^k (\overline{W}-\overline{W}_\pm) | \le C|\overline{W}_x|\le Ce^{-\theta |x|}$ as
$x\to \pm \infty$.}
\label{expdecay}
\end{equation}
\end{lemma}

(Stable manifold theorem, plus hyperbolicity of rest points $\overline{W}_\pm$ of \eqref{profileODE}.)

\begin{lemma}[\cite{MZ2,Hu}]\label{lem:skew}
Let $D$ be diagonal, with real entries appearing with prescribed 
multiplicity in order of increasing size, and let $E$ be arbitrary.
Then, there exists a smooth skew-symmetric matrix-valued function 
$K(D,E)$ such that
 $$
  \Re \left(E- KD\right)=\Re \, \diag \,  E,
 $$
where $\diag\,E$ denotes the diagonal part of $E$.
\end{lemma}

\begin{lemma}[\cite{MZ2}]\label{lem:LQR}
Under the assumptions of Proposition \ref{damping}, there exist diagonalizing matrices $L_\pm$, $R_\pm$, $(LAR)_\pm$ diagonal, $(LR)_\pm=I$, such that
$$
\Re\;\diag \,(LER)_\pm<0.
$$
\end{lemma}

\begin{lemma}[\cite{MZ2}]\label{lem:symmdiag}
There is a correspondence between symmetric positive definite
symmetrizers $A^0$, $A^0 A$ symmetric, and diagonalizing transformations
$L$, $R$, $LAR$ diagonal, given by $A^0 = L^* L$,
or equivalently $L=O^*(A^0)^{\frac{1}{2}}$, where $O$ is an orthonormal matrix
diagonalizing the symmetric matrix $(A^0)^{\frac{1}{2}}A(A^0)^{-\frac{1}{2}}$.
Moreover, the matrix $O$ (or equivalently $L$) may be chosen with the 
same degree of smoothness as $A^0$, on any simply connected domain.
\end{lemma}

Following \cite{MZ2}, we recall also the relations
\begin{equation}
	\langle W, SW_x\rangle_{(a,b)}= -\frac{1}{2}\langle W,S_x W\rangle_{(a,b)} +\frac 12 W\cdot SW |_a^b,
\label{symmetric}
\end{equation}
and
\begin{equation}
	\partial_t \frac{1}{2}\langle W_{x}, KW\rangle_{(a,b)}= \langle W_{x},K W_{t}\rangle_{(a,b)}
	+\frac{1}{2} \langle W_{x},K_t W\rangle_{(a,b)}
	+\frac{1}{2} \langle W,K_x W_{t}\rangle_{(a,b)}
	+\frac{1}{2} W_t \cdot K W|_a^b,
\label{antisymmetric}
\end{equation}
here adapted to the case of a domain with boundary, where
$S$ is symmetric and $K$ skew-symmetric, and $\langle \cdot, \cdot\rangle_{(a,b)}$ denotes $L^2$ inner product
on $(a,b)$. When the domain $(a,b)$ is clear (as below, where all energy estimates will be carried out on
$(-\infty,0)$ and $(0,\infty)$), we omit the subscript $(a,b)$.

\subsubsection{Boundary dissipativity}\label{s:bdiss}
A new aspect in the present, discontinuous, case is boundary dissipativity at the subshock.
For a general symmetrizable initial boundary-value problem on $(-\infty, 0]$
\ba\label{ibvp}
 \mathbb{V}_t + \mathbb{A} \mathbb{V}_x &= \mathbb{F}, \quad x\in (-\infty, 0), \\
\mathbb{B}\mathbb{V}&= \mathbb{G},\quad x=0,
\ea
with symmetrizer $\mathbb{A}^0$ symmetric positive definite and $\mathbb{A}^0\mathbb{A}$ symmetric,
that is noncharacteristic in the sense that $\det (\mathbb{A})\neq 0$ at the boundary $x=0^-$,
is {\it Lopatinsky stable} in the sense of Kreiss \cite{Kr} if $\mathbb{B}$ is full rank on the stable subspace of $\mathbb{A}$.
It is {\it maximally dissipative} with respect to the symmetrizer $\mathbb{A}^0$ if $\mathbb{A}^0 \mathbb{A}$
is positive definite on $\ker \mathbb{B}$, which yields readily the following key consequence.

\bl\label{bcestlem}Suppose that \eqref{ibvp} has
maximally dissipative boundary conditions with respect to symmetrizer $\mathbb{A}^0$.
Then, for some $\theta, C>0$,
\be\label{bcest}
-\mathbb{V}(0^-)\cdot \mathbb{A}^0 \mathbb{A} \mathbb{V}(0^-) \leq 
%-\theta  |P v(0^-)|^2 + C |g|^2 \leq C|g|^2.
-\theta  | \mathbb{V}(0^-)|^2 + C |\mathbb{G}|^2 .
\ee
\el

\begin{proof}
	Decompose $\mathbb{V}(0^-)=v_{ker}+v_{\perp}$, where $v_{ker}\in \ker \mathbb{B}$ and $v_{\perp}\in (\ker \mathbb{B})^\perp$. Then, 
$$
	-\mathbb{V}(0^-)\cdot \mathbb{A}^0 \mathbb{A} \mathbb{V}(0^-)=
-v_{ker} \cdot \mathbb{A}^0 \mathbb{A} v_{ker} -2v_{ker} \cdot \mathbb{A}^0 \mathbb{A} v_{\perp} -v_{\perp} \cdot \mathbb{A}^0 \mathbb{A} v_{\perp}.
$$
By maximal dissipativity, $-v_{ker} \cdot \mathbb{A}^0 \mathbb{A} v_{ker} \leq -\theta_1 |v_{ker}|^2$. Using Young's inequality, the middle cross term is bounded by $|2v_{ker} \cdot \mathbb{A}^0 \mathbb{A} v_{\perp}|\leq \theta_1|v_{ker}|^2/2
+ C_1 |v_{\perp}|^2$. The last term is bounded by $|v_{\perp} \cdot \mathbb{A}^0 \mathbb{A} v_{\perp}|\le C_2|v_{\perp}|^2$. Summing these estimates, we obtain \be 
\label{vA0Av}
-\mathbb{V}(0^-)\cdot \mathbb{A}^0(0^-) \mathbb{A}(0^-) \mathbb{V}(0^-)\le-\frac{\theta_1}{2}|v_{ker}|^2
+ (C_1+C_2)|v_\perp|^2
\ee
From the fact that $\mathbb{B}$ is full rank on $(\ker \mathbb{B})^\perp$ and the boundary condition, 
$$
\theta_2 |v_{\perp}|\le|\mathbb{B}v_{\perp}|=|\mathbb{B}\mathbb{V}(0^-)|=|\mathbb{G}|.
$$
Therefore, $-|v_{ker}|^2=-(|v_{ker}|^2+|v_{\perp}|^2)+|v_{\perp}|^2=-|\mathbb{V}(0^-)|^2
+|v_\perp|^2\le-|\mathbb{V}(0^-)|^2
+|\mathbb{G}|^2/\theta_2^2$. Substituting in \eqref{vA0Av}, we obtain
$$
-\mathbb{V}(0^-)\cdot \mathbb{A}^0 (0^-)\mathbb{A}(0^-) \mathbb{V}(0^-)\le-\frac{\theta_1}{2}|\mathbb{V}(0^-)|^2
+ \frac{C_1+C_2+\theta_1/2}{\theta_2^2}|\mathbb{G}|^2,
$$
which yields \eqref{bcest}.
\end{proof}

The Lopatinsky condition is necessary and sufficient for maximal $L^2$ estimates \cite{Kr}.
For maximally dissipative boundary conditions, maximal $L^2$ estimates may be obtained by
taking the $L^2$ inner product of $v$ against the symmetrized equation $A^0v_t + A^0 Av_x=A^0f$ and
applying \eqref{symmetric}, \eqref{bcest}.  Thus, maximally dissipative boundary conditions are always Lopatinsky stable.
The following result shows that the converse is true as well, {\it for some choice of symmetrizer $A^0$}.

\bl\label{disslem}
For any symmetrizable initial boundary-value problem \eqref{ibvp} that is Lopatinsky stable, there exists a symmetrizer $A^0$ with respect to which \eqref{ibvp} is maximally dissipative.
\el

\begin{proof}
Equivalently, $M=\tilde b^T A^0A \tilde b$ is positive definite, where 
$\tilde b\in \R^{(n-r)\times n}$ is a matrix whose columns span $\ker b$.
Let $A=S^{-1}\bdiag\{\Lambda^-, \Lambda^+\} S$, 
$$
\Lambda^-=\diag \{\lambda_1, \dots, \lambda_r\}, \quad \Lambda^+=\diag\{\lambda_{r+1}, \dots, \lambda_k\},
$$
with $\lambda_1\le \dots \lambda_r<0<\lambda_{r+1} < \dots \le \lambda_k$,
and set $A^0=S^T \diag \{ a,\dots, a, 1, \dots, 1\} S$, $a>0$. 
Then, $A^0 A$ is symmetric and $M=- M_1a+ M_2$ with
$ M_1=  \tilde b^T S^T E_1 S\tilde b$, $M_2= \tilde b^T  S^T E_2 S \tilde b $,
where $E_1:=\bdiag \{-\Lambda^-, 0\}$ and $E_2:= \bdiag \{0, \Lambda^+\}$ 
are symmetric positive semidefinite.  Thus, we may achieve $M>0$ for $a$ sufficiently small
if and only if $M_2$ is positive definite, or $\tilde b\cap \ker E_2 S=\emptyset$:
equivalently, $b$ is full rank on the stable subspace $\ker E_2 S$ of $A$.
\end{proof}

\br\label{multidrmk}
Lemma \ref{disslem} is special to one spatial dimension.
A generalization to multi-dimensions is given by the (pseudodifferential)
frequency-dependent symmetrizers of Kreiss \cite{Kr,BS}.
\er

\subsection{Energy estimates.}
We are now ready to carry out the main energy estimates, adapting the argument of \cite{MZ2}.
Define the nonlinear perturbation
$v(x,t):= \tilde W (x+ct-\eta(t),t)- \overline{W}(x)$ as in the statement of Proposition \ref{damping},
where $ct-\eta(t)$ denotes subshock location; for definiteness, fix without loss of generality $\eta(0)=0$.  
As computed in \cite[Eq. (3.1) p. 87]{MZ2} the interior equation \eqref{lineareq}(i) for $v$ may be put in 
the alternate quasilinear form 
\begin{equation}
{v}_t+\tilde Av_{x}-\tilde E v=M_1(v)\overline{W}_x+ M_2(v)
-\dot \eta (t)(\overline{W}_x + v_{x}),
\label{strat}
\end{equation}
where $\tilde A:= dF(\tilde W(x+ct-\eta(t),t))-c\Id,$ $\tilde E:=dR(\tilde W(x+ct-\eta(t),t))$ and
$$
M_1(v):= A(x)-\tilde A(x,t)=O(|v|),
\qquad
M_2(v)=\bp 0\\ O(|v|^2)\ep.
$$

Following \cite{MZ2}, let $\tilde A^0:=A^0(\tilde W(x+ct-\eta(t),t)$ denote a symmetrizer of $\tilde A$
as guaranteed by Lemma \ref{lem:symmdiag} with values $A^0(\tilde{W}(0^\pm+ct-\eta(t),t))$ to be specified later, and factor
$\tilde A^0 \tilde A= (\tilde A^0)^{\frac{1}{2}}\tilde O \tilde D \tilde O^t
(\tilde A^0)^{\frac{1}{2}}$, or, equivalently, 
$ \tilde A= (\tilde A^0)^{-\frac{1}{2}} \tilde O \tilde D \tilde O^t (\tilde A^0)^{\frac{1}{2}}$,
where $\tilde O$ is orthogonal, $\tilde O^t=\tilde O^{-1}$, and $C^3$
as a function of $(u,v)$ (see Lemma \ref{lem:symmdiag})
and $\tilde D=\diag\{\tilde a_1, \tilde a_2\}$, where $\tilde a_j$
denote the eigenvalues of $\tilde A$, indexed in increasing order.
Define the weighting matrix $\alpha(x):= \diag\{ \alpha_1, \alpha_2\}$,
where $\alpha_j>0$ are defined by ODE
 \be 
 \label{defalpha}
  (\alpha_j)_x= -C_* \sgn{a_j} |\overline{W}_x|\alpha_j,
  \qquad\quad \alpha_j(0)=1,
 \ee 
$C_*>0$ a sufficiently large constant to be determined later, and set
\begin{equation}
\tilde A^0_\alpha:=(\tilde A^0)^{\frac{1}{2}}\tilde O\alpha\tilde O^t (\tilde A^0)^{\frac{1}{2}}.
\label{A0alpha}
\end{equation}

Let $ K_1:=K\big(2\tilde D,2\alpha\tilde O^t (\tilde A^0)^{\frac{1}{2}}\tilde E (\tilde A^0)^{-\frac{1}{2}}\tilde O + N \big), $
where $K(\cdot)$ is as in Lemma \ref{lem:skew}, 
and $N$ is an arbitrary matrix with $|N|_{{}_{C^1_{x,t}}}\le C(C_*)$
and vanishing on diagonal blocks, to be determined later, and set
\be\label{Kalpha}
\tilde K_\alpha:= (\tilde A^0)^{\frac{1}{2}}\tilde O K_1 \tilde O^t (\tilde A^0)^{\frac{1}{2}}.
\ee
Finally, define
 \begin{equation}
  {\mathcal E}(v):=\langle \tilde A^0_\alpha v_{xx},v_{xx}\rangle
   +\langle v_{xx},\tilde K_\alpha v_x \rangle
   +M|v|_{{}_{L^2}}^2.
 \label{EofU}
 \end{equation}
 for $M>0$.
Since, for $v\in H^2$, $|v_x|_{{}_{L^2}}$ can be bounded by 
$C\left(|v|_{{}_{L^2}}+|v_{xx}|_{{}_{L^2}}\right)$ for some 
$C>0$, then the functional defined in (\ref{EofU}) is equivalent 
to $|v|^2_{{}_{H^2}}$ if $M$ is large enough.

Assume without loss of generality that $v_0\in H^3$ (since we may pass to the $H^2$ limit by Proposition \ref{existence}).
Then, following to the letter the computations of \cite{MZ2}, we obtain using \eqref{symmetric}--\eqref{antisymmetric}
the key estimate
\be\label{diffdamp}
  \frac{d{\mathcal E}}{dt}
  \leq -\theta{\mathcal E}+C(|v|_{{}_{L^2}}^2 + |\dot\eta(t)|^2) + [v_{xx}\cdot \tilde A^0_\alpha \tilde Av_{xx}]+\dot\eta(t)[v_{xx}\cdot\tilde{A}^0_\alpha v_{xx}]
  -[v_{xt} \cdot\tilde{K}_\alpha v_x]
 \ee
 for some $C,\theta>0$, where the terms $[v_{xx}\cdot \tilde A^0_\alpha \tilde Av_{xx}]$, $\dot\eta(t)[v_{xx}\cdot\tilde{A}^0_\alpha v_{xx}]$
 and  $-[v_{xt} \cdot \tilde{K}_\alpha v_x]$ arising through integration by parts at the boundary $x=0$  of 
 $\langle v_{xx}, \tilde A^0_\alpha \tilde A v_{xxx}\rangle$, $\dot\eta(t)\langle\tilde{A}^0_\alpha v_{xx},v_{xxx}\rangle$, and $\langle v_{xx},\tilde{K}_\alpha v_x\rangle$ through \eqref{symmetric} and \eqref{antisymmetric}
 are the sole differences from the whole-line estimate of \cite{MZ2}.

 \begin{proof}[Proof of Proposition \ref{damping}]
For clarity, we carry out the proof for the lowest level of regularity $s=2$; higher orders $s>2$ go similarly.
Starting with the $H^2$ estimate \eqref{diffdamp},
it remains only to show that the new trace terms $[v_{xx}\cdot \tilde A^0_\alpha \tilde Av_{xx}]$, $\dot{\eta}(t)[v_{xx}\cdot\tilde{A}^0_\alpha v_{xx}]$ 
and  $-[v_{xt} \cdot \tilde{K}_\alpha v_x]$ in the righthand side
may be absorbed in other terms, after which, multiplying by $e^{\theta t}$ and integrating in time from $0$ to $t$ 
as in \cite{MZ2}, we obtain \eqref{dampest}, completing the proof.
To this end, recall the nonlinear boundary condition \eqref{lineareq}(ii) at $x=0^\pm$, written
in the alternative form  
\be\label{repeatbc}
	 \eta_{{t}}[\tilde{W}]+[\tilde{A}v]= -[N_1(v,v)] =O\big(|v(0^\pm)|^2\big),
\ee
where now $N_1(v,v):=F(\tilde{W})-F(W)-(\tilde{A}+c Id)v$.

To obtain boundary conditions for $v_{xx}$, we may differentiate \eqref{repeatbc} with respect to $t$,
then convert any $t$-derivatives of $v$ into $x$-derivatives using the interior equations.
%To obtain boundary conditions for $v_{xx}$, 
Namely, we may first differentiate \eqref{strat} with respect to $t$, $x$ and differentiate \eqref{repeatbc} 
with respect to $t$ to get estimates
\ba 
\label{esti1}
v_t(0^\pm)=&O(|v_x(0^\pm)|+|v(0^\pm)|+|\eta_t|)\le C\zeta,\\
v_{tx}(0^\pm)=&-\tilde{A}v_{xx}(0^\pm)+O(|v_x(0^\pm)|+|v(0^\pm)|+|\eta_t(0^\pm)|+\zeta|v_{xx}(0^\pm)|),\\
v_{tt}(0^\pm)=&\tilde{A}^2v_{xx}(0^\pm)+O(|v_x(0^\pm)|+|v(0^\pm)|+|\eta_t|+\zeta |v_{xx}(0^\pm)|),\\
\eta_{tt}=&\frac{[\tilde{W}]^T\left(-\eta_t[v_t]-[\tilde{A}_tv]-[\tilde{A}v_t] -\left[\frac{d N_1(v,v)}{dt}\right]\right)}{[\tilde{W}]^T[\tilde{W}]}\\
=&O(|v_x(0^\pm)|+|v(0^\pm)|+|\eta_t|)\le C\zeta \, ,\\
\ea 
where $\zeta$ is the small constant chosen in Proposition \ref{damping}.
Then, by differentiating \eqref{repeatbc} with respect to $t$ twice and applying estimates \eqref{esti1}, we get the second-order boundary conditions
\be 
\label{Avxx_1}
\eta_{ttt}[\tilde{W}]+[\tilde{A}^3v_{xx}]=g=O(|v_x(0^\pm)|+|v(0^\pm)|+|\eta_t|+\zeta |v_{xx}(0^\pm)|), \quad x=0.
\ee 
As noted in \cite{Ma,Me}, a key point in dealing with the transmission problem \eqref{strat} coupled with boundary condition \eqref{Avxx_1} is that one may eliminate the front variable $\eta_{ttt}$,
converting the boundary condition \eqref{Avxx_1} to a standard boundary condition
\be\label{redcond}
M[\tilde{A}^3v_{xx}]= M g, \quad x=0,
\ee
where $M$ is a row vector which spans the subspace  $[\tilde{W}]^\perp $.

Another key point \cite{Ma,Me} is that one may double the coordinates and convert the transmission problem to a conventional half-line problem. That is, for $x\in(-\infty,0)$, defining  $$ \mathbb{V}(x,t):=\left[\begin{array}{c}v(x,t)\\v(-x,t)\end{array}\right],\quad\chi(t):=\left[\begin{array}{rr}\eta(t)Id&0\\0&-\eta(t)Id\end{array}\right],\quad  \mathbb{A}(x,t):=\left[\begin{array}{rr}\tilde{A}(x,t)&0\\0&-\tilde{A}(-x,t)\end{array}\right],$$ 
and similarly defining doubling matrices $\mathbb{E}$, $\mathbb{W}$, $\mathbb{M}_1$, and $\mathbb{M}_2$, in the doubling coordinates, the interior equation \eqref{strat} reduces to a equation on half-line: 
\be 
\label{mathbbinterior}
\mathbb{V}_t+\mathbb{A}\mathbb{V}_x=\mathbb{E}\mathbb{V}+\mathbb{M}_1(\mathbb{V})\mathbb{W}_x+\mathbb{M}_2(\mathbb{V})-\chi(t)(\mathbb{W}_x+\mathbb{V}_x),\quad x\in(-\infty,0),\\
\ee
from which we deduce
\be 
\label{mathbbinteriorxx}
(\mathbb{V}_{xx})_t+\mathbb{A}(\mathbb{V}_{xx})_x=\Big(\mathbb{E}\mathbb{V}+\mathbb{M}_1(\mathbb{V})\mathbb{W}_x+\mathbb{M}_2(\mathbb{V})-\chi(t)(\mathbb{W}_x+\mathbb{V}_x)\Big)_{xx}-2\mathbb{A}_x\mathbb{V}_{xx}-\mathbb{A}_{xx}\mathbb{V}_{x}.\\
\ee
In the doubling coordinates, the second-order boundary conditions \eqref{redcond} becomes 
\be
\label{doubleboundary}
\mathbb{B}\mathbb{V}_{xx}(t,0^-)=-M[\tilde{A}^3v_{xx}]=:\mathbb{G},
\ee 
where $\mathbb{B}:=\left(\begin{array}{rr}M&M\end{array}\right)\mathbb{A}^3(t,0^-)$. We then see
that the half line problem \eqref{mathbbinteriorxx}-\eqref{doubleboundary} is of the form \eqref{ibvp} with $\mathbb{V}_{xx}$ in the place of $\mathbb{V}$ and the previous Kreiss theory \cite{Kr} may be applied.

\bc\label{disscor} 
Under the assumptions of Proposition \ref{damping}, there exists a choice of symmetrizer 
$\mathbb{A}^0_\alpha(0^-)=\blockdiag\{A^0(0^-),A^0(0^+)\}$ such that the half-line problem \eqref{mathbbinteriorxx}-\eqref{doubleboundary} is maximally dissipative with respect to $\mathbb{A}^0_\alpha$.
\ec

\begin{proof}
	As noted previously (see Remark \ref{subloprmk}), Majda's shock Lopatinsky condition for the subshock
follows in the high-frequency limit from  Evans-Lopatinsky stability of the relaxation profile.
From this, it follows in turn that the Lopatinsky condition is satisfied  for the doubled 
    half-line problem \eqref{mathbbinteriorxx}-\eqref{doubleboundary}.
Applying Lemma \ref{disslem}, we find that there exists a symmetrizer $\mathbb{A}^0$ with respect to which the half-line problem is maximally dissipative, i.e.,
$\mathbb{A}^0$ is positive definite on the kernel of the boundary condition.
Recalling that $\mathbb{A}$ is block-diagonal, we find that the block-diagonal part of $\mathbb{A}^0$ 
must be a symmetrizer as well, and, moreover, is positive definite whenever the full matrix is, in particular on 
the kernel of the boundary condition.  Thus, we may take without loss of generality
$\mathbb{A}^0(0^-)=\blockdiag\{A^0(0^-),A^0(0^+)\}$.
Now extend the definition of $\mathbb{A}^0$ from $0^-$ to $-\infty$ and define $\alpha$ 
by \eqref{defalpha} and $\mathbb{A}^0_\alpha$ by \eqref{A0alpha}. Because $\alpha(0^\pm)=Id$, $\mathbb{A}^0_\alpha(0^-,t)$ is equal to $\mathbb{A}^0(0^-,t)$. Therefore, the half-line problem is maximally dissipative with respect to $\mathbb{A}^0_\alpha$.
\end{proof}

\br\label{singlermk}
The step in the proof where we extend the value of $\mathbb{A}^0$ from the boundary $0^-$ to $-\infty$ is the point
where we require the property that there is only a single subshock.
If there were subshocks at $x_0< 0$, then we could not necessarily simultaneously prescribe dissipative values at $x_0^+$,
$0^-$ and also achieve smoothness on $(x_0,0)$.
\er

Applying Lemma \ref{bcestlem}, we have
\ba 
\label{absorb1}
[v_{xx}\cdot\tilde{A}^0_\alpha\tilde{A} v_{xx}]
=&-\mathbb{V}_{xx}(0^-)\cdot \mathbb{A}^0_\alpha \mathbb{A}(0^-)\mathbb{V}_{xx}(0^-)\\
\le& -\theta |\mathbb{V}_{xx}(0^-)|^2+O\left(|\mathbb{V}_x(0^-)|^2+|\mathbb{V}(0^-)|^2+|\chi_t|^2+\zeta^2|\mathbb{V}_{xx}(0^-)|^2\right)\\
\le& -\frac{\theta}{2} \left(|v_{xx}(0^-)|^2+|v_{xx}(0^+)|^2\right)+O\left(|v_x(0^\pm)|^2+|v(0^\pm)|^2+|\eta_t|^2\right)
\ea 
where in the last inequality we take $\zeta\ll \theta$ to eliminate $O(\zeta^2|\mathbb{V}_{xx}(0^-)|^2)$. The term $O(|\eta_t|^2)$ is evidently absorbable in \eqref{diffdamp}. 
And, by Sobolev embedding 
$$
|v_x(0^\pm)|^2\le|v|^2_{H^2}\le \bar\zeta|v_{xx}|_{L^2}^2+C(\bar\zeta)|v|_{L^2}^2,\quad |v(0^\pm)|^2\le|v|^2_{H^1}\le \bar\zeta|v_{x}|_{L^2}^2+C(\bar\zeta)|v|_{L^2}^2,
$$
hence $O\left(|v_x(0^\pm)|^2+|v(0^\pm)|^2\right)$ is controlled by $\bar\zeta{\mathcal E}+C(\bar{\zeta})|v|_{{}_{L^2}}^2$.

Because $
\dot\eta(t)[v_{xx}\cdot\tilde{A}^0_\alpha v_{xx}]=O(\zeta|v_{xx}(0^\pm)|^2)$, the trace term $\dot\eta(t)[v_{xx}\cdot\tilde{A}^0_\alpha v_{xx}]$ can be eliminated by taking $\zeta\ll \theta$. 

By \eqref{esti1} and Young's inequality, the term $[v_{xt} \cdot \tilde{K}_\alpha v_x]$ may be estimated  as
$$
[v_{xt} \cdot \tilde{K}_\alpha v_x]=-[\tilde{A}v_{xx}\cdot\tilde{K}_\alpha v_x]+O\left(|v_x(0^\pm)|^2+|v(0^\pm)|^2+|\eta_t|^2+\zeta^2|v_{xx}(0^\pm)|^2\right).
$$
By our estimates just above, the terms within the ``Big-Oh'' term $O(\dots)$ are either controlled by 
 $\bar\zeta{\mathcal E}+C(|v|_{{}_{L^2}}^2+|\dot\eta(t)|^2)$ or eliminted by $-\theta|\mathbb{V}_{xx}(0^-)|^2$. 
 Moreover, applying Young's inequality and Sobolev embedding, we have
$$
[\tilde{A}v_{xx}\cdot\tilde{K}_\alpha v_x]
\le \tilde\zeta|v_{xx}(0^\pm)|^2+C(\tilde{\zeta})|v_x(0^\pm)|^2
\le \tilde\zeta|v_{xx}(0^\pm)|^2+C(\tilde{\zeta})\bar{\zeta}|v_{xx}|_{L^2}^2+C(\tilde{\zeta},\bar{\zeta})|v|_{L^2}^2.$$
Taking $\tilde{\zeta}\ll\theta$ and $\bar{\zeta}\ll\tilde{\zeta}$, the estimate \eqref{diffdamp} thus becomes 
\ba\label{tdiffdamp}
  \frac{d{\mathcal E}}{dt}\leq&-\theta{\mathcal E}+C(|v|_{{}_{L^2}}^2 + |\dot\eta(t)|^2) + [v_{xx}\cdot \tilde A^0_\alpha \tilde Av_{xx}]+\dot\eta(t)[v_{xx}\cdot\tilde{A}^0_\alpha v_{xx}]
  -[v_{xt} \cdot\tilde{K}_\alpha v_x]\\
  \leq& -\theta' ( {\mathcal E}  + |v_{xx}(0^\pm)|^2) +C(|v|_{{}_{L^2}}^2 + |\dot\eta(t)|^2),
 \ea
 implying, and slightly improving, the estimate
 $\frac{d{\mathcal E}}{dt} \leq -\theta  {\mathcal E}   +C(|v|_{{}_{L^2}}^2 + |\dot\eta(t)|^2)$
 required to finish the argument (the same one established in the smooth case \cite{MZ2}).  
 This completes the proof.
\end{proof}

\section{Nonlinear stability}\label{s:nonlinear}
With the above preparations, nonlinear orbital asymptotic stability now follows essentially as in \cite{MZ2}.
After, we obtain nonlinear stability/boundedness of the phase $\eta$ by a bootstrap argument using the ``Strichartz-type''
bounds of Section \ref{s:aux}, a new aspect of our analysis not present in the smooth case.
Finally, by a further, approximate characteristic estimate, we establish convergence of the phase and full phase-asymptotic orbital stability.
Let $v(x,t)=\tilde W(x+ct-\eta(t),t) -\overline{W}(x)$ be the nonlinear perturbation defined in Section \ref{s:3}.
For $s\geq 2$, define
\begin{equation}
 \zeta(t):= \sup_{0\le s \le t, 2\le p\le \infty}
  \Big(  |v(\cdot, s)|_{{}_{L^p}}(1+s)^{\frac{1}{2}(1-\frac1p)}
  + |\dot \eta (s)|(1+s)^{\frac12}  \Big).
 \label{zeta2}
\end{equation}

\bl\label{zetalem}
Under the assumptions of Theorem \ref{main},
for all $t\ge 0$ for which a solution $v$ exists with
$\zeta(t)$ uniformly bounded by some fixed, sufficiently small constant,
there holds
\be\label{zetaclaim}
 \zeta(t) \leq C_2(|v_0|_{{}_{L^1 \cap H^s}} + \zeta(t)^2).
 \ee
\el

\begin{proof}
Following \cite{MZ2}, we show in turn that each of 
$|v(\cdot,s)|_{{}_{L^p}}(1+s)^{\frac{1}{2}(1-\frac1p)}$ and
$|\dot \eta(s)|(1+s)^{\frac12}$ is separately bounded by
 $ C(|v_0|_{{}_{L^1\cap H^2}} + \zeta(t)^2), $
for some $C>0$, all $0\le s\le t$, so long as $\zeta(t)$ remains
sufficiently small.

({\it $|v|_{L^p}$ bound.}) 
Applying integral equation \eqref{originalsol}(i) of Proposition \ref{p:integral_rep}, we find that $v$ may
be split into the sum of an interior term 
	\ba\label{intterm}
&v_I(x,s)\\
=&\int_{-\infty}^\infty G(x,s;y)v_0(y)dy+\int_0^s\int_{-\infty}^\infty G(x,s-\tau;y)I_S(y,\tau)dyd\tau\\
=&\int_{-\infty}^\infty G(x,s;y)v_0(y)dy+\int_0^s\int_{-\infty}^\infty \bar{G}_y(x,s-\tau;z)\Big(\eta_t v(z,\tau)+N_1(v(z,\tau))\Big)dzd\tau\\
 &+\int_0^s\int_{-\infty}^\infty H(x,s-\tau;z)\Big(\eta_t v_y(z,\tau)+N_1(v(z,\tau))_y\Big)dzd\tau\\
 &+\int_0^s\int_{-\infty}^\infty G(x,s-\tau;y)\left(\begin{array}{c}0\\N_2(v(y,\tau))\end{array}\right)dyd\tau\\
 &+\int_0^s \Big[\bar{G}(x,s-\tau;\cdot)\Big(\eta_t v(\cdot,\tau)+N_1(v(\cdot,\tau))\Big)\Big]d\tau\\
=:&\,v_{I1}(x,s)+v_{I2}(x,s)+v_{I3}(x,s)+v_{I4}(x,s)+v_{I5}(x,s)\\
	\ea
involving the Green kernel $G$ and a boundary term 
\be\label{boundterm}
v_B(x,s)=\int_0^sK(x,s-\tau)B_S(\tau)d\tau
\ee
involving the boundary kernel $K$, where $[\cdot]$ as elsewhere denotes jump at $y=0$.
Noting by Lemma \ref{Glem} that $G$ satisfies exactly the same $L^q\to L^p$ estimates as the corresponding 
kernel in the smooth case \cite{MZ2}, and that interior source terms $I_S$ have the same form, 
we find by the same computations as in \cite[proof of Thm. 1.2]{MZ2} that $|v_{I1}(\cdot,s)+v_{I2}(\cdot,s)+v_{I3}(\cdot,s)+v_{I4}(\cdot,s)|_{L^p}(1+s)^{\frac{1}{2}(1-\frac{1}{p})}$ is bounded 
by $ C(|v_0|_{{}_{L^1\cap H^2}} + \zeta(t)^2)$. The $L^p$ norm of the additional term $v_{I5}(\cdot,s)$ arising from integration by parts may be estimated as
$$
\begin{aligned}
&\left|\int_0^s \Big[\bar{G}(x,s-\tau;\cdot)\Big(\eta_t v(\cdot,\tau)+N_1(v(\cdot,\tau))\Big)\Big]d\tau\right|_{L^p(x)}\\
\le&\left|\int_0^s \bar{G}(x,s-\tau;0^-)\Big(\eta_t v(0^-,\tau)+N_1(v(0^-,\tau))\Big)d\tau\right|_{L^p(x)}\\
&+\left|\int_0^s \bar{G}(x,s-\tau;0^+)\Big(\eta_t v(0^+,\tau)+N_1(v(0^+,\tau))\Big)d\tau\right|_{L^p(x)}.
\end{aligned}
$$
Using the pointwise estimate on $\bar{G}(x,t;y)=S^1+R$ Theorem \ref{pointwiseintvGreen}, we thus have\footnote{Notably, there is no scattering term $S^1$ when taking $y=0^\pm$.}
$$
\begin{aligned}
&\left|\int_0^s \bar{G}(x,s-\tau;0^-)\Big(\eta_t v(0^-,\tau)+N_1(v(0^-,\tau))\Big)d\tau\right|_{L^p(x)}\\
=&\left|\int_0^s R(x,s-\tau;0^-)\Big(\eta_t v(0^-,\tau)+N_1(v(0^-,\tau))\Big)d\tau\right|_{L^p(x)}\\
\le&\int_0^sC(1+s-\tau)^{-1+\frac{1}{2p}}\zeta^2(t)(1+\tau)^{-1}d\tau\\
\le& C\zeta(t)^2(1+s)^{-1+\frac{1}{2p}}\log(1+\frac{s}{2})\le C\zeta(t)^2(1+s)^{-1+\frac{1}{2p}+\upsilon}.
\end{aligned}
$$
Therefore, $|v_{I5}(\cdot,s)|_{L^p}(1+s)^{\frac{1}{2}(1-\frac{1}{p})}$ is also bounded by $ C(|v_0|_{{}_{L^1\cap H^2}} + \zeta(t)^2)$. It remains only to treat the new boundary portion $v_B$. 
Recalling from \eqref{lineareq} that
$$
B_S(\eta_t,v)= -\eta_t[v]-[N_1(v,v)]= O\Big( (|\eta_t|+ |v(0^\pm)|)^2 \Big),
$$
we have by $\left(|v(\cdot,\tau)|_{L^\infty}+|\eta_t(\tau)|\right)(1+\tau)^{\frac{1}{2}}<\zeta(t)$ that
$ |B_S(\tau)|\leq C \zeta(t)^2(1+\tau)^{-1} $.
By Lemma \ref{Klem} we thus get
$ |v_B(s)|_{L^p}\leq  |e^{-\bar \eta(s-\cdot)} B_S|_{L^1(0,s)} +|e^{-\bar \eta(s-\cdot)} B_S|_{L^p(0,s)} \leq C \zeta(t)^2(1+s)^{-1}$,
giving the result.

\medskip

({\it $|\dot \eta|$ bound.}) 
Similarly, by integral equation \eqref{timeprimeoriginalsol2} of Lemma \ref{etadotlem}, we have that $\dot \eta$ may
be split into the sum of an interior term
$$
\begin{aligned}
&\dot \eta_I(s)\\=& \int_{-\infty}^\infty G_{1t}(s;y)v_0(y)dy+\int_{-\infty}^\infty G_1(0;y)I_S(s;y)dy+\int_0^s\int_{-\infty}^\infty G_{1t}(s-\tau;y)I_S(y,\tau)dyd\tau\\
=&\int_{-\infty}^\infty G_{1t}(s;y)v_0(y)dy+\int_{-\infty}^\infty G_{1y}(0;z)\Big(\eta_t v(z,s)+N_1(v(z,s))\Big)dz\\
 &+\int_{-\infty}^\infty G_{1}(0;y)\left(\begin{array}{c}0\\N_2(v(y,s))\end{array}\right)dy+\int_0^s\int_{-\infty}^\infty \bar{G}_{1ty}(s-\tau;z)\Big(\eta_t v(z,\tau)+N_1(v(z,\tau))\Big)dz d\tau\\
 &+\int_0^s\int_{-\infty}^\infty G_{1t}(s-\tau;y)\left(\begin{array}{c}0\\N_2(v(y,\tau))\end{array}\right)dyd\tau\\
 &+\int_0^s\int_{-\infty}^\infty H_{1}(s-\tau;z)\Big(\eta_t v_y(z,\tau)+N_1(v(z,\tau))_y\Big)dz d\tau+\Big[G_1(0;\cdot)\Big(\eta_t(s)v(\cdot,s)+N_1(v(\cdot,s))\Big)\Big]\\
&+\int_0^s\Big[ \bar{G}_{1t}(s-\tau;\cdot)\Big(\eta_t v(\cdot,\tau)+N_1(v(\cdot,\tau))\Big)\Big] d\tau\\
=:&\,\dot \eta_{I1}(s)+\dot \eta_{I2}(s)+\dot \eta_{I3}(s)+\dot \eta_{I4}(s)+\dot \eta_{I5}(s)+\dot \eta_{I6}(s)+\dot \eta_{I7}(s)+\dot \eta_{I8}(s)
\end{aligned}
$$
and a boundary term
$$
\dot \eta_B(s)= K_1(0)B_S(s)+\int_0^s K_{1t}(s-\tau)B_S(\tau)d\tau.
$$
For $\dot \eta_{I1}(s)$, $\dot \eta_{I4}(s)$, $\dot \eta_{I5}(s)$, and  $\dot \eta_{I6}(s)$ in the interior term $\dot \eta_I(t)$, both the estimates given in Lemma \ref{G1lem} and the form of the interior source term $I_S$ 
are identical to those given for the smooth case in \cite{MZ2}.
Thus, we have by the same computations as in \cite[proof of Thm. 1.2]{MZ2} that
$$(1+s)^{\frac{1}{2}}\left(\left|\dot \eta_{I1}(s)\right|+ \left|\dot \eta_{I4}(s)\right|+\left|\dot \eta_{I5}(s)\right|+\left|\dot \eta_{I6}(s)\right|\right)$$ is bounded by $ C(|v_0|_{{}_{L^1\cap H^2}} + \zeta(t)^2)$. 
Applying Lemma \ref{intG1lem} with $t=0$, $q=\infty$, we find
$$
\left|\dot \eta_{I2}(s)\right|+\left|\dot \eta_{I3}(s)\right|\le|\eta_t(s)||v(s,\cdot)|_{L^{\infty}}+|v(s,\cdot)|^2_{L^{\infty}}\le \zeta(t)^2(1+s)^{-1}.
$$
The trace term $\dot \eta_{I7}(s)$ and $\dot \eta_{I8}(s)$ arising from integration by part may be treated by using 
the fact that $G_1$ is uniformly bounded in space and time \eqref{G1bounded} together with
the pointwise estimate on $\bar{G}_{1t}=S_1+R_1$ in Theorem \ref{pointwiseinttimeetaGreen} to obtain $|\dot \eta_{I6}(s)|\le C\zeta(t)^2(1+s)^{-1}$ and $|\dot \eta_{I7}(s)|\le C\zeta(t)^2(1+s)^{-1}$. Combining, we get that
$(1+s)^{\frac{1}{2}}\dot \eta_I(s)$ is bounded by $C(|v_0|_{{}_{L^1\cap H^2}} + \zeta(t)^2)$.

Likewise, using the fact that $K_1(0)$ is a constant row vector, the bound $|B_S(\tau)|\leq C \zeta(t)^2(1+\tau)^{-1} $,
and Lemma \ref{K1lem}, we find that $ |\dot\eta_B(s)| \leq C \zeta(t)^2(1+s)^{-1}$, giving the result.
\end{proof}

\begin{proof}[Proof of Theorem \ref{main}]
({\it $v$ and  $\dot \eta$ bounds}) (following \cite[proof of Thm. 1.2]{MZ2}).
From Lemma \ref{zetalem}, it follows by continuous induction that, provided 
$|v_0|_{{}_{L^1\cap H^2}} < 1/4C_2^2$, there holds 
\begin{equation}
 \zeta(t) \leq 2C_2 |v_0|_{{}_{L^1\cap H^s}},
 \label{bd}
\end{equation}
for all $t\geq 0$ such that $\zeta$ remains small.
For, by Proposition \ref{existprop}, there exists a solution $v(\cdot,t)\in H^s$
on the open time-interval for which $|v|_{H^s}$ remains bounded and sufficiently small,
and thus $\zeta$ is well-defined and continuous. 
Now, let $[0,T)$ be the maximal interval on which $|v|_{{}_{H^s}}$
remains strictly bounded by some fixed, sufficiently small constant $\delta>0$.
By Proposition \ref{damping}, we have
\begin{equation}
 \begin{aligned}
 |v(t)|_{{}_{H^s}}^2&\le C |v(0)|^2_{{}_{H^s}}e^{-\theta t}
  +C\int_0^t e^{-\theta_2 (t-\tau )}(|v|_{{}_{L^2}}^2
  +|\dot\eta|^2)(\tau) d\tau \\
 &\le C_2\big(|v(0)|^2_{{}_{H^s}}+ \zeta(t)^2\big) (1+t)^{-\frac12},
 \end{aligned}
 \label{2calc}
\end{equation}
and so the solution continues so long as $\zeta(t)$ remains small,
with bound (\ref{bd}), at once yielding existence and the claimed 
bounds on  $|v|_{L^p\cap H^s}  $, $2\le p\le \infty$, and $|\dot \eta|$.

\medskip
({\it Auxiliary (vertical) $v$ bound.}) At this point, we have established asymptotic orbital stability,
with sharp decay rates for $|v|$ and the derivative $|\dot \eta|$ of the phase.
To obtain estimates on the phase $|\eta|$ and get full nonlinear stability, 
we first establish the {\it vertical estimate}:
\be\label{vert}
\int_0^t (1+s)^{-1/2} |v(x,s)|\, ds\leq C,\quad \forall\quad t>0,\;x\gtrless 0.
\ee
This follows by substituting for $v(x,s)$ the reprentation $v(x,s)=v_I(x,s)+v_B(x,s)$ given in
\eqref{intterm}--\eqref{boundterm} and applying the bounds of Lemma \ref{Glemaux}, together with
the bounds
\ba\label{sourcebounds}
&\Big|\eta_t v(\cdot,s)+N_1(v(\cdot,s))\Big|_{L^2}, 
\; |N_2(v(\cdot,s))|_{L^{2}}  \leq    C (|v|_{L^\infty}+|\dot \eta|)|v|_{L^2} \leq C (1+s)^{-3/4} ,\\
&\Big|\eta_t v(\cdot,s)+N_1(v(\cdot,s))\Big|_{L^{\infty}} \leq    C (|v|_{L^\infty}+|\dot \eta|)|v|_{L^\infty} \leq C (1+s)^{-1},\\
&\Big|\partial_y\big(\eta_t v(\cdot,s )+N_1(v(\cdot,s))\Big|_{L^{\infty}}\leq C (|v|_{L^\infty}+|\dot \eta|)|v_y|_{L^\infty}\leq C(1+s)^{-3/4},
\ea
following from our previously established estimates on $v$, $\dot \eta$.
Here, we have used Sobolev embedding to bound $|v_y|_{L^\infty}\leq |v|_{{H^2}}\leq C(1+s)^{-1/4}$.

\medskip
({\it $\eta$ bound.}) 
Continuing, by integral equation \eqref{originalsol}(ii) of Proposition \ref{p:integral_rep}, we have that $\eta$ may
be split into the sum of an interior term
	\ba\label{intetabd}
	\eta_I(t)=&\eta_0+ \int_{-\infty}^\infty G_1(t;y)v_0(y)dy+\int_0^t\int_{-\infty}^\infty G_1(t-s;y)I_S(s,y)dyds\\
=&\eta_0+ \int_{-\infty}^\infty G_1(t;y)v_0(y)dy+\int_0^t\int_{-\infty}^\infty G_{1y}(t-s;z)\Big(\eta_t v(z,s)+N_1(v(z,s))\Big)dzds\\
&+\int_0^t\int_{-\infty}^\infty G_1(t-s;y)\left(\begin{array}{c}0\\N_2(v(y,s))\end{array}\right)dyds\\
&+\int_0^t \Big[G_1(t-s;\cdot)\Big(\eta_t v(\cdot,s)+N_1(v(\cdot,s))\Big)\Big]ds\\
=:&\,\eta_0+\eta_{I1}(t)+\eta_{I2}(t)+\eta_{I3}(t)+\eta_{I4}(t)
	\ea
and a boundary term
\be\label{Betabd}
\eta_B(t)= \int_0^t K_1(t-s)B_S(s)ds.
\ee
By boundedness of $|G_1|$ \eqref{G1bounded}, we find that $|\eta_{I1}(t)|$ is bounded by $C|v_0|_{L^1}$.  By estimates \eqref{G1y} and \eqref{G1v},
together with vertical estimate \eqref{vert}, boundedness of $|\eta_{I2}(t)|+|\eta_{I3}(t)|$ follows from
\ba 
\label{etaI23bounded0}
&C\int_0^t (1+t-s)^{-\frac{1}{2q}}\left(|v^2|_{L^q}+|\dot\eta v|_{L^q}\right) ds+ C\int_0^t e^{-\bar\eta(t-s)}\left|\dot\eta v+N_1(v)\right|_{L^\infty} ds\\
\le&C\int_0^t (1+t-s)^{-\frac{1}{2q}}(|v|_{L^\infty}+|\dot\eta|) |v|_{L^q} ds+C\int_0^t e^{-\bar\eta(t-s)}|v|_{L^\infty}\left(|v|_{L^\infty}+|\dot\eta|\right) ds\\
\le& C\int_0^t (1+t-s)^{-\frac{1}{2q}}\zeta(t)^2(1+s)^{-1+\frac{1}{2q}} ds+C\int_0^t e^{-\bar\eta(t-s)}\zeta(t)^2 (1+s)^{-1}ds\\
\le& C\zeta(t)^2
\ea 
and 
\ba 
\label{etaI23bounded}
&\int_0^t\int_{-\infty}^\infty|\psi(z)|\left|\dot\eta(s) v(z,s)+N_1(v(z,s))\right|dzds,\int_0^t\int_{-\infty}^\infty|\tilde \psi(z)|\left|N_2(v(z,s))\right|dzds\\
\leq &C\int_0^t\int_{-\infty}^\infty e^{-\theta|z|}\zeta(t)(1+s)^{-1/2}|v(z,s)|dzds\\
=&C\int_{-\infty}^\infty e^{-\theta|z|}\zeta(t)\int_0^t(1+s)^{-1/2}|v(z,s)|dsdz\\
\leq& C\zeta(t)\int_{-\infty}^\infty e^{-\theta|z|}dz\leq C\zeta(t).
\ea 
Using the fact that $G_1$ is uniformly bounded in space and time \eqref{G1bounded}, 
together with vertical estimate \eqref{vert}, we find that
$$
\begin{aligned}
|\eta_{I4}(t)|\le&\int_0^t \left|G_1(t-s;0^-)\Big(\eta_t v(0^-,s)+N_1(v(0^-,s))\Big)\right|ds\\
&+\int_0^t \left|G_1(t-s;0^+)\Big(\eta_t v(0^+,s)+N_1(v(0^+,s))\Big)\right|ds\\
\le&C \int_0^t |v(0^\pm,s)|(|v(0^\pm,s)|+|\dot \eta|(s)\, ds \leq C\zeta(t)\int_0^t (1+s)^{-1/2}|v(0^\pm,s)|ds\leq C\zeta(t).
\end{aligned}
$$
Applying Lemma \ref{intK1lem} and using vertical estimate \eqref{vert}, we find that $|\eta_B(t)|$ is bounded by
$$
C \int_0^t |B_S|(s)ds \leq 
C \int_0^t |v(0^\pm,s)|(|v(0^\pm,s)|+|\dot \eta|(s))\, ds \leq 
C\zeta(t)\int_0^t (1+s)^{-1/2}|v(0^\pm,s)|ds
\leq C\zeta(t).
$$
Summing, we obtain the claimed bound \eqref{mainests}(iv) on $|\eta(t)|$, completing the proof.

\medskip
({\it $\eta$ convergence.}) Finally, we establish \eqref{mainests}(v) and convergence of the phase $\eta$,
by showing convergence as $t\rightarrow \infty$ of each of the terms 
in the decomposition of $\eta$ given in \eqref{intetabd}--\eqref{Betabd}.

\noindent{\bf $\eta_{I2}$, $\eta_{I3}$}:

Term $\eta_{I2}$ may be decomposed as 
\ba 
\eta_{I2}(t)=&\int_0^t\int_{-\infty}^\infty \psi(z)\Big(\eta_t v(z,s)+N_1(v(z,s))\Big)dz\,ds\\
&+\int_0^t\int_{-\infty}^\infty \big(G_{1y}(t-s;z)-\psi(z)\big)\Big(\eta_t v(z,s)+N_1(v(z,s))\Big)dz\,ds.
\ea 
By estimate \eqref{etaI23bounded}, the first integral is absolutely convergent and thus converges to a limit as $t\to +\infty$. We show now that the remaining part of $\eta_{I2}(t)$ converges to zero, completing the proof. The part corresponding to integrating high frequency term $H$ against $S_I$ converges to zero by estimate \eqref{etaI23bounded0}. It remains to show convergence to $0$ of parts corresponding to integrating Gaussian scattering term $S$ and faster-decaying $R$ terms with $\eta_t v(z,s)+N_1(v(z,s)$. These parts are bounded by
\ba 
\label{eta23con}
&\int_0^t\left(\int_{-\infty}^{0}e^{-\frac{(t-s+c_{2,-}^1y)^2}{M(t-s)}}+\int_{0}^{\infty}e^{-\frac{(t-s-c_{1,+}^1y)^2}{M(t-s)}}\right)\chi_{{}_{t-s>1}}\frac{\zeta(t)|v(y,s)|}{\sqrt{(t-s)(1+s)}}\,dy\,ds\\
\le&\int_0^{t-1}\int_{-\frac{t-s}{2c_{2,-}^1}}^{\frac{t-s}{2c_{1,+}^1}}e^{-\frac{t-s}{4M}}\frac{\zeta(t)|v(y,s)|}{\sqrt{(t-s)(1+s)}}\,dy\,ds\\
&+\int_0^{t-1}\left(\int_{-\infty}^{-\frac{t-s}{2c_{2,-}^1}}e^{-\frac{(t-s+c_{2,-}^1y)^2}{M(t-s)}}+\int_{\frac{t-s}{2c_{1,+}^1}}^{\infty}e^{-\frac{(t-s-c_{1,+}^1y)^2}{M(t-s)}}\right)\frac{\zeta(t)|v(y,s)|}{\sqrt{(t-s)(1+s)}}\,dy\,ds,
\ea 
where
$$
\int_0^{t-1}\int_{-\frac{t-s}{2c_{2,-}^1}}^{\frac{t-s}{2c_{1,+}^1}}e^{-\frac{t-s}{4M}}\frac{\zeta(t)|v(y,s)|}{\sqrt{(t-s)(1+s)}}\,dy\,ds\le\int_0^{t-1}\left(\frac{1}{2c_{1,+}^1}+\frac{1}{2c_{2,-}^1}\right)e^{-\frac{t-s}{4M}}\frac{\zeta(t)^2\sqrt{t-s}}{1+s}\,ds
$$
converges to $0$ at rate $t^{-1}$.

To show convergence to $0$ of the remaining part in \eqref{eta23con}, we establish an improved ``approximate characteristic'' estimate \eqref{charest} on the variable $v$,
giving different decay rates
based on approximate domains of influence of tail and center contributions of the initial perturbation $v_0$.
To this end, motivated by \eqref{eta23con}, it is convenient to define  
\ba\label{Gammat}
\Gamma_t:=&\{ \;(y,s): \,  -a_-(t-s)/2 < y< -a_+ (t-s)/2; \quad 0\leq s\leq t\},\\
\gamma_{\tau}:=&\{ \;y: \,  -a_-\tau/2 < y< -a_+ \tau/2\},
\ea
where $a_+ =-1/c_{1,+}^1,\;a_- =1/c_{2,-}^1$, with $a_->0>a_+$.
With this definition, we have that backward characteristics originating outside $\Gamma_t$ stay outside the set, lying strictly in its complement
$\Gamma_t^c$.
From this fact, we obtain using our previously established bounds on $v$, $\dot\eta$, and separating out principal, approximate 
equilibrium characteristic parts of the Green kernels from remaining, faster-decaying terms, the more detailed {\it approximate characteristic estimate}:
\be\label{charest}
|v(\cdot, s)|_{L^2(\gamma_{t-s}^c)}\leq
C(1+s)^{-1/4} |v_0|_{L^1(\gamma_{t}^c)} + C(1+s)^{-1/2 +\upsilon}|v_0|_{L^1\cap H^s}
\ee
valid for any $\upsilon>0$, as we shall show further below. 

With the aid of the bound \eqref{charest} and Holder's inequality, the remaining part in \eqref{eta23con} may be estimated as
$$
\begin{aligned}
&\int_0^{t-1}\left(\int_{-\infty}^{-a_-(t-s)/2}e^{-\frac{(t-s+c_{2,-}^1y)^2}{M(t-s)}}+\int_{-a_+(t-s)/2}^{\infty}e^{-\frac{(t-s-c_{1,+}^1y)^2}{M(t-s)}}\right)\frac{\zeta(t)|v(y,s)|}{\sqrt{(t-s)(1+s)}}\,dy\,ds\\
\le&C\zeta(t)\int_0^{t}(1+t-s)^{-1/4}(1+s)^{-1/2}|v(\cdot, s)|_{L^2(\gamma_{t-s}^c)}ds\\
\le&C\zeta(t)|v_0|_{L^1(\gamma_{t}^c)}\int_0^{t}(1+t-s)^{-1/4}(1+s)^{-3/4}ds+C\zeta(t)|v_0|_{L^1\cap H^s}\int_0^{t}(1+t-s)^{-1/4}(1+s)^{-1+\upsilon}ds\\
\le&C\zeta(t)|v_0|_{L^1(\gamma_{t}^c)}+C\zeta(t)|v_0|_{L^1\cap H^s}(1+t)^{-1/4+\upsilon}\rightarrow 0,\quad \text{as $t\rightarrow \infty$.}
\end{aligned}
$$
This completes the proof that $\eta_{I2}$ converges to a limit at a rate given by
the slower of $|v_0|_{L^1(\gamma_{t}^c)}$ and $\eps(1+t)^{-1/4+\upsilon}$, where $\eps$ is the
$L^1\cap H^s$ norm of the initial perturbation $v_0$.
Convergence of $\eta_{I3}$ can be shown similarly, with the same rate.

To prove \eqref{charest}, note in decompositions \eqref{intterm}--\eqref{boundterm} on $v(x,t)$, by previous estimates, the whole-space $L^2$ norms of terms $v_{I3}$, $v_{I5}$, and $v_{B}$ decay at faster rate. That is $|v_{I3}(\cdot,s)|_{L^2}\le C(1+s)^{-3/4}|v_0|_{L^1\cap H^s}$, $|v_{I5}(\cdot,s)|_{L^2}\le C(1+s)^{-3/4+\upsilon}|v_0|_{L^1\cap H^s}$, and $|v_{B}(\cdot,s)|_{L^2}\le C(1+s)^{-1}|v_0|_{L^1\cap H^s}$.  And, by Theorem \ref{pointwiseintvGreen}, the whole-space $L^2$ norms of parts contribute to $v_{I1}$, $v_{I2}$, and $v_{I4}$ from integrating with $H^{1,2}$, $R$, and $R_y$ terms also decay at faster rate $(1+t)^{-3/4+\upsilon}|v_0|_{L^1\cap H^s}$ than the rate
$(1+t)^{-1/4}|v_0|_{L^1\cap H^s}$ of $L^2$ norms of parts from integrating with terms $S^1$, $S^1_y$, and $R\left(\begin{array}{c}0\\1\end{array}\right)$.
Hence, these contributions may be absorbed in the second term on the righthand side of \eqref{charest},
and so we may focus on the parts that are contributed from $S$, $S^1_y$, and $R\left(\begin{array}{c}0\\1\end{array}\right)$. This leaves us with the task of estimating the scattering part of $v_{I1}$, giving an integral of form
\be\label{init}
\int_{-\infty}^{\infty} S^1(y,s;z)v_0(z) \, dz,
\ee
and the scattering parts of $v_{I2}$ and $v_{I3}$, giving integrals of form
\ba \label{vI2scat}
&\int_0^s \int_{-\infty}^{\infty} S^1_y(y,s-\tau;z)\Big(\eta_t(\tau)v(z,\tau)+N_1(v(z,\tau))\Big)  dz\, d\tau,\\
&\int_0^s \int_{-\infty}^{\infty} R(y,s-\tau;z)\left(\begin{array}{c}0\\N_2(v(z,\tau))\end{array}\right)  dz\, d\tau.
\ea

To estimate the $L^2(\gamma_{t-s}^c)$ norm of \eqref{init}, notice that for $(y,s)\not \in \Gamma_t$ and $z \in \gamma_{t}$,
$S^1$ gives a time-exponentially decaying contribution $|S^1(y,s; z)|\leq Ce^{-\bar \eta|y-z-a_\pm s|}/\sqrt{s}$,
hence we may bound
$$
\Big|\int_{\gamma_{t}} |S^1(\cdot, s;z)| |v_0|(z) \, dz \Big|_{L^2(\gamma_{t-s}^c)} \leq Ce^{- \eta' s}|v_0|_{L^2}.
$$
Meanwhile, the remaining contribution may be estimated as
$$
	\Big|\int_{\gamma_{t}^c} |S^1(\cdot, s;z)| |v_0|(z) \, dz \Big|_{L^2(\gamma_{t-s}^c)}\leq
\int_{\gamma_{t}^c} |S(\cdot, s;z)|_{L^2} |v_0|(y) \, dy 
\leq C(1+s)^{-1/4} |v_0|_{L^1(\gamma_{t}^c)}.
$$
Summing, we find that the ${L^2(\gamma_{t-s}^c)}$ norm of \eqref{init} is controlled by the righthand side of \eqref{charest}.

The estimate for \eqref{vI2scat}[i] goes similarly, noting for $(y,s)\not \in \Gamma_t$
and contributions of source $\eta_t(\tau)v(z,\tau)+N_1(v(z,\tau))$ originating from $(z,\tau)\in \Gamma_t$, 
the propagator $S^1_y(y,s-\tau;z)$ is exponentially decaying in $s-\tau$ and $|y-z -a_\pm (s-\tau)|$,
hence, using our prior bounds on $(|v|+|\dot \eta|)|v|(z,\tau)$, the total of such contributions is bounded by 
$C(1+s)^{-1/2}|v_0|_{L^1\cap H^s}$.
On the other hand, defining
\be\label{tildezeta}
\tilde \zeta(t)= \sup_{0\leq s \leq t} 
\frac{ |v(\cdot, s)|_{L^2(\gamma_{t-s}^c)} }
{
 (1+s)^{-1/4}|v_0|_{L^1(\gamma_{t}^c)}+ (1+s)^{-1/2 +\upsilon}|v_0|_{L^1\cap H^s}} ,
\ee
we obtain that $\big|\eta_t(\tau)v(z,\tau)+N_1(v(z,\tau))\big|_{L^2(\gamma_{t-\tau}^c)}$  may be bounded by
$$
C\zeta(t)(1+\tau)^{-1/2}|v(\cdot,\tau)|_{L^2(\gamma_{t-\tau}^c)}
\leq
C\zeta(t) \tilde \zeta(t) \Big( (1+\tau)^{-3/4}|v_0|_{L^1(\gamma_{t}^c)}+ (1+\tau)^{-1 +\upsilon}|v_0|_{L^1\cap H^s}\Big).
$$
Applying Young's convolution inequality yields
$$
\begin{aligned}
	&\left|\int_0^s \int_{\gamma_{t-\tau}^c}S^1_y(\cdot, s-\tau;z) \big(\eta_t(\tau)v(z,\tau)+N_1(v(z,\tau))\big) dz\, d\tau\right|_{L^2(\gamma_{t-s}^c)}\\
	\leq& \int_0^s \left|\int_{\gamma_{t-\tau}^c}S^1_y(\cdot, s-\tau;z) \big(\eta_t(\tau)v(z,\tau)+N_1(v(z,\tau))\big) dz\right|_{L^2(\gamma_{t-s}^c)}\, d\tau\\
	\leq&\int_0^s |S^1_y(\cdot, s-\tau;z)|_{L^1} \big|\eta_t(\tau)v(z,\tau)+N_1(v(z,\tau))\big|_{L^2(\gamma_{t-\tau}^c)}\, d\tau\\
	\leq&\int_0^s (1+s-\tau)^{-1/2} C\zeta(t) \tilde \zeta(t) \Big( (1+\tau)^{-3/4}|v_0|_{L^1(\gamma_{t}^c)}+ (1+\tau)^{-1 +\upsilon}|v_0|_{L^1\cap H^s}\Big)\, d\tau\\
	\leq &C\zeta(t)\tilde{\zeta}(t)\Big((1+s)^{-1/4}|v_0|_{L^1(\gamma_{t}^c)}+(1+s)^{-1/2 +\upsilon}|v_0|_{L^1\cap H^s}\Big)
\end{aligned}
$$
from which we obtain, combining with our previous estimates,
$$
|v(\cdot, s)|_{L^2(\gamma_{t-s}^c)}\leq C\big(1+\zeta(t)\tilde{\zeta}(t)\big)\Big((1+s)^{-1/4}|v_0|_{L^1(\gamma_{t}^c)}+(1+s)^{-1/2 +\upsilon}|v_0|_{L^1\cap H^s}\Big).
$$
Dividing $\big((1+s)^{-1/4}|v_0|_{L^1(\gamma_{t}^c)}+(1+s)^{-1/2 +\upsilon}|v_0|_{L^1\cap H^s}\big)$ and taking the supremum over $0\le s\le t$, we obtain
$$
\tilde \zeta(t)\leq C\big(1+ \tilde \zeta(t) \zeta(t)\big),
$$
yielding $\tilde \zeta(t)\leq 2C$ for $\zeta(t)$, or equivalently, $v_0|_{L^1\cap H^s}$, sufficiently small.
By definition \eqref{tildezeta}, this yields the desired bound \eqref{charest}.

\noindent{\bf $\eta_{I1}:$}

Integrating \eqref{G1tdecompose} gives $G_1(t;y)=G_1(0;y)+\int_0^t(H_1+S_1+R_1)(s;y)ds$ and thus $$\eta_{I1}(t)=\int_{-\infty}^\infty\left(G_1(0;y)+\int_0^t(H_1+S_1+R_1)(s;y)ds\right)v_0(y)dy.$$ Applying Theorem \ref{pointwiseinttimeetaGreen}, for $1<t_1<t_2$, we have the estimate
$$
\begin{aligned}
\left|\eta_{I1}(t_2)-\eta_{I1}(t_1)\right|=&\left|\int_{-\infty}^\infty\int_{t_1}^{t_2}(H_1+S_1+R_1)(s;y)v_0(y)\,ds\,dy\right|\\
\le&\int_{-\infty}^\infty\left|\int_{t_1}^{t_2}H_1(s;y)ds\right| |v_0(y)|dy+\int_{-\infty}^\infty\int_{t_1}^{t_2}\left|(S_1+R_1)(s;y)\right|ds \,|v_0(y)|dy
\end{aligned}
$$
where
$$
\int_{-\infty}^\infty\left|\int_{t_1}^{t_2}H_1(s;y)ds\right| |v_0(y)|dy\le C\int_{|y|>ct_1}e^{-\bar{\eta}|y|}|v_0(y)|dy\le Ce^{-c\bar\eta t_1 }|v_0|_{L^1}\rightarrow 0,\quad  \text{as $t_1,\;t_2\rightarrow+\infty$,}
$$
and
$$
\begin{aligned}
&\int_{-\infty}^\infty\int_{t_1}^{t_2}\left|(S_1+R_1)(s;y)\right|ds\,|v_0(y)|dy\le C\left(\int_{\gamma_{t_1}^{}}+\int_{\gamma_{t_1}^c}\right)\int_{t_1}^{\infty}\left|S_1(s;y)\right|ds\,|v_0(y)|dy\\
\le& C\int_{\gamma_{t_1}^{}}\int_{t_1}^\infty\frac{1}{\sqrt{s}}e^{-\frac{s}{4M}} ds \,|v_0(y)|dy+C\int_{\gamma_{t_1}^c}\int_{1}^{\infty}\left|S_1(s;y)\right| ds \,|v_0(y)|dy\\
\le&C\int_{\gamma_{t_1}^{}} erfc\left(\sqrt{\frac{t_1}{2M}}\right) \,|v_0(y)|dy +C|v_0|_{L^1(\gamma_{t_1}^c)}\le Ce^{- t_1/(2M) }|v_0|_{L^1}+C|v_0|_{L^1(\gamma_{t_1}^c)}
\end{aligned}
$$
where we have used the estimates $erfc(x)\le {e^{-x^2}}$ for the error function and $\int_1^\infty \theta(y,s)ds<C,\;\forall y$ for a moving Gaussian kernel. Therefore $\eta_{I1}(t)$ approaches a limit at rate $ |v_0|_{L^1(\gamma_{t}^c)}$.

\noindent{\bf $\eta_{I4}:$}

Again using $G_1(t;y)=G_1(0;y)+\int_0^t(H_1+S_1+R_1)(s;y)ds$, we have for $1<t_1<t_2$,
$$
\begin{aligned}
&\eta_{I4}(t_2)-\eta_{I4}(t_1)\\
=&\int_{t_1}^{t_2}\Big[G_1(0;\cdot)\Big(\eta_t(s)v(\cdot,s)+N_1(v(\cdot,s))\Big)\Big]ds\\
&+\left[\left(\int_0^{t_1}\int_{t_1-s}^{t_2-s}+\int_{t_1}^{t_2}\int_0^{t_2-s}\right)(H_1+S_1+R_1)(\tau;\cdot)d\tau\Big(\eta_t(s)v(\cdot,s)+N_1(v(\cdot,s))\Big)\,ds\right]
\end{aligned}
$$
where the first part is controlled by $\int_{t_1}^{t_2} (1+s)^{-1/2}|v(0^\pm,s)|ds$ and by vertical estimate \eqref{vert} it converges to $0$ as $t_1,t_2\rightarrow+\infty$. As for convergence rates, replacing integrals in Lemma \ref{Glemaux} by tail integrals $\int_t^\infty$, we find that the convergence
rate of integral \eqref{vert} is $(1+t)^{-1/4+\upsilon}$, namely
\be\label{vertcon}
\int_t^\infty (1+s)^{-1/2} |v(x,s)|\, ds\leq C(1+t)^{-1/4+\upsilon},\quad \forall\quad t>0,\;x\gtrless 0.
\ee
It remains to show that the remaining part converges to $0$. Straigntforward computation shows that
$ \int_{t_1-s}^{t_2-s}H_1(\tau;0^\pm)d\tau=0$, with $\int_0^{t_2-s}H_1(\tau;0^\pm)d\tau=H_1(0;0^\pm)$ identically equal to some constant vectors.
Thus, the integral of the term involving $H_1$ can also be controlled by $\int_{t_1}^{t_2}(1+s)^{-1/2}|v(0^\pm,s)|ds$, hence converges. 
The integral of the term involving $S_1+R_1$ can be controlled by
$$
\begin{aligned}
&\left(\int_0^{t_1}\int_{t_1-s}^{t_2-s}+\int_{t_1}^{t_2}\int_0^{t_2-s}\right)|S_1(\tau;0^\pm)|d\tau\Big|\eta_t(s)v(0^\pm,s)+N_1(v(0^\pm,s))\Big|\,ds\\
\le&C \int_0^{t_1}\int_{t_1-s}^\infty\chi_{{}_\tau>1}\frac{1}{\sqrt{\tau}}e^{-\frac{\tau}{M}}d\tau\Big|\eta_t(s)v(0^\pm,s)+N_1(v(0^\pm,s))\Big|\,ds\\
&+C\int_{t_1}^{t_2}\int_1^{\infty}\frac{1}{\sqrt{\tau}}e^{-\frac{\tau}{M}}d\tau\Big|\eta_t(s)v(0^\pm,s)+N_1(v(0^\pm,s))\Big|\,ds\\
\le&C \int_0^{t_1-1}erfc(\sqrt{(t_1-s)/M})\Big|\eta_t(s)v(0^\pm,s)+N_1(v(0^\pm,s))\Big|\,ds+C\int_{t_1-1}^{t_2}(1+s)^{-1/2}|v(0^\pm,s)|\,ds\\
\le&C \int_0^{t_1-1}e^{-\frac{t_1-s}{M}}\zeta(t)^2(1+s)^{-1}\,ds+C\int_{t_1-1}^{t_2}(1+s)^{-1/2}|v(0^\pm,s)|\,ds\rightarrow 0,\quad \text{as $t_1,t_2\rightarrow+\infty$.}
\end{aligned}
$$
Combining, we find that $\eta_{I4}$ converges to a limit at rate $(1+t)^{-1/4+\upsilon}$.

\noindent{\bf $\eta_{IB}:$}

Convergence of $\eta_{IB}(t)$ can be proven similarly by applying Theorem \ref{pointwiseboutimeetaGreen}.

This completes the proof of convergence $\eta(t)$ to a limit $\eta_\infty$.
Collecting estimates, we obtain a total rate of convergence given by the 
slower of $C\eps(1+t)^{-1/4+\upsilon}$ and $C|v_0|_{L^1(\gamma_{t}^c)}$, verifying \eqref{mainests}(v).

\end{proof}

\br\label{linphasermk}
For algebraically-decaying initial perturbation $|v_0(x)|\leq C(1+|x|)^{-r}$, with $v_0\in H^s$, $s>2$, our estimates
give convergence of the phase $\eta$ at rate
$$
|\eta(t)-\eta_\infty|\leq \begin{cases} C(1+t)^{1-r}  & \hbox{\rm for $ 1<r<5/4$, } \\
C(1+t)^{-1/4+\upsilon}& \hbox{\rm for $ r\ge 5/4$,}
\end{cases}
$$
for any $\upsilon>0$.  For subalgebraically-decaying perturbations, essentially the same argument gives rate
$ |\eta(t)-\eta_\infty|\leq C\eps(1+t)^{-1/4+\upsilon} + C|v_0|_{L^1([-(1-\upsilon)a_- t, -(1-\upsilon)a_+ t]^c)}$
for $\upsilon>0$, $C=C(\upsilon)>0$, arbitrarily close to the expected rate $C|v_0|_{L^1([-a_- t, -a_+ t]^c)}$
described in the introduction.
\er

\br\label{sysphasermk}
Our argument for phase-convergence, based on approximate characteristic estimate \eqref{charest},
though it may appear to be limited to the case of a scalar equilibrium system for which all equilibrium characteristics approach the
shock, is in principle generalizable to arbitrary relaxation systems of the type studied in \cite{MZ2}, and to the class of systems of viscous
conservation laws studied in \cite{MZ4}. 
For, as noted in \cite{MZ2,MZ4}, non-decaying contributions to the phase shift $\eta$ consist of products of Gaussian scattering-type terms
multiplying constant projections, which projections annihilate vectors in outgoing characteristic modes, ``seeing'' only incoming modes.
Thus, to obtain asymptotic phase-convergence, it is sufficient to prove an approximate characteristic estimate of form \eqref{charest}
on incoming characteristic modes only, a task to which the present argument structure is in principle still suited.
To carry out such an estimate and obtain phase-convergence in the general system case, assuming only $L^1$ boundedness of the initial
perturbation with no algebraic rate of decay, would be a significant advance in the theory.
\er

\br\label{G1rmk}
One may deduce from \eqref{originalsol}(ii)
that $G_1(0;y)=0$ for $y\neq 0$, by finite propagation speed for the linearized problem $I_S=0$, considering
perturbations vanishing in a vicinity of $y=0$.
Similarly, by conservation of mass principles, one may deduce that $\lim_{t\to +\infty}G_1(t;y)\equiv (\alpha,0)^T$
for some constant $\alpha$; see Remark \ref{noratermk}. However, we neither require nor derive these here.
\er

\section{Numerical verifications}\label{s:numerics}
In this section we verify numerically the spectral stability assumptions made in the analysis.

\subsection{Numerical calculation of the Evans-Lopatinsky determinant}\label{hybridsub}
For robustness of numerical implementation, let 
\be 
\label{wcoordinate}
w_{1,-}(\lambda,x)=e^{-\gamma_{1,-}(\lambda)x}T_-(\lambda,x)e^{\gamma_{1,-}(\lambda)x}z_{1,-}(\lambda)=T_-(\lambda,x)z_{1,-}(\lambda),\quad x<0,
\ee 
We find the $w_{1,-}$ solves
\be
\label{shift1}
w'=\left(A^{-1}(E-\lambda I-A_x)-\gamma_{1,-}\right)w.
\ee
In $w_{1,-}$, the Evans-Lopatinsky determinant \eqref{lopatinsky} becomes
\be 
\label{reslop}
\Delta(\lambda)=\det\left(\begin{array}{rr}[\lambda W-R(W)]&A(0^-) w_{1,-}(\lambda,0^-)\end{array}\right)
\ee

\subsubsection{Change of independent variable}\label{s:independent}
Profile $H(x)$ solves (\ref{profileODE}).
The fact that $H'< 0$ for $x<0$ allows us to make the change of independent variable
$\tilde{w}(\lambda,H)=w(\lambda,x)$
for system \eqref{shift1}, yielding
\be
\label{shiftwH}
H'\tilde{w}'=\left(A^{-1}(E-\lambda I-A_x)-\gamma_{1,-}\right)\tilde{w}
\ee
The Evans-Lopatinsky determinant \eqref{reslop} becomes
\be 
\label{hcorlop}
\Delta(\lambda)=\det\left(\begin{array}{rr}[\lambda W-R(W)]& A(H_*) \tilde{w}(\lambda,H_*)\end{array}\right)
\ee
with $[\cdot]=\cdot|_{H_R}-\cdot|_{H_*}$. 
By this change of independent variable, we convert to a problem on the finite interval $[H_*,1]$ and introduce $H=1$ as a singular point in ODE \eqref{shiftwH}. We then may use the hybrid method introduced in \cite{JNRYZ} to calculate mode $\tilde{w}(H)$, combining power series expansion with numerical ODE solution.

To be specific, we expand $\tilde{w}(H)$ as a power series of in the vicinity of $H=1$ to write
\ba
\tilde{w}(\lambda,H)=\sum_{n=0}^\infty c_n(F,H_R,\lambda)(H-1)^n.
\ea
Truncating and evaluating the series at some $H_-\in(H_*,1)$ gives approximations
\ba
\tilde{w}(\lambda,H_-)\approx\sum_{n=0}^N c_n(F,H_R,\lambda)(H-1)^n:=\tilde{w}_-.
\ea
We then evolve ODE \eqref{shiftwH} from $H_-$ to $H_*$ with initial condition $\tilde{w}_-$ to get an approximation for $\tilde{w}(\lambda,H_*)$ which is then substituted in \eqref{hcorlop} to obtain an approximate
value of the Evans-Lopatinsky determinant.

\subsection{Numerical calculation of the Evans function (smooth case)}\label{hybridsmooth}
In this section, we study the spectral stability of small amplitude traveling waves, 
as depicted in Figure \ref{profile}(c), using the Evans function.
In the small amplitude region $H_R<H_L<H_R\frac{1+2F+\sqrt{1+4F}}{2F^2}$, we first see conditions \eqref{signgamma} become
\ba 
\label{signgamma2}
&\Re\gamma_{1,-}(\lambda)>0,\;\Re\gamma_{2,-}(\lambda)<0,\quad\text{for all $\Re\lambda>0,\;F<2,\;\nu>1$}\\
&\Re\gamma_{1,+}(\lambda)>0,\;\Re\gamma_{2,+}(\lambda)<0,\quad\text{for all $\Re\lambda>0,\;\nu<\frac{1+\sqrt{1+4F}}{2F}$}.
\ea
We then define the corresponding Evans function, following \cite{MZ,GZ,AGJ}.
\begin{definition}\label{defEvans}
Let $v_{1,-}(\lambda,x)$ ($v_{2,+}(\lambda,x)$) be decaying mode as $x\rightarrow-\infty$ ($x\rightarrow+\infty$) of eigenvalue equation \eqref{smootheigen}
\be
\label{smootheigen}
v'=\left(A^{-1}(E-\lambda I-A_x)\right)v
\ee
The Evans function $D(\lambda,x_0)$ is defined as 
\be
\label{Evans}
D(\lambda,x_0):=\det\left(\begin{array}{rr}v_{1,-}(\lambda,x_0)&v_{2,+}(\lambda,x_0)\end{array}\right).
\ee
\end{definition}
Again for numerical robustness and efficiency, we rescale the modes by
\be 
w_{1,-}(\lambda,x)=e^{-\gamma_{1,-}(\lambda)x}v_{1,-}(\lambda,x),\quad w_{2,+}(\lambda,x)=e^{-\gamma_{2,+}(\lambda)x}v_{2,+}(\lambda,x)
\ee
to find that $w_{1,-}$, $w_{2,+}$ solve
\be 
\label{smoothode}
w'=\left(A^{-1}(E-\lambda I-A_x)-\gamma_{1,-}\right)w,\quad w'=\left(A^{-1}(E-\lambda I-A_x)-\gamma_{2,+}\right)w,
\ee
respectively. Performing the change of independent variable $\tilde{w}_{1,-}(\lambda,H)=w_{1,-}(\lambda,x)$ and $\tilde{w}_{2,+}(\lambda,H)=w_{2,+}(\lambda,x)$, we find that $\tilde{w}_{1,-}$, $\tilde{w}_{2,+}$ satisfy
\be
H'\tilde{w}'=\left(A^{-1}(E-\lambda I-A_x)-\gamma_{1,-}\right)\tilde{w},\quad H'\tilde{w}'=\left(A^{-1}(E-\lambda I-A_x)-\gamma_{2,+}\right)\tilde{w}.
\ee

We then expand $\tilde{w}_{1,-}(H)$, $\tilde{w}_{2,+}(H)$ as power series
\ba
\label{powersmooth}
\tilde{w}_{1,-}(\lambda,H)=\sum_{n=0}^\infty c^-_n(F,H_R,\lambda)(H-1)^n, \quad\tilde{w}_{2,+}(\lambda,H)=\sum_{n=0}^\infty c^+_n(F,H_R,\lambda)(H-H_R)^n.
\ea
Accordingly, in $H$ coordinates a rescaled Evans function is defined as
\be 
\label{hcorEvans}
D(\lambda,H_m):=\det\left(\begin{array}{rr}\tilde{w}_{1,-}(\lambda,H_m)&\tilde{w}_{2,+}(\lambda,H_m)\end{array}\right)
\ee 
for some $H_m$ $H_R<H_m<1$. 

Note that $|\gamma_{2,+}|\gg |\gamma_{1,-}|$.
Thus, it is numerically more robust if we evaluate $D(\lambda,\cdot)$ at some $H_m$ closer to $H_R$. (In fact, we find this in practice essential in order to do computations for even reasonably sized $|\lambda|$ of order one.)
In the extreme case, we only evolve \eqref{smoothode}(i) toward $H_R$ and never evolve \eqref{smoothode}(ii) towards $1$. 
That is, after evaluating the truncated series \eqref{powersmooth} at some $H_{l,r}$ $H_R<H_r<H_l<1$, we evolve \eqref{smoothode}(i) from $H_l$ to $H_r$.

(Note, the numerically calculated Evans function differs from the defined one by a nonzero analytic function, but
this is harmless as we are searching for roots.)

\subsection{High-frequency stability}
Using the result of Lemma \ref{orihflemma}, we now prove high-frequency stability of both smooth and
discontinuous hydraulic shock waves. 

\textbf{Nonvanishing of Evans-Lopatinsky determinant \eqref{lopatinsky} at high frequency.}
Evaluating \eqref{stablesolution} at $x=0^-$ and substituting in the second column of \eqref{lopatinsky}, yields
\be
\label{secondcolumn}
\begin{aligned}
A(0^-)w_{1,-}(\lambda,0^-)
=&A(0^-)R_1(0^-)+O(1/|\lambda|)=-\frac{1}{\mu_1(0^-)}R_1(0^-)+O(1/|\lambda|)
\end{aligned}
\ee
where $R_1$ is the first column of $R$.
\begin{proposition}\label{hfprop}
For any $F,H_R$, there exists $C(F,H_R)$, such that $\Delta(\lambda)$ does not vanish for all $\Re\lambda>
-\bar\eta,|\lambda|>C(F,H_R)$.
\begin{proof}
Substituting \eqref{secondcolumn} in the Evans-Lopatinsky determinant \eqref{lopatinsky}, in the high frequency regime, we have 
\ba
\label{hfexpansion}
\Delta(\lambda)=&-\frac{\lambda}{\mu_1(H_*)}\det\left(\begin{array}{rr}H_R- H_* &-FH_*(\sqrt{H_R}+1)\\Q_R-Q_*&H_*^{3/2}(\sqrt{H_R}+1)-F(H_*-H_R+H_*H_R+H_*\sqrt{H_R}))\end{array}\right)+O(1)\\
=&-\frac{\lambda\left(H_R-H_*\right){\left({H_*}^{3/2}+\sqrt{H_R}{H_*}^{3/2}+FH_R\right)}^2}{FH_*\left(\sqrt{H_R}+1\right)}+O(1)
\ea
which is nonvanishing. The constant $C$ should be sufficiently large such that $\mathcal{T}_\lambda$ becomes contraction mapping.
\end{proof}
\end{proposition}
\begin{remark}\label{subloprmk}
The principal, $\lambda$-order, term in the righthand side of \eqref{hfexpansion} can be recognized as
the Lopatinsky condition of Majda \cite{Ma} for short-time stability/well-posedness of the component subshock,
considered as a solution of the first-order part of \eqref{sv} with forcing terms
set to zero; see \cite{Er1,JLW,Z2} for similar observations in the context of detonations.
As the first-order system in this case coincides with the equations of isentropic gas dynamics with $\gamma$-law
pressure (see Introduction), nonvanishing of this principal part is a special case of the theorem of \cite{Ma,Se2} that
shock waves of isentropic gas dynamics are Lopatinsky stable for any monotone pressure function.
\end{remark}

\noindent\textbf{Nonvanishing of Evans function \eqref{Evans} at high frequency.}\\
The high frequency analysis of Section \ref{HighFrequency} also applies to the smooth case,
yielding the following result.

\begin{proposition}\label{hfpropEvans}
For any $F,H_R$, there exists $C(F,H_R)$, such that $D(\lambda,0)$ does not vanish for all $\Re\lambda>
-\bar\eta,|\lambda|>C(F,H_R)$.
\end{proposition}

\begin{proof}
In the high frequency regime, following Lemma \ref{orihflemma}, we find that the decaying modes $v_{1,-}$, $v_{2,+}$ in Definition \ref{defEvans} are up to a scalar multiple equal to
\ba 
v_{1,-}(\lambda,x)=&e^{\int_0^x\left(\Lambda_{11}(\lambda,y)+\frac{1}{\lambda}N_{11}(\lambda,y)+\frac{1}{\lambda}N_{12}(\lambda,y)\Phi_1(\lambda,y)\right)dy}\left(R_1(x)+O(1/|\lambda|)\right),\\
v_{2,+}(\lambda,x)=&e^{\int_0^x\left(\Lambda_{22}(\lambda,y)+\frac{1}{\lambda}N_{22}(\lambda,y)+\frac{1}{\lambda}N_{21}(\lambda,y)\Phi_2(\lambda,y)\right)dy}\left(R_2(x)+O(1/|\lambda|)\right).
\ea 
Evaluating the Evans function \ref{Evans} at $x_0=0$ yields
\be
\label{hfexpansionEvans}
D(\lambda)=\det\left(\begin{array}{rr}R_1(0) &R_2(0)\end{array}\right)+O(1/|\lambda|)
\ee
which is nonvanishing. The constant $C$ should be sufficiently large such that $\mathcal{T}_\lambda$ becomes contraction mapping.
\end{proof}

\begin{remark}
High frequency stability restricts the study of spectral stability to 
{\it investigation of the bounded domain $\{\lambda:\Re\lambda>-\bar\eta,|\lambda|\le C(F,H_R)\}$,}
a numerically feasible problem.
\end{remark}

\subsection{Verification of mid- and low-frequency stability}
\label{EvansLopMidLow}
The hybrid schemes described in Sections \ref{hybridsub}, \ref{hybridsmooth} are implemented in Matlab and show great efficiency (see Table \ref {table1}, Table \ref {table2} in Appendix \ref{computertime} for computation time). To determine stability, we fix $0<r<R$ and $a\ll 1$ and examine the presence of spectrum within the set
$$\Omega(r,R,a):=\{\lambda -a:\Re\lambda>0,r<|\lambda|<R\}.$$ At the end, we compute numerically the
winding numbers of contours $\Delta(\partial\Omega(r,R,a))$ and $D(\partial\Omega(r,R,a))$, i.e. we discretize  $\partial\Omega(r,R,a)$ as $\lambda_0,\;\lambda_1,\cdots,\lambda_n,\lambda_{n+1}=\lambda_0$ and calculate the winding number by
$$
n(\Omega):=\frac{1}{2\pi}\sum_{i=0}^n \angle \Big(\Delta(\lambda_i),\Delta(\lambda_{i+1})\Big), \quad \Bigg(n(\Omega):=\frac{1}{2\pi}\sum_{i=0}^n \angle \Big(D(\lambda_i),D(\lambda_{i+1})\Big)\Bigg)
$$
where $\angle(z_1,z_2)$ denotes the angle change from $z_1$ to $z_2$.
Since $\Delta(\lambda)$ ($D(\lambda)$) is analytic in $\lambda$, it is clear that $n$ counts its number of zeros of $\Delta(\lambda)$ ($D(\lambda )$) within the set $\Omega$.

We have verified that all 
%large-amplitude 
discontinuous hydraulic profiles are mid- and low-frequency stable. Here ``all" is limited to discretized existence domain $F\in[0.05:0.05:1.95]$, $H_R\in[0.01:0.01:H_C(F)-0.01]$ ($1559$ points in total) and mid- and low-stability is checked for $\Omega:=\Omega(0.1,C(F,H_R),0.000001)$ where $C(F,H_R)$ defined in Proposition \ref{hfprop} can be estimated by Lemma \ref{cone}. 
Note that, exceptionally, there are $191$ points in the low $F$ regime requiring $C(F,H_R)>2000$ and one parameter $(F=0.85,H_R=0.25)$ even requiring a $C(F,H_R)$ as large as $1.1664\times10^5$. It turns out for these values that
for that large $\lambda$, in the power series evaluation step, the hybrid scheme cannot move enough distance away from the singular point $H=1$, causing problems in the later ODE-evolution step. 
Numerics are then not robust for these pair of $F,H_R$. We have restricted $C(F,H_R)=2000$ for roughly half of these low-$F$ points, and 
$C(F,H_R)=100-1,000$ for the rest.

See Figure \ref{windingsub}(a)-(b) for typical images of contours $\partial{\Omega}(r,R,a)$ under function $\Delta(\lambda)$.

We have also verified that all small amplitude smooth hydraulic shock waves are mid- and low-frequency stable. Here ``all" is limited to discretized existence domain $F\in[0.05:0.05:1.95]$, $H_R\in[0.99:-0.01:H_C(F)+0.01]$ ($2227$ points in total) and mid- and low-stability is checked for $\Omega:=\Omega(0.1,C(F,H_R),0)$ where $C(F,H_R)$ defined in Proposition \ref{hfpropEvans} can be estimated by Lemma \ref{cone}. Note that, exceptionally, there are $18$ points in the $F\approx 1$ regime requiring $C(F,H_R)>2000$. For the same reasoning, numerics is then not robust for these pairs of $F,H_R$. We have restricted $C(F,H_R)=2000$ for these points. 

See Figure \ref{windingsub}(c)-(d) for typical images of contours $\partial{\Omega}$ under function $D(\lambda)$.

\begin{figure}
\begin{center}
\includegraphics[scale=0.32]{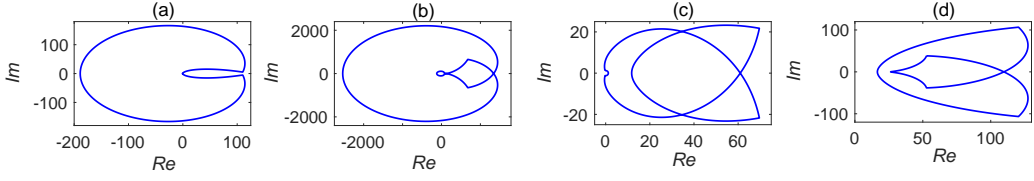}
\end{center}
\caption{Images of contours under Evans Lopatinsky determinant $\Delta(\lambda)$ for (a) $F=1.5$, $H_R=0.2$, $\Omega_1:=\Omega(0.1,5,0.000001)$ (b) $F=1.5$, $H_R=0.2$, $\Omega_2:=\Omega(4,15,0.000001)$; and images of contours under Evans function $D(\lambda)$ for (c) $F=1.5$, $H_R=0.8$, $\Omega_1:=\Omega(0.1,5,0)$ (b) $F=1.5$, $H_R=0.8$, $\Omega_2:=\Omega(4,10,0)$. Winding numbers are all zeros.}
\label{windingsub}
\end{figure}

\subsection{Time evolution of perturbed hydraulic shock profiles}
We have carried out also a time-evolution study using CLAWPACK \cite{C1,C2}, illustrating stability under perturbations of large amplitude discontinuous hydraulic shocks and small amplitude smooth hydraulic shocks. In both cases, all evolutions 
clearly indicate stability. In Figure \ref{timeev}, we display the results under two different perturbations of a discontinuous profile.  In Figure \ref{smoothev}, we display the results for a perturbed smooth profile.
Note that for the exceptional points for which we were not able to carry out a winding-number study out to the full theoretical radius
provided by high-frequency asymptotics, 
these time-evolution studies bridge the gap between computed ($100-2,000$) and theoretical ($>2,000$) radius.
For, nonstable eigenmodes $\Re \lambda\geq 0$ with $|\lambda|\geq 100$ should be clearly visible on the timescale $0\leq t\leq 20$ considered,
dominating the solution by time $t=20$.

\begin{figure}
\begin{center}
\includegraphics[scale=0.32]{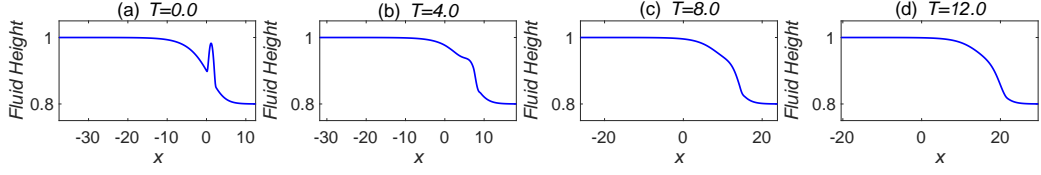}
\end{center}
\caption{Time-evolution study using CLAWPACK \cite{C1,C2}, illustrating stability under perturbation of a smooth hydraulic shock.  In (a) we show a perturbed profile with $C^\infty$ ``bump-type'' perturbation.  In (b) and (c) we show the solution at intermediate times $T=4.0$ and $8.0$ of the waveform in (a) after evolution under \eqref{sv}; stability and smoothness away from the subshock are clearly visible. In (d) we show the solution at time $T=12.0$, exhibiting convergence to a shift of the original waveform (slightly compressed in the horizontal direction due to scaling of the figure).}
\label{smoothev}
\end{figure}

\appendix
\section{Decomposition map}
\label{Decompositionmap}
The decomposition of Green kernel function $G$ can be summarized as
\ba 
\label{decomposition1}
&G=\chi_{|x-y|/t<S}\left(I+II+III\right),\quad I=\chi_{t\le 1}I+\chi_{t> 1}I,\quad \chi_{t> 1}I=\chi_{t> 1}(I^1+I^2),\\
&I^1=I^1_S+I^1_R=I_{S1}^1+I_{S2}^1+I_{S3}^1+I_{R1}^1+I_{R2}^1+I_{R3}^1,\quad I^2=I^2_{R1}+I_{R2}^2+I_{R3}^2,\\
&I_{S2}^1=\chi_{\frac{\bar{\alpha}}{p}>\eps}I_{S2}^1+\chi_{\frac{\bar{\alpha}}{p}\le\eps}\left(S^1+I_{S2Ri}^1+I_{S2Rii}^1\right),\quad I_{R2}^{1,2}=\chi_{\frac{\bar{\alpha}}{p}>\eps}I_{R2}^{1,2}+\chi_{\frac{\bar{\alpha}}{p}\le\eps}I_{R2}^{1,2},\\
&III=III^1+III^2=III^1_a+III^1_b+III^1_c+III^2_a+III^2_b+III^2_c,\quad III^{1,2}_a:=H^{1,2},
\ea
in which we see
\small
\ba
\label{decomposition2}
&H^{1,2}=\chi_{|x-y|/t<S}III^{1,2}_a,\quad S^1=\chi_{|x-y|/t<S,t>1,\bar{\alpha}/p\le \eps}S^1,\\
&R=\chi_{|x-y|/t<S}\left(II+III^{1,2}_{b,c}+\chi_{t\le 1}I+\chi_{t> 1}\left(I^1_{S1,S3,R}+\chi_{\frac{\bar{\alpha}}{p}>\eps}I_{S2}^1+\chi_{\frac{\bar{\alpha}}{p}\le\eps}\left(I_{S2Ri}^{1}+I_{S2Rii}^1\right)+I^2\right)\right).
\ea
\normalsize

\section{Integral estimates}\label{ApEstimate}
\textbf{$I^1_{S2Ri}:$}
Setting $f(u)=\frac{1}{\sqrt{4c^2_{2,-}\pi u}}e^{-\frac{(t-c^1_{2,-}(x-y))^2}{4c^2_{2,-} u}}$, yields $I^1_{S2Ri}=f(\frac{t}{c^1_{2,-}})-f(x-y)$, in which by \eqref{comparable} $\frac{t}{c^1_{2,-}}$ and $x-y$ are comparable. 
Writing the difference as an integral yields
\be
|I^1_{S2Ri}|=\frac{1}{\sqrt{4c^2_{2,-}\pi}}\left|\int_{x-y}^{\frac{t}{c^1_{2,-}}}e^{-\frac{(t-c^1_{2,-}(x-y))^2}{4c^2_{2,-}u}}\frac{\frac{(t-c^1_{2,-}(x-y))^2}{2c^2_{2,-}}-u}{2u^{\frac{5}{2}}}du\right|.
\ee
Using that $\frac{t}{c^1_{2,-}}$ and $x-y$ are comparable, we have $e^{-\frac{(t-c^1_{2,-}(x-y))^2}{4c^2_{2,-}u}}\le e^{-\frac{c^1_{2,-}(t-c^1_{2,-}(x-y))^2}{8
c^2_{2,-}t}}$, which, together with $x^ne^{-x^2}\lesssim e^{-\frac{x^2}{2}}$ for any $n$ positive, yields
{\small
\ba
|I^1_{S2Ri}|\lesssim& \int_{x-y}^{\frac{t}{c^1_{2,-}}}\left|e^{-\frac{(t-c^1_{2,-}(x-y))^2}{4c^2_{2,-}u}}\frac{\frac{(t-c^1_{2,-}(x-y))^2}{2c^2_{2,-}}-u}{2u^{\frac{5}{2}}}\right||du|\\
\lesssim&e^{-\frac{(t-c^1_{2,-}(x-y))^2}{4c^2_{2,-}\frac{2t}{c^1_{2,-}}}}(t-c^1_{2,-}(x-y))^2\left|\int_{x-y}^{\frac{t}{c^1_{2,-}}}u^{-\frac{5}{2}}du\right|+e^{-\frac{(t-c^1_{2,-}(x-y))^2}{4c^2_{2,-}\frac{2t}{c^1_{2,-}}}}\left|\int_{x-y}^{\frac{t}{c^1_{2,-}}}u^{-\frac{3}{2}}du\right|\\
\lesssim&e^{-\frac{c^1_{2,-}(t-c^1_{2,-}(x-y))^2}{8c^2_{2,-}t}}(t-c^1_{2,-}(x-y))^2\left|\frac{1}{\left(\frac{t}{c^1_{2,-}}\right)^{1.5}}-\frac{1}{\left(x-y\right)^{1.5}}\right|\\&+e^{-\frac{c^1_{2,-}(t-c^1_{2,-}(x-y))^2}{8c^2_{2,-}t}}\left|\frac{1}{\left(\frac{t}{c^1_{2,-}}\right)^{0.5}}-\frac{1}{\left(x-y\right)^{0.5}}\right|\\
\lesssim&e^{-\frac{c^1_{2,-}(t-c^1_{2,-}(x-y))^2}{8c^2_{2,-}t}}\frac{(t-c^1_{2,-}(x-y))^3}{t^{2.5}}+e^{-\frac{(t-c^1_{2,-}(x-y))^2}{8c^2_{2,-}t}}\frac{(t-c^1_{2,-}(x-y))}{t^{1.5}}\\
\lesssim&e^{-\frac{c^1_{2,-}(t-c^1_{2,-}(x-y))^2}{16c^2_{2,-}t}}\frac{1}{t}.
\ea}
\normalsize
$\frac{\partial I^1_{S2Ri}}{\partial y}:$
{\small
\be
|\frac{\partial I^1_{S2Ri}}{\partial y}|=\frac{1}{\sqrt{4c^2_{2,-}\pi}}\left|\frac{\partial}{\partial y}\int_{x-y}^{\frac{t}{c^1_{2,-}}}e^{-\frac{(t-c^1_{2,-}(x-y))^2}{4c^2_{2,-}u}}\frac{\frac{(t-c^1_{2,-}(x-y))^2}{2c^2_{2,-}}-u}{2u^{\frac{5}{2}}}du\right|
\ee
$$
\begin{aligned}
\lesssim&\left|e^{-\frac{(t-c^1_{2,-}(x-y))^2}{4c^2_{2,-}(x-y)}}\frac{\frac{(t-c^1_{2,-}(x-y))^2}{2c^2_{2,-}}-(x-y)}{2(x-y)^{\frac{5}{2}}}\right|+e^{-\frac{(t-c^1_{2,-}(x-y))^2}{4c^2_{2,-}\frac{2t}{c^1_{2,-}}}}(t-c^1_{2,-}(x-y))^3\left|\int_{x-y}^{\frac{t}{c^1_{2,-}}}u^{-\frac{7}{2}}du\right|\\&+e^{-\frac{(t-c^1_{2,-}(x-y))^2}{4c^2_{2,-}\frac{2t}{c^1_{2,-}}}}(t-c^1_{2,-}(x-y))\left|\int_{x-y}^{\frac{t}{c^1_{2,-}}}u^{-\frac{5}{2}}du\right|\\
\lesssim&e^{-\frac{c^1_{2,-}(t-c^1_{2,-}(x-y))^2}{16c^2_{2,-}t}}\frac{1}{t^{1.5}}.
\end{aligned}
$$}
\normalsize
\textbf{$I^1_{S2Rii}:$}
Using that $e^{-\frac{(t-c^1_{2,-}(x-y))^2}{4c^2_{2,-}(x-y)}}<1$ and that for the complementary error function $erfc(x):=\frac{2}{\sqrt{\pi}}\int_x^{\infty}e^{-z^2}dz$, there is the estimate
$erfc(x)\le {e^{-x^2}}$, $I^1_{S2Rii}$ can be bounded by
{\small
\be 
|I^1_{S2Rii}|\lesssim \int_r^\infty e^{-c^2_{2,-}(x-y)\xi^2}d\xi=\frac{1}{\sqrt{(x-y)c^2_{2,-}}}erfc(\sqrt{c^2_{2,-}(x-y)}r)\lesssim e^{-r^2c^2_{2,-}(x-y)}\le e^{-r^2c^2_{2,-}\frac{t}{2}},
\ee}
in which we have used that $x-y$ is comparable to $t$ hence is bounded away from $0$ and is greater than $\frac{t}{2}$.
Term $I_{S2Rii}$ is then time-exponentially small.

\textbf{$\frac{\partial I^1_{S2Rii}}{\partial y}:$}
When the partial derivative hits the exponential outside the integral we get time- exponentially small terms by following
the proof for \textbf{$I^1_{S2Rii}:$}.
When the partial derivative hits inside the integral we use $x^2e^{-x^2}\lesssim e^{-x^2/2}$ and
again get time-exponentially small terms by following the proof for \textbf{$I^1_{S2Rii}.$}

\textbf{$I^1_{R2i}:$} Using that $xe^{-x^2}\lesssim e^{\frac{-x^2}{2}}$ and that $\frac{t}{c^1_{2,-}}$ 
and $x-y$ are comparable ($\frac{t}{2c^1_{2,-}}<x-y<\frac{2t}{c^1_{2,-}}$), we have
{\small
\ba
&e^{-\frac{\left(t-c^1_{2,-}(x-y)\right)^2}{4c^2_{2,-}\left(x-y\right)}}\int_{-r}^r e^{-\xi^2c^2_{2,-}(x-y)}O|\eta_*|d\xi\lesssim e^{-\frac{\left(t-c^1_{2,-}(x-y)\right)^2}{4c^2_{2,-}\left(x-y\right)}}\int_{-r}^r e^{-\xi^2c^2_{2,-}(x-y)}\frac{|t-c^1_{2,-}(x-y)|}{x-y}d\xi\\
\lesssim& \frac{|t-c^1_{2,-}(x-y)|}{(x-y)^{\frac{3}{2}}}e^{-\frac{\left(t-c^1_{2,-}(x-y)\right)^2}{4c^2_{2,-}\left(x-y\right)}}\lesssim\frac{1}{(x-y)}e^{-\frac{\left(t-c^1_{2,-}(x-y)\right)^2}{8c^2_{2,-}\left(x-y\right)}}\le\frac{1}{\frac{t}{2c^1_{2,-}}}e^{-\frac{\left(t-c^1_{2,-}(x-y)\right)^2}{8c^2_{2,-}\frac{2t}{c^1_{2,-}}}},
\ea 
\ba
&e^{-\frac{\left(t-c^1_{2,-}(x-y)\right)^2}{4c^2_{2,-}\left(x-y\right)}}\int_{-r}^r e^{-\xi^2c^2_{2,-}(x-y)}O|\xi|d\xi\lesssim e^{-\frac{\left(t-c^1_{2,-}(x-y)\right)^2}{4c^2_{2,-}\left(x-y\right)}}\int_{0}^r e^{-\xi^2c^2_{2,-}(x-y)}\xi d\xi\\\lesssim&\frac{1}{x-y}e^{-\frac{\left(t-c^1_{2,-}(x-y)\right)^2}{4c^2_{2,-}\left(x-y\right)}}(1-e^{-r^2c^2_{2,-}(x-y)})\le\frac{1}{\frac{t}{2c^1_{2,-}}}e^{-\frac{\left(t-c^1_{2,-}(x-y)\right)^2}{8c^2_{2,-}\frac{2t}{c^1_{2,-}}}},
\ea
\ba 
&e^{-\frac{\left(t-c^1_{2,-}(x-y)\right)^2}{4c^2_{2,-}\left(x-y\right)}}\int_{-r}^r e^{-\xi^2c^2_{2,-}(x-y)}O|\eta_*^3(x-y)|d\xi\\
\lesssim &e^{-\frac{\left(t-c^1_{2,-}(x-y)\right)^2}{4c^2_{2,-}\left(x-y\right)}}\int_{-r}^r e^{-\xi^2c^2_{2,-}(x-y)}\frac{|t-c^1_{2,-}(x-y)|^3}{(x-y)^2}d\xi\lesssim \frac{|t-c^1_{2,-}(x-y)|^3}{(x-y)^{\frac{5}{2}}}e^{-\frac{\left(t-c^1_{2,-}(x-y)\right)^2}{4c^2_{2,-}\left(x-y\right)}}\\
\lesssim&\frac{1}{(x-y)}e^{-\frac{\left(t-c^1_{2,-}(x-y)\right)^2}{8c^2_{2,-}\left(x-y\right)}}\lesssim\frac{1}{t}e^{-\frac{\left(t-c^1_{2,-}(x-y)\right)^2}{8c^2_{2,-}\frac{2t}{c^1_{2,-}}}},
\ea
\ba
&e^{-\frac{\left(t-c^1_{2,-}(x-y)\right)^2}{4c^2_{2,-}\left(x-y\right)}}\int_{-r}^r e^{-\xi^2c^2_{2,-}(x-y)}O|\eta_*^2\xi(x-y)|d\xi\\
\lesssim& e^{-\frac{\left(t-c^1_{2,-}(x-y)\right)^2}{4c^2_{2,-}\left(x-y\right)}}\frac{|t-c^1_{2,-}(x-y)|^2}{x-y}\int_{0}^r e^{-\xi^2c^2_{2,-}(x-y)}\xi d\xi\\
\lesssim& \frac{|t-c^1_{2,-}(x-y)|^2}{(x-y)^2}e^{-\frac{\left(t-c^1_{2,-}(x-y)\right)^2}{4c^2_{2,-}\left(x-y\right)}}(1-e^{-r^2c^2_{2,-}(x-y)})
\lesssim\frac{1}{t}e^{-\frac{\left(t-c^1_{2,-}(x-y)\right)^2}{8c^2_{2,-}\frac{2t}{c^1_{2,-}}}},
\ea
\ba
&e^{-\frac{\left(t-c^1_{2,-}(x-y)\right)^2}{4c^2_{2,-}\left(x-y\right)}}\int_{-r}^r e^{-\xi^2c^2_{2,-}(x-y)}O|\eta_*\xi^2(x-y)|d\xi\\
\lesssim& |t-c^1_{2,-}(x-y)|e^{-\frac{\left(t-c^1_{2,-}(x-y)\right)^2}{4c^2_{2,-}\left(x-y\right)}}\int_0^r e^{-\xi^2c^2_{2,-}(x-y)}\xi^2d\xi
\lesssim \frac{|t-c^1_{2,-}(x-y)|}{(x-y)^{\frac{3}{2}}}e^{-\frac{\left(t-c^1_{2,-}(x-y)\right)^2}{4c^2_{2,-}\left(x-y\right)}}\\
\lesssim&\frac{1}{(x-y)}e^{-\frac{\left(t-c^1_{2,-}(x-y)\right)^2}{8c^2_{2,-}\left(x-y\right)}}
\lesssim\frac{1}{t}e^{-\frac{\left(t-c^1_{2,-}(x-y)\right)^2}{8c^2_{2,-}\frac{2t}{c^1_{2,-}}}},
\ea
\ba
&e^{-\frac{\left(t-c^1_{2,-}(x-y)\right)^2}{4c^2_{2,-}\left(x-y\right)}}\int_{-r}^r e^{-\xi^2c^2_{2,-}(x-y)}O|\xi^3(x-y)|d\xi
\lesssim (x-y)e^{-\frac{\left(t-c^1_{2,-}(x-y)\right)^2}{4c^2_{2,-}\left(x-y\right)}}\int_{0}^r e^{-\xi^2c^2_{2,-}(x-y)}\xi^3 d\xi\\
\lesssim &\frac{1}{x-y}e^{-\frac{\left(t-c^1_{2,-}(x-y)\right)^2}{4c^2_{2,-}\left(x-y\right)}}
\lesssim\frac{1}{t}e^{-\frac{\left(t-c^1_{2,-}(x-y)\right)^2}{8c^2_{2,-}\frac{2t}{c^1_{2,-}}}}.
\ea}
\normalsize
We then see that all terms are absorbable in $R$ \eqref{Rbound}.

\section{Computational framework}\label{s:computations}

\subsection{Computational environment} In carrying out our numerical investigations, we have used MacBook Pro 2017 with 16GB memory and Intel Core i7 processor with 2.8GHz processing speed for coding and debugging. The main parallelized computation is done in the compute nodes of IU Karst, a high-throughput computing cluster. It has $228$ compute nodes. Each node is an IBM NeXtScale nx360 M4 server equipped with two Intel Xeon E5-2650 v2 8-core processors and with 32 GB of RAM and 250 GB of local disk storage. 

\subsection{Computational time}\label{computertime}
The following computational times are times elapsed in a single processor of IU Karst.
\begin{table}[htbp]
\centering
\caption{Times to compute a single Evans-Lopatinsky determinant $\Delta_{F,H_R}(\lambda)$ .}
\label{table1}
\begin{tabular}{|l|l|l|l|l|l|l|}
\hline
   \diaghead{DDDDDD}
{$\lambda$}{$F,H_R$}      & $0.1,H_C(0.1)-10^{-5}$ & $0.1,0.002$ & $1,H_C(1)-10^{-5}$ & $1,0.2$ & $1.9,H_{C}(1.9)-10^{-5}$ & $1.9,0.5$ \\ \hline
$0.01$ & 0.06s                              & 0.06s   & 0.02s                              & 0.03s   & 0.02	s                                & 0.02s    \\ \hline
$1$     & 0.19s                              & 0.21s   & 0.04s                              & 0.04s   & 0.04s                                & 0.03s    \\ \hline
$100$  & 4.59s                              & 5.45s   & 0.78s                               & 0.87s   & 0.02s                                & 0.03s    \\ \hline
\end{tabular}
\end{table}
\begin{table}[htbp]
\centering
\caption{Times to compute a single Evans determinant $D_{F,H_R}(\lambda)$ .}
\label{table2}
\begin{tabular}{|l|l|l|l|l|l|l|}
\hline
   \diaghead{DDDDDD}
{$\lambda$}{$F,H_R$}      & $0.1,H_C(0.1)+10^{-2}$ & $0.1,0.9$ & $1,H_C(1)+10^{-2}$ & $1,0.9$ & $1.9,H_{C}(1.9)+10^{-2}$ & $1.9,0.99$ \\ \hline
$0.01$ & 0.11s                              & 0.09s   & 0.06s                              & 0.06s   & 0.06	s                                & 0.05s    \\ \hline
$1$     & 0.25s                              & 0.43s   & 0.07s                              & 0.05s   & 0.12s                                & 0.15s    \\ \hline
$100$  & 3.05s                              & 4.84s   & 0.32s                               & 0.58s   & 0.73s                                & 2.36s    \\ \hline
\end{tabular}
\end{table}

\end{document}